\documentclass{amsart}

\usepackage[latin9]{inputenc}
\usepackage[margin=1in]{geometry}
\usepackage{amsthm}
\usepackage{amsmath}
\usepackage{amssymb}
\usepackage{amsrefs}
\usepackage{graphicx}
\usepackage{esint}
\usepackage{hyperref}

\usepackage{fourier}
\usepackage{xcolor} 
\usepackage{epstopdf}

\def\C{{\mathbb C}}
\def\R{{\mathbb R}}
\def\RR{{\mathbb{R}}}
\def\N{{\mathbb N}}

\def\le{\leqslant}
\def\ge{\geqslant}

\newcommand{\veps}{\varepsilon}
\newcommand{\eps}{\epsilon}
\newcommand{\wt}{\widetilde}
\newcommand{\ud}{\,\mathrm{d}}

\newcommand{\Or}{\mathcal{O}}

\theoremstyle{plain}
\newtheorem{theorem}{Theorem}[section]
\newtheorem{lemma}[theorem]{Lemma}

\newtheorem{proposition}[theorem]{Proposition}
\newtheorem{hyp}{Assumption}

\theoremstyle{definition}

\newtheorem*{remark*}{Remark}

\numberwithin{equation}{section}

\newcommand{\abs}[1]{\lvert#1\rvert}
\newcommand{\Abs}[1]{\left\lvert#1\right\rvert}

\newcommand{\norm}[1]{\lVert#1\rVert}
\newcommand{\Norm}[1]{\left\lVert#1\right\rVert}

\newcommand{\dps}{\displaystyle}
\newcommand{\FGA}{\text{FGA}}

\DeclareMathOperator{\tr}{tr}

\usepackage{todonotes}

\newcommand{\red}[1]{#1}

\title[Frozen Gaussian approximation with surface hopping]{Frozen Gaussian approximation with surface hopping for mixed quantum-classical dynamics: A mathematical justification of fewest switches surface hopping algorithms} 

\author{Jianfeng Lu} 
\address{Department of Mathematics, Department of
  Physics, and Department of Chemistry, Duke University, Box 90320,
  Durham NC 27708, USA} 
\email{jianfeng@math.duke.edu}

\author{Zhennan Zhou}
\address{Department of Mathematics, Duke University, Box 90320, Durham NC 27708, USA}
\email{zhennan@math.duke.edu}

\date{\today}\thanks{This work is partially supported by the National
  Science Foundation under grants DMS-1312659, DMS-1454939 and
  RNMS11-07444 (KI-Net). J.L. is also partially supported by the
  Alfred P.~Sloan Foundation. J.L. would like to thank Sara Bonella,
  Giovanni Ciccotti, Joe Subotnik and Jonathan Weare for helpful
  discussions, and especially John Tully for encouragement.}

\begin{document}

\begin{abstract}
  We develop a surface hopping algorithm based on frozen Gaussian
  approximation for semiclassical matrix Schr\"odinger equations,
  \red{in the spirit of Tully's fewest switches surface hopping
    method}. The algorithm is asymptotically derived from the
  Schr\"odinger equation with rigorous approximation error
  analysis. The resulting algorithm can be viewed as a path integral
  stochastic representation of the semiclassical matrix Schr\"odinger
  equations. Our results provide mathematical understanding to and
  shed new light on the important class of surface hopping methods in
  theoretical and computational chemistry.
\end{abstract}

\maketitle


\section{Introduction}

The surface hopping algorithms, in particular the celebrated Tully's
fewest switches surface hopping (FSSH) algorithm \cites{Tully,Tully2},
are widely used in theoretical and computational chemistry for mixed
quantum-classical dynamics in the non-adiabatic regime.

The Schr\"odinger equation, which is often high dimensional in
chemistry applications, is impractical to solve directly due to curse
of dimensionality. Thus, development of algorithms based on
semiclassical approximation, which only involve solving ODEs, is
necessary. Within the Born-Oppenheimer approximation, the resulting
algorithm from the semiclassical approximation is the familiar ab
initio molecular dynamics \red{and related semiclassical algorithms}. However, in many applications, the
adiabatic assumption in the Born-Oppenheimer approximation is
violated, thus, we need to consider the non-adiabatic dynamics. The
surface hopping algorithms are hence proposed to incorporate in quantum behavior due to the
non-adiabaticity.

Despite the huge popularity of the algorithm and the many attempts in
the chemistry literature for corrections and further improvements (see \textit{e.g.},
\cites{Prezhdo,Schutte,Schwartz,Subotnik1,Subotnik2,Subotnik3,HannaK1}),
which is a very active area to date, the understanding of such
algorithms, in particular, how surface hopping type algorithms can be
derived from the nuclei Schr\"odinger equations, remains rather poor.

\red{In this work, we rigorously derive a surface hopping algorithm,
  named frozen Gaussian approximation with surface hopping (FGA-SH),
  to approximate the Schr\"odinger equations with multiple adiabatic
  states in the semiclassical regime. The FGA-SH algorithm shares
  similar spirit as the FSSH algorithms used in the chemistry
  literature \cite{Tully}, while it also differs in some essential
  ways.} Hence, besides providing a rigorously asymptotically correct
approximation, our derivation hopefully will also help clarify several
issues and mysteries around the FSSH algorithm, and lead to systematic
improvement of this type of algorithms.

The key observation behind our work is a path integral stochastic
representation to the solution to the semiclassical Schr\"odinger
equations. The surface hopping algorithm can be in fact viewed as a
direct Monte Carlo method for evaluating the path integral. Thus, the
path space average provides an approximation to the solution of a high
dimensional PDE, similar to the familiar Feynman-Kac formula for
reaction diffusion type equations. To the best of our knowledge, this
has not been observed in the literature, and it is crucial for
understanding what the surface hopping algorithm really tries to
compute.

In this stochastic representation, the path space consists of
continuous trajectory in the phase space, whose evolution switches
between classical Hamiltonian flows corresponding to each energy
surface, and is hence piecewise deterministic, except at
hoppings. This is why these algorithms are called surface hopping
algorithms. Also to avoid any potential confusion, while we
approximate solutions to the Schr\"odinger equation, the path integral
we consider here (as it only works in the semiclassical limit) is very
different from the usual Feynman path integral for quantum
mechanics. In particular, the stochastic representation is well
defined and gives an accurate approximation to the solution of the
Schr\"odinger equation in the semiclassical regime.

Before we continue, let us review some related mathematical works.
Somehow rather confusingly, sometimes the term ``surface hopping'' is
used for a very different algorithm \cite{Tully71} which is based on
Landau-Zener transition asymptotics \cites{Landau,Zener}. This
algorithm is designed for the situation of a single avoided crossing,
while the type of surface hopping algorithm we consider in this paper,
which is mostly often used in chemistry today, is quite different and
aims to work for general situations. The Landau-Zener asymptotics has
been mathematically studied by Hagedorn and
Joye~\cites{HagedornLZ1,HagedornLZ2}. The algorithm based on
Landau-Zener formula is also studied in the mathematics literature,
see \textit{e.g.}, \cites{Lasser1,Lasser2,SH_SQZ}. While the algorithm
we consider is very different, some of these numerical techniques
might be used in our context as well.

For the surface hopping algorithm we studied in this work, the
understanding in the chemistry literature (see e.g.,
\cites{HannaK1,Schutte,Subotnik3}) often starts from the
quantum-classical Liouville equation \cite{KCmodel}, which is a
natural generalization of the usual Moyal's evolution equation of
Wigner distribution to the matrix Schr\"odinger equations. In the
mathematics literature, the quantum-classical Liouville equation was
studied numerically in \cite{Chai} in low dimensions very
recently. While we are able to derive a surface hopping type
algorithm, our derivation is based on a different tool for
semiclassical analysis, the frozen Gaussian approximation,
\textit{aka} the Herman-Kluk propagator
\cites{HermanKluk,Kay1,Kay2,SwartRousse,FGA_Conv,FGACMS}. It is not
yet clear to us whether the surface hopping algorithms used in the
chemistry literature (and the one we derived) can be rigorously
justified from the view point of quantum-classical Liouville
equation. This remains an interesting research direction.

\medskip 

The surface hopping algorithm we derive is based on asymptotic
analysis on the phase space. The ansatz of the solution, represented
as an integration over the phase space and possible configurations of
hopping times is given in Section~\ref{sec:fga}, after a brief review
of the frozen Gaussian approximation for single surface case. For the
algorithmic purpose, it is more useful to take a stochastic
representation of the ansatz as a path integral, which is given in
Section~\ref{sec:prob}. A simple Monte Carlo algorithm for the path
space average then leads to a rigorously justifiable surface hopping
algorithm, which we will compare with and connect to those used in the
chemistry literature in Section~\ref{sec:literature}. The asymptotic
derivation of the ansatz is given in Section~\ref{sec:asymptotic}. The
main rigorous approximation result is stated in
Section~\ref{sec:convergence}, together with a few illustrating
examples. Some numerical examples of the algorithm are discussed in
Section~\ref{sec:numerics}. We conclude the paper with proofs of the
main result in Section~\ref{sec:proof}.

\section{\red{Integral representation for semiclassical matrix Schr\"odinger equations}}\label{sec:fga}

\subsection{Two-state \red{matrix} Schr\"odinger equation}

Consider the rescaled Schr\"odinger equation for nuclei and electrons
\begin{equation}\label{SE2}
i\veps \frac{\partial}{\partial t} u = - \frac{\veps^2}{2} \Delta_{x} u - \frac{1}{2} \Delta_{r}  u  +  V (x, r)  u.
\end{equation}
where $u(t,x,r)$ is the total wave function, $x\in \R^m$ represents
the nuclear degrees of freedom, $r\in \R^{n}$ denotes the electronic
degrees of freedom,  and $V(x,r)$ is the total interaction potential. Here, $\veps \ll 1$ is the square root of the mass ratio between the electrons and the nuclei (for simplicity, we assume that all nuclei have the same
mass). 

We define the electronic Hamiltonian 
\[
H_e=-\frac{1}{2}\Delta_r+V(x,r),
\]
whose eigenstates $\Psi_k(r;x)$, given by 
\begin{equation}\label{eig:eH}
H_e \Psi_k(r;x)=E_k(x)\Psi_k(r;x),
\end{equation}
are called the adiabatic states. Note that in the eigenvalue problem
\eqref{eig:eH}, $x$ enters as a parameter. \red{In particular, viewed as a function of $x$, the eigenvalues $E_k(x)$ will be referred as energy surfaces.}

In this work, we will only consider a finite number of adiabatic
states, that is, we assume the following expansion of the total wave
function
\[
u(t,x,r)=\sum_{n=0}^{N-1} u_n(t,x) \Psi_n(r;x).
\] 
This is justified if the rest of the spectrum of $H_e$ is far
separated from that of the states under consideration, so that the
transition between these $N$ energy surfaces and others is negligible.
For the separation condition and the corresponding spectral gap
assumption, the readers may refer to \cites{BO1,BO2} for detailed
discussions.

In fact, for simplicity of notation, we will assume that the number of
states is $N = 2$, the extension to any finite $N$ is
straightforward. In this case,
$u(t, x, r) = u_0(t, x) \Psi_0(r; x) + u_1(t, x) \Psi_1(r; x)$, the
original equation is equivalent to a system of PDEs of
$U = 
\bigl(\begin{smallmatrix} u_0 \\
  u_1
\end{smallmatrix} \bigr)$,
which we will henceforth refer to as the \emph{matrix Schr\"odinger
  equation}:
\begin{equation}\label{vSE}
  i \veps \partial_t  
  \begin{pmatrix} u_0 \\
    u_1 \end{pmatrix}  = -\frac{\veps^2}{2} \Delta_x \begin{pmatrix} u_0 \\
    u_1 \end{pmatrix} + 
  \begin{pmatrix} 
    E_0 \\
    & E_1
  \end{pmatrix} \begin{pmatrix} u_0 \\
    u_1 \end{pmatrix} -\red{\frac{\veps^2}{2}} \begin{pmatrix}
    D_{00} & {D_{01}} \\
{  D_{10}} & D_{11}
  \end{pmatrix} \begin{pmatrix} u_0 \\
    u_1 \end{pmatrix} -\red{\veps^2}  \sum_{j=1}^m 
  \begin{pmatrix} 
    d_{00} &{ d_{01}} \\
 {  d_{10}} & d_{11}
  \end{pmatrix}_j
  \partial_{x_j} \begin{pmatrix} u_0 \\
    u_1 \end{pmatrix},
\end{equation}
where
\red
{
\[
D_{kl}(x) =\langle \Psi_k(r;x), \Delta_x \Psi_l(r;x) \rangle_r, \quad
\left(d_{kl}(x)\right)_j =\langle \Psi_k(r;x), \partial_{x_j}
\Psi_l(r;x) \rangle_r, \quad \text{for } k, l = 0, 1,\; j = 1, \ldots,
m.
\]
}

\subsection{Brief review of the frozen Gaussian approximation}\label{sec:revfga}

Before we consider the matrix Schr\"odinger equation \eqref{vSE}, let us recall the  ansatz of frozen Gaussian approximation (\textit{aka} Herman-Kluk propagator) \cites{HermanKluk,Kay1,Kay2,SwartRousse} for scalar Schr\"odinger equation
\begin{equation}\label{eq:singleSE}
  i\veps \frac{\partial}{\partial t} u(t, x) = - \frac{\veps^2}{2} \Delta u(t, x) + E(x) u(t, x). 
\end{equation}
\red{Note that if we drop the terms
depending on $d$ and $D$ in \eqref{vSE}, it decouples to two equations of the form of \eqref{eq:singleSE}. }
The algorithm that we will derive for \eqref{vSE} can be viewed as an extension of the FGA to the matrix Schr\"odinger equation. 

The frozen Gaussian approximation is a convergent approximation to the
solution of \eqref{eq:singleSE} with $\Or(\veps)$ error
\cite{SwartRousse}.  It is based on the following \red{integral representation of an approximate solution to \eqref{eq:singleSE}}
\begin{equation}
  u_{\FGA}(t,x) =\frac{1}{(2\pi \veps)^{3m/2}} \int_{\R^{3m}} a(t,q,p)e^{\frac{i}{\veps}\Phi(t,x,y,q,p) } u(0, y) \ud y \ud q \ud p,
\end{equation}
where $u(0, \cdot)$ is the initial condition. Here, the phase function $\Phi$
is given by
\[
\Phi(t,x,y,q,p)=S(t,q,p)+P(t,q,p)\cdot (x-Q(t,q,p))-p\cdot(y-q) + \frac{i}{2}|x-Q(t,q,p)|^2 + \frac{i}{2}|y-q|^2.
\]
Given $q$ and $p$ as parameters, the evolution of $Q$ and $P$ are governed by the Hamiltonian flow according to classical Hamiltonian $h(q,p) = \frac{1}{2}\abs{p}^2 + E(q)$,
\begin{align*}
\frac{\ud}{\ud t}Q & =\partial_p h(Q,P),\\
\frac{\ud}{\ud t}P & =-\partial_q h(Q,P),
\end{align*}
with initial conditions $Q(0,p,q)=q$ and $P(0,q,p)=p$. The solution to
the Hamiltonian equations defines a trajectory on the phase space
$\R^{2m}$, which we call \emph{FGA trajectory}.  $S$ is the
action corresponding to the Hamiltonian flow, with initial condition
$S(0,q,p)=0$.  The equation of $a$ is obtained by matched asymptotics \red{and is given by:}
\begin{equation}
  \frac{\ud}{\ud t} a  = \frac 1 2  a \tr\bigl( Z^{-1}\bigl(\partial_z P - i \partial_z Q  \nabla^2_Q E(Q) \bigr) \bigr)
\end{equation}
with initial condition $a(0,q,p)= 2^{m/2}$, where we have used the
short hand notations
\[
\partial_z = \partial_q - i \partial_p,\quad\text{and} \quad  Z=\partial_z (Q+iP).
\]
Equivalently, we can rewrite as
\[
u_{\FGA}(t,x) =\frac{1}{(2\pi \veps)^{3m/2}} \int_{\R^{2m}} A(t,q,p) e^{\frac{i}{\veps}\Theta(t,x,q,p) } \ud q \ud p,
\]
where 
\begin{align*}
\Theta(t,x,q,p)& =S(t,q,p)+P(t,q,p)\cdot (x-Q(t,q,p))+ \frac{i}{2}|x-Q(t,q,p)|^2, \\
A(t,q,p) & = a(t,q,p)\int_{\R^m} u_0 (y) e^{\frac{i}{\veps} (-p\cdot(y-q)+ \frac{i}{2}|y-q|^2)} \ud y.
\end{align*}
As $A$ only differs from $a$ by a constant multiplication factor, it satisfies the same equation as $a$ does (with different initial condition). 

The following lemma, which we directly quote from \cite{FGACMS}*{Lemma 3.1}, states that the FGA ansatz reproduces the initial condition.
\begin{lemma}\label{lem:initial}
For $u\in L^2(\R^m)$, we have 
\[
u(x) = \frac{1}{(2\pi \veps)^{3m/2}} \int_{\R^{3m}} 2^{\frac{m}{2}} e^{\frac{i}{\veps}\Phi(0,x,y,q,p) } u(y) \ud y \ud q \ud p.
\]
\end{lemma}

The next lemma is crucial for the asymptotic matching to derive the
evolution equations. The proof can be found in \cite{FGACMS}*{Lemma
  3.2} and \cite{FGA_Conv}*{Lemma 5.2}. \red{We recall this
  lemma here as it will be used in the extension of frozen Gaussian
  approximation to the matrix Schr\"odinger equations.}
\begin{lemma} \label{lem:asym}
For any vector $b(y,q,p)$ and any matrix $M(y,q,p)$ in Schwartz class viewed as functions of $(y,q,p)$, we have
\[
b \cdot (x-Q) \sim - \veps \partial_{z_k} (b_j Z_{jk}^{-1}),
\]
and
\[
(x-Q) \cdot M (x-Q) \sim \veps \partial_{z_l} Q_j M_{jk}Z_{kl}^{-1} +\Or(\veps^2)
\]
where Einstein's summation convention has been assumed.  Moreover, for
multi-index $\alpha$ that $|\alpha| \ge 3$,
\[
(x-Q)^{\alpha} \sim \Or(\veps^{|\alpha|-1}).
\]
Here, we denote by $f \sim g$ that
\[
\int_{\R^{3m}} f e^{\frac i \veps \Phi} \ud y \ud q \ud p =\int_{\R^{3m}} g  e^{\frac i \veps \Phi} \ud y \ud q \ud p.
 \] 
\end{lemma}

\subsection{\red{The integral representation for surface hopping}}\label{sec:fgash}

We now consider extending the \red{integral representation in
  the previous subsection} to the matrix Schr\"odinger equation by
incorporating the coupling of the two energy surfaces, \red{
  which is the basis of the FGA-SH algorithm.}

Let us assume, for simplicity of notation, that the initial condition
concentrates on energy surface $E_0$ (\textit{i.e.}, $u_1(0, x) = 0$
\red{and $u_0$ is non-zero}).  The extension to general
initial condition is straightforward as the equation is linear.
\red{We construct an approximation to the total wave function
  following the ansatz below.  We will prove rigorously that it gives
  an $\Or(\veps)$ approximation to the true solution; see the
  convergence statement in Section~\ref{sec:convergence}.  The
  integral representation here is fully deterministic, and our FGA-SH
  algorithm can be understood as a Monte Carlo algorithm for
  evaluation.}
\begin{equation}\label{ansatz}
  u_{\FGA}(t,x,r)= K^{(0)}_{00}(t,x,r)+ K^{(1)}_{01}(t,x,r)+K^{(2)}_{00}(t,x,r)+K^{(3)}_{01}(t,x,r)+\cdots
\end{equation}
where,
\begin{equation}\label{eq:defK}
K^{(l)}_{mn}(t,x,r)=\Psi_n(r;x) u_m^{(l)}(t,x)
\end{equation}
\red{represents the contribution to the ansatz by wave packets initiated at surface  $m$, ends at surface
  $n$, and switches the propagating surface $l$ times in between --- the meaning of which will become clear below.}
Thus, we can rewrite \eqref{ansatz} as
\[
u_{\FGA}(t,x,r)=\Psi_0(r; x) \left(u_0^{(0)}(t, x)+u_0^{(2)}(t, x)+\cdots\right) + \Psi_1(r; x) \left(u_0^{(1)}(t, x)+u_0^{(3)}(t, x)+\cdots\right).
\]
We will refer this as the surface hopping ansatz. \red{The
  idea of splitting the wave function in this way is similar to that
  used in the work by Wu and Herman \cites{WuHerman1,WuHerman2,
    WuHerman3}, which is also based on the frozen Gaussian
  approximation. The two approaches are different however in several
  essential ways, as we will explain in \S\ref{sec:literature}.}

\red{As we consider initial condition starting from the surface $0$,
  for simplicity of notation, we will drop the subscripts $0$ in
  $u_0^{(n)}$ for now. In the ansatz, $u^{(0)}$ consists of
  contribution from wave packets propagating only on energy surface
  $E_0$, without switching to $E_1$ surface.} It is given by the
ansatz of frozen Gaussian approximation \red{on a single surface} as
in Section \ref{sec:revfga}:
\begin{equation}\label{eq:u00}
  u^{(0)}(t, x) =\frac{1}{(2\pi\veps)^{3m/2}} \int   A^{(0)}(t,z_0) \exp \left( \frac{i}{\veps} \Theta^{(0)}(t,z_0,x) \right) \ud z_0,
\end{equation}
where we have used $z_0=(q_0,p_0)$ to denote phase space variables, 
\[
\Theta^{(0)}(t,q,p,x)=S^{(0)}(t,q,p)+P^{(0)}(t,q,p)\cdot \left( x-Q^{(0)} (t,q,p)\right) + \frac{i}{2} \left|x-Q^{(0)}(t,q,p)\right|^2,
\]
and
\begin{equation}\label{eq:c00}
A^{(0)}(t,q,p) = a^{(0)}(t,q,p)\int_{\R^m} u_0 (y) e^{\frac{i}{\veps} (-p\cdot(y-q)+ \frac{i}{2}|y-q|^2)} \ud y.
\end{equation}
Here, the evolution of the quantities $S^{(0)}$, $P^{(0)}$,
$Q^{(0)}$ and $A^{(0)}$ are determined by matched asymptotic and
will be specified below.  We will refer these quantities as \emph{FGA
  variables} in the sequel.

\red{For $n > 0$, the wave function $u^{(n)}$ counts for contribution of wave packets that switch between the two energy surfaces $n$ times. Given $n$, to specify the integral representation, let us denote 
$T_{n:1}=(t_n,\cdots,t_1)$ a sequence of times ordered backwardly, \textit{i.e.,} they satisfy
\[
0 \le t_1 \le t_2 \le \cdots \le t_n \le t.
\]}
 The ansatz for $u^{(n)}$
is given by
\begin{multline}\label{eq:u0n}
  u^{(n)}(t, x)  = \frac{1}{(2\pi\veps)^{3m/2}} \int \ud z_0 \int_0^t \ud t_n \int_0^{t_n} \ud t_{n-1} \cdots \int_0^{t_2} \ud t_1 \; \tau^{(1)}(T_{1:1}, z_0)\cdots \tau^{(n)} (T_{n:1},z_0) \times \\ 
\times A^{(n)}(t, T_{n:1}, z_0) \exp\left( \frac{i}{\veps}
    \Theta^{(n)}(t,T_{n:1}, z_0, x) \right) 
\end{multline}
where
\red{
\[
\Theta^{(n)}(t,T_{n:1}, z_0,x)=S^{(n)}(t,T_{n:1}, z_0)+P^{(n)}(t,T_{n:1}, z_0)\cdot \left( x-Q^{(n)} (t,T_{n:1}, z_0)\right) + \frac{i}{2} \left|x-Q^{(n)}(t,T_{n:1}, z_0)\right|^2,
\]
and}
\[
A^{(n)}(t, T_{n:1}, z_0)  =a^{(n)}(t, T_{n:1}, z_0)\int_{\R^m} u_0 (y) e^{\frac{i}{\veps} (-p\cdot(y-q)+ \frac{i}{2}|y-q|^2)} \ud y. 
\]
To simplify the notation, we will often write \eqref{eq:u0n} as
\begin{equation}
  u^{(n)}(t, x) =  \frac{1}{(2\pi\veps)^{3m/2}} \int \ud z_0  \int_{0 \le t_1 \le\cdots \le t_n \le t} \tau^{(1)}\cdots\tau^{(n)} \; A^{(n)} \; \exp\left( \frac{i}{\veps}
    \Theta^{(n)}\right) \ud T_{n:1},
\end{equation}
where $\ud T_{n:1}=\ud t_1 \cdots \ud t_n$. \red{Note that in
\eqref{eq:u0n}, we integrate over all possible sequences of $n$ ordered times
in the time interval $[0, t]$.}

\red{Note that given the time sequence $T_{n:1}$,
  \eqref{eq:u0n} depends on the FGA variables
  $S^{(n)}, P^{(n)}, Q^{(n)}, A^{(n)}$, and also $\tau^{(k)}$ for
  $k = 1, \ldots, n$. We will refer $\tau^{(k)}$ as the \emph{hopping
    coefficients}, since they are related to the jumping intensity of
  our stochastic algorithm. Note that as other FGA variables,
  $\tau^{(k)}(T_{k:1},z_0)$ depend on the time sequence $T_{k:1}$ and
  $z_0$.}

Let us now specify the evolution equations for the FGA variables and hopping coefficients involved in \eqref{eq:u00} and \eqref{eq:u0n} \red{to complete the integral representation}. The asymptotic derivation of these equations will be given in Section~\ref{sec:asymptotic}.

\red{Recall that for $n = 0$, the FGA trajectory evolves on a
  single energy surface $E_0$. For $n > 0$, the trajectory will switch
  between the two surfaces at given time sequences $T_{n:1}$. More
  precisely, $T_{n:1} = (t_n, t_{n-1}, \ldots, t_1)$ determines a
  partition of the time interval $[0,t]$.} Each FGA variable evolves
piecewisely in time on alternating energy surfaces, starting on energy
surface $0$ (due to our assumption of the initial condition). For
convenience, we take the convention $t_0 = 0$ and $t_{n+1} = t$ in the
following.

\red{When $t\in [t_k,t_{k+1})$ for $k$ being an even integer,
  all the FGA variables evolve on energy surface $l^{(k)}=0$, and for
  $k$ odd, the trajectory evolves on energy surface $l^{(k)}=1$.  The
  evolution equations are given accordingly as
\begin{subequations}\label{eq:evenevolve}
\begin{align}
\frac{\ud}{\ud t} Q^{(k)} & = P^{(k)},  \\
\frac{\ud}{\ud t} P^{(k) }& =- \nabla E_{l^{(k)}} (Q^{(k)}),\\
\frac{\ud}{\ud t} S^{(k)} &=\frac{1}{2} ( P^{(k)} )^2 - E_{l^{(k)}} (Q^{(k)}),\\ 
\frac{\ud}{\ud t} A^{(k)} & =\frac 1 2  A^{(k)} \tr\left( (Z^{(k)})^{-1}\left(\partial_z P^{(k)} - i \partial_z Q^{(k)}  \nabla^2_Q E_{l^{(k)}}(Q^{(k)}) \right) \right) - A^{(k)} d_{l^{(k)}l^{(k)}}\cdot P^{(k)}.
\end{align}
\end{subequations}
We observe that the evolution equations \eqref{eq:evenevolve}  are similar to those in the single surface case. This connection will become more clear in our asymptotic derivation later in Section~\ref{sec:asymptotic}.
}

The crucial difference with the single surface case is that the
trajectory now switches between the two energy surfaces. At time $t = t_k$ for $1 \le k \le n$, the trajectory switches
from one energy surface to the other. The FGA variables are continuous in time
\begin{subequations}\label{eq:contcond}
\begin{align} 
  A^{(k)}(t_{k}, T_{k:1},z_0) = A^{(k-1)}(t_{k},T_{k-1:1},z_0), \\
  S^{(k)}(t_{k}, T_{k:1},z_0) = S^{(k-1)}(t_{k},T_{k-1:1},z_0), \\
  P^{(k)}(t_{k}, T_{k:1},z_0) = P^{(k-1)}(t_{k},T_{k-1:1},z_0), \\
  Q^{(k)}(t_{k}, T_{k:1},z_0) = Q^{(k-1)}(t_{k},T_{k-1:1},z_0), 
\end{align}
\end{subequations}
such that the left hand sides serve as the initial conditions for the
evolution equations during the next time interval $[t_k, t_{k+1})$.
The FGA trajectory \red{for two energy surfaces} is thus defined on the
extended phase space $\R^{2m} \times \{0, 1\}$, 
the piecewise Hamiltonian dynamics on each energy surface, and the
\red{continuity} condition \eqref{eq:contcond}.  
Finally, the hopping coefficient $\tau^{(k)}$ is given by
\begin{equation}\label{eq:taudef}
\tau^{(k)}(T_{k:1}, z_0)=
\begin{cases} 
  -P^{(k)}(T_{k:1}, z_0) \cdot \red{d_{01}}\bigl(Q^{(k)}(\red{t_k,}T_{k:1}, z_0)\bigr), & k \text{ even}; \\
  -P^{(k)}(T_{k:1}, z_0) \cdot \red{d_{10}}\bigl(Q^{(k)}(\red{t_k,}T_{k:1}, z_0)\bigr), & k \text{ odd}.  \\
\end{cases}
\end{equation}
It is worth remarking that, $\tau^{(k)}(T_{k:1}, z_0)$ is complex valued in general, and therefore, we will later choose its modulus as the jumpinp intensity in the probabilistic interpretation of the ansatz in Section~\ref{probab}.

\section{\red{Frozen Gaussian approximation with surface hopping as a stochastic interpretation}}\label{sec:prob}

We have seen in Section~\ref{sec:fgash}, the \red{surface hopping
  ansatz} is a sum of contributions involving integration on the phase
space and of all possible sequence of ordered times.  Since the phase
space could be of high dimension in chemical applications and number
of time sequence grows factorially fast \red{with respect to $n$}, the
direct discretization of the integral does not give a \red{practical} algorithm. Observe that essentially we have a high dimension
integral to deal with, and hence it is natural to look for stochastic
methods (in analogy to Monte Carlo method for quadrature). Motivated
by this, in this section, we will present a stochastic representation
of the \red{surface hopping ansatz}, which can be used to numerically approximate the
\red{solution to the Schr\"odinger equations}. The resulting algorithm bares similarity to the surface
hopping algorithm developed in the chemistry literature, which will be
elaborated in Section \ref{sec:literature}.

\subsection{Probabilistic interpretation of FGA for single surface} 

Before we consider the frozen Gaussian approximation with surface
hopping, let us start with the usual FGA on a single surface. Recall that the ansatz is given in this case by 
\begin{equation}\label{FGA_single}
  \begin{aligned}
    u_{\FGA} (t,x) & =\frac{1}{(2\pi\veps)^{3m/2}} \int_{\R^{2m}}  \ud z_0 \; A(t, z_0) \exp \Bigl( \frac{i}{\veps} \Theta(t,z_0,x) \Bigr) \\
    & =\frac{1}{(2\pi\veps)^{3m/2}} \int_{\R^{2m}} \ud z_0 \Abs{A(0,
      z_0)} \; \frac{A(t, z_0)}{\Abs{A(0, z_0)}} \exp \Bigl(
    \frac{i}{\veps} \Theta(t,z_0,x) \Bigr), 
\end{aligned}
\end{equation}
\red{and from  Lemma~\ref{lem:initial} that }
\begin{equation}
  u_0(x) = \frac{1}{(2\pi\veps)^{3m/2}} \int_{\R^{2m}} \ud z_0 \; A(0, z_0) 
  \exp \Bigl( \frac{i}{\veps} \Theta(0,z_0,x) \Bigr). 
\end{equation}

Assuming that $A(0,z_0)$ is an integrable function in $\R^{2m}$,
\textit{i.e.},
\begin{equation*}
\int_{\R^{2m}} \ud z_0 |A(0,z_0)| < \infty, 
\end{equation*}
we can define a probability measure $\mathbb{P}_0$ on $\R^{2m}$ such
that
\begin{equation}\label{eq:P0single}
  \mathbb{P}_0(\Omega) = \mathcal{Z}^{-1}  \frac{1}{(2\pi\veps)^{3m/2}} \int_{\Omega} \ud z_0 \Abs{A(0,z_0)}
\end{equation}
for any $\Omega \subset \R^{2m}$, where $\mathcal{Z} = \frac{1}{(2\pi\veps)^{3m/2}} \int_{\R^{2m}} \ud z_0 \Abs{A(0,z_0)}$ is a normalization factor so that $\mathbb{P}_0$ is a probability measure. Note that in general $A(0, z_0)$ is complex valued, and hence the necessity in taking the modulus in the definition \eqref{eq:P0single}. We can thus rewrite 
\begin{equation}\label{eq:FGAprobsingle}
  \begin{aligned}
    u_{\FGA}(t,x) & =  \mathcal{Z} \int \mathbb{P}_0(\ud z_0) \; \frac{A(t, z_0)}{\Abs{A(0,z_0)}} \exp \Bigl( \frac{i}{\veps} \Theta(t,z_0,x) \Bigr) \\
    & = \mathcal{Z} \mathbb{E}_{z_0} \Biggl[ \frac{A(t, z_0)}{\Abs{A(0,z_0)}} \exp \Bigl( \frac{i}{\veps} \Theta(t,z_0,x) \Bigr) \Biggr], 
  \end{aligned}
\end{equation}
where the expectation is taken with respect to $\mathbb{P}_0$. Thus, we may use a Monte Carlo sampling for $u_{\FGA}(t, x)$ as 
\begin{equation}
  u_{\FGA}(t, x) \approx \frac{\mathcal{Z}}{M} \sum_{i=1}^M \frac{A(t, z_0^{(i)})}{\bigl\lvert A(0,z_0^{(i)}) \bigr\rvert} \exp \Bigl( \frac{i}{\veps} \Theta(t,z_0^{(i)},x) \Bigr), 
\end{equation}
where $\{z_0^{(i)}\}_{i = 1, \ldots, M} \subset \R^{2m}$ are
independent identically distributed samples from the probability
measure $\mathbb{P}_0$. Algorithmically, once $z_0^{(i)}$ is sampled,
we evolve the FGA variables $Q, P, A, S$ up to time $t$, which gives
the value of the integrand. Denote $z_t = (Q_t, P_t)$ for the FGA
trajectory, so that $z_t$ satisfies the Hamiltonian flow with
Hamiltonian $h(q, p)$:
\[
\ud z_t = (h_p,-h_q) \ud t.
\]
The trajectory $z_t$ corresponds to a one-to-one map on the phase
space: $z_0 \mapsto z_t$. As the trajectory is deterministic once the
initial point $z_0$ is prescribed, we can equivalently view the
expectation over initial condition in \eqref{eq:FGAprobsingle} as
expectation over \red{ensemble of} trajectories $z_t$; this point of view is
useful for the extension to cases with surface hopping.

\red{In summary,} in the single surface case, the FGA ansatz can be evaluated by a
stochastic approximation where the randomness comes from sampling of
initial points of the FGA trajectory. 

\subsection{Probabilistic interpretation for FGA with surface hopping} \label{probab} 

We now extend the probabilistic interpretation to the cases with
surface hopping. Since the FGA trajectory in this case depends on the
energy surface on which it evolves, to prescribe a trajectory, we need
to also keep track of the energy surface. Thus, the phase space
extends to $\wt{z}_t = (z_t, l_t) \in \R^{2m} \times \{0, 1\}$, where $l_t$
indicates the energy surface that the trajectory is on at time $t$. 

To take into account the possible hopping times, we will construct a
stochastic process for $\wt{z}_t$, in consistency with the ansatz we
have. The evolution of $z_t$ is deterministic on the energy surface that
$l_t$ indicates, given by the corresponding Hamiltonian flow: 
\begin{equation}\label{eq:drift}
\ud z_t= \bigl( p_t, -\nabla_{q} E_{l_t}(p_t, q_t) \bigr)  \ud t. 
\end{equation}
This is coupled with a Markov jump process of $l_t$ which is c\`adl\`ag and hops between $0$ and $1$, with infinitesimal transition rate  
\begin{equation}
\mathbb{P}\bigl(l_{t+ \delta t}=m \mid l_t=n, \,z_t = z \bigr) = \delta_{nm} + \lambda_{nm}(z) \delta t + o(\delta t) 
\end{equation}
for $m, n \in \{0, 1\}$, where the rate matrix is given by 
\begin{equation}
  \lambda(z) = 
  \begin{pmatrix}
    \lambda_{00}(z) & \lambda_{01}(z) \\
    \lambda_{10}(z) & \lambda_{11}(z) 
  \end{pmatrix} =
  \begin{pmatrix}
    - \Abs{p \cdot \red{d_{10}}(q)} & \Abs{p \cdot \red{d_{10}}(q)} \\
    \Abs{p \cdot \red{d_{01}}(q)} & - \Abs{p \cdot \red{ d_{01}}(q)}
  \end{pmatrix}.
\end{equation}
\red{Note that $\lambda_{01}(z)$ corresponds to the infinitesimal rate from surface $0$ to $1$, and thus it is given by $\Abs{ p \cdot d_{10}(q)}$.}
We remark $p \cdot d_{10}(q)$ is in general complex, and hence we take
its modulus in the rate matrix; also note that the rate is state
dependent (on $z$). The $\wt{z}_t$ is thus a Markov switching
process. Equivalently, denote the probability distribution on the
extended phase space at time $t$ by $F_t(z, l)$, the corresponding
forward Kolmogorov equation is given by
\begin{equation}
  \frac{\partial}{\partial t} 
  F_t(z, l) + 
  \bigl\{ h_l, F_t(z, l) \bigr\} = \sum_{m = 0}^1 \lambda_{ml}(z) F_t(z, m),  
\end{equation}
where $\{\cdot, \cdot\}$ stands for the Poisson bracket corresponding to the Hamiltonian dynamics \eqref{eq:drift}, 
\begin{equation*}
  \bigl\{h, F\bigr\} =  \partial_p h \cdot \partial_q F - \partial_q h \cdot \partial_p F. 
\end{equation*}

Given a time interval $[0, t]$, thanks to \eqref{eq:drift}, the $z_s$
part of the trajectory $\wt{z}_s = (z_s, l_s)$ is continuous and
piecewise differentiable, while $l_s$ is piecewise constant with
almost surely finite many jumps. Given a realization of the trajectory $\wt{z}_s = (z_s, l_s)$ starting from $\wt{z}_0 = (z_0, 0)$,\footnote{Generalization to initial condition starting from both energy surface is straightforward.} we denote by $n$ the number of jumps $l_s$ has (thus $n$ is a random variable) and also the discontinuity set of $l_s$ as $\bigl\{t_1, \cdots, t_n\bigr\}$, which is an increasingly ordered random sequence. By the properties of the associated
counting process, the probability that there is no jump $(n = 0)$ is
given by
\begin{equation}
  \mathbb{P}(n = 0) = \exp\Biggl(-\int_0^t \lambda_{01}(z_s) \ud s\Biggr)
  = \exp\Biggl(-\int_0^t \Abs{\tau^{(1)}(s, z_0)} \ud s \Biggr), 
\end{equation}
where $\tau^{(1)}$ is defined in \eqref{eq:taudef} the hopping coefficient in the ansatz of FGA with surface hopping. Similarly, the probability with one jump $(n=1)$ is given by 
\begin{equation}
  \mathbb{P}(n = 1) = \int_0^t \ud t_1 \; \Abs{\tau^{(1)}(t_1, z_0)} \exp\Biggl( - \int_0^{t_1} \Abs{\tau^{(1)}(s, z_0)} \ud s \Biggr) \exp\Biggl( - \int_{t_1}^{t} \Abs{\tau^{(2)}(s, T_{1:1}, z_0)} \ud s \Biggr). 
\end{equation}
In addition, conditioning on $n = 1$, the hopping time is distributed with probability density 
\begin{equation}
  \varrho_1(t_1) \propto \Abs{\tau^{(1)}(t_1, z_0)} \exp\Biggl( - \int_0^{t_1} \Abs{\tau^{(1)}(s, z_0)} \ud s \Biggr) \exp\Biggl( - \int_{t_1}^{t} \Abs{\tau^{(2)}(s, T_{1:1}, z_0)} \ud s \Biggr). 
\end{equation}
More generally, we have 
\begin{multline}\label{eq:Pnk}
  \mathbb{P}(n = k) = \int_{0<t_1<\cdots<t_k<t} \ud T_{k:1} \; \prod_{j=1}^k \Abs{\tau^{(j)}(T_{j:1}, z_0)} \\
  \times \exp\Biggl(-\int_{t_k}^t \Abs{\tau^{(k+1)}(s, T_{k:1}, z_0)} \ud s  \Biggr) \prod_{j=1}^{k} \exp\Biggl( - \int_{t_{j-1}}^{t_j} \Abs{\tau^{(j)}(s, T_{j-1:1}, z_0)} \ud s \Biggr), 
\end{multline}
and the probability density of $(t_1, \cdots, t_k)$ given there are $k$ jumps in total is 
\begin{equation}\label{eq:varrhok}
  \varrho_k(t_1, \cdots, t_k) \propto
  \begin{cases}
    \begin{aligned}
      & \prod_{j=1}^k \Abs{\tau^{(j)}(T_{j:1}, z_0)} \exp\Biggl(-\int_{t_k}^t \Abs{\tau^{(k+1)}(s, T_{k:1}, z_0)} \ud s  \Biggr) \\
    & \qquad \qquad \times \prod_{j=1}^{k} \exp\Biggl( - \int_{t_{j-1}}^{t_j} \Abs{\tau^{(j)}(s, T_{j-1:1}, z_0)} \ud s \Biggr),
    \end{aligned} & \text{if } t_1 \le t_2 \le \cdots \le t_k; \\
    0, & \text{otherwise}.
  \end{cases}
\end{equation}
We remark that the complicated expressions are due to the fact that the intensity function $\lambda(z)$ of the jumping process depends on the current state variable $z$, and thus depends on the previous hopping times. These formula reduce to the usual familiar expressions for homogeneous Poisson process if the intensity is uniform. 

Let us now consider a \red{path integral that takes average over the ensemble of trajectories}
\begin{multline}\label{eq:trajavg}
  \wt{u}(t, x, r) = \mathcal{Z} \mathbb{E}_{\wt{z}_t} \Biggl[\Psi_{n\bmod 2}(r; x)
   \biggl( \prod_{k=1}^n \frac{\tau^{(k)}(T_{k:1}, z_0)}{\Abs{\tau^{(k)}(T_{k:1}, z_0)}} \biggr) \frac{A^{(n)}(t, T_{n:1}, z_0)}{\Abs{A^{(0)}(0, z_0)}} \exp\Bigl(\frac{i}{\veps} \Theta^{(n)}(t, T_{n:1}, z_0, x)\Bigr) \\
  \times \exp\Biggl( \int_{t_n}^t \Abs{\tau^{(n+1)}(s, T_{n:1}, z_0)} \ud s  \Biggr) \prod_{k=1}^{n} \exp\Biggl( \int_{t_{k-1}}^{t_k} \Abs{\tau^{(k)}(s, T_{k-1:1}, z_0)} \ud s \Biggr) \Biggr], 
\end{multline}
where the initial condition $z_0$ is sampled from $\mathbb{P}_0$ with probability density on $\R^{2m}$ proportional to $\abs{A^{(0)}(0, z_0)}$ and
$\mathcal{Z}$ is a normalization factor (assuming integrability of $A^{(0)}(0, z_0)$ as before)
\begin{equation}
  \mathcal{Z} = \frac{1}{(2\pi\veps)^{3m/2}} \int_{\R^{2m}} \Abs{A^{(0)}(0, z_0)} \ud z_0. 
\end{equation}
Here, the terms on the second line of \eqref{eq:trajavg}, namely
\[
\exp\Biggl( \int_{t_n}^t \Abs{\tau^{(n+1)}(s, T_{n:1}, z_0)} \ud s  \Biggr) \prod_{k=1}^{n} \exp\Biggl( \int_{t_{k-1}}^{t_k} \Abs{\tau^{(k)}(s, T_{k-1:1}, z_0)} \ud s \Biggr) 
\]
are the weighting terms due to the non-homogeneous state dependent
Poisson process.  Note that the whole term inside the square bracket
in \eqref{eq:trajavg} is determined by the trajectory $\wt{z}_t$, and
thus can be viewed as a functional (with fixed $t$ and $x$) evaluated
on the trajectory.  We now show that \eqref{eq:trajavg} is in fact a
stochastic representation of the \red{FGA surface hopping ansatz given in \S\ref{sec:fgash}}, and hence we obtain an
asymptotically \red{convergent} path integral representation of the
semiclassical matrix Schr\"odinger equation. By the choice of the
initial condition, we have
\begin{multline*}
  \wt{u}(t, x, r) = \frac{1}{(2\pi\veps)^{3m/2}} \int_{\R^{2m}} \ud z_0 \; \mathbb{E}_{\wt{z}_t} \Psi_{n\bmod 2}(r; x)
  \Biggl[ \biggl( \prod_{k=1}^n \frac{\tau^{(k)}(T_{k:1}, z_0)}{\Abs{\tau^{(k)}(T_{k:1}, z_0)}} \biggr) A^{(n)}(t, T_{n:1}, z_0) \exp\Bigl(\frac{i}{\veps} \Theta^{(n)}(t, T_{n:1}, z_0, x)\Bigr) \\
  \times \exp\Biggl( \int_{t_n}^t \Abs{\tau^{(n+1)}(s, T_{n:1}, z_0)} \ud s  \Biggr) \prod_{k=1}^{n} \exp\Biggl( \int_{t_{k-1}}^{t_k} \Abs{\tau^{(k)}(s, T_{k-1:1}, z_0)} \ud s \Biggr) \;\Bigg\vert\; \wt{z}_t = (z_0, 0) \Biggr]. 
\end{multline*}
Since the randomness of the trajectory given initial condition only lies in the hopping times, we further calculate
\begin{equation*}
  \begin{aligned}
    \wt{u}(t, x, r) & = \frac{1}{(2\pi\veps)^{3m/2}} \int_{\R^{2m}} \ud z_0
    \; \sum_{n=0}^{\infty} \mathbb{P}(n) \Psi_{n\bmod 2}(r; x)
    \int_{([0, t])^{n}} \varrho_n(\ud t_1 \cdots \ud t_n)
    \Biggl[ \biggl( \prod_{k=1}^n \frac{\tau^{(k)}(T_{k:1}, z_0)}{\Abs{\tau^{(k)}(T_{k:1}, z_0)}} \biggr) A^{(n)}(t, T_{n:1}, z_0)\\
    & \qquad \times \exp\Bigl(\frac{i}{\veps} \Theta^{(n)}(t, T_{n:1},
    z_0, x)\Bigr) \exp\Biggl( \int_{t_n}^t \Abs{\tau^{(n+1)}(s,
      T_{n:1}, z_0)} \ud s \Biggr) \prod_{k=1}^{n} \exp\Biggl(
    \int_{t_{k-1}}^{t_k} \Abs{\tau^{(k)}(s, T_{k-1:1}, z_0)} \ud s
    \Biggr)  \Biggr] \\
    & = \frac{1}{(2\pi\veps)^{3m/2}} \int_{\R^{2m}} \ud z_0 \;
    \sum_{n=0}^{\infty} \Psi_{n\bmod 2}(r; x)\int_{0<t_1<\cdots<t_n<t} d T_{n:1}
    \; \prod_{k=1}^n \tau^{(k)}(T_{k:1}, z_0) A^{(n)}(t, T_{n:1}, z_0)\\
    & \hspace{27em} \times \exp\Bigl(\frac{i}{\veps} \Theta^{(n)}(t, T_{n:1}, z_0, x)\Bigr) \\
    & = u_{\FGA}(t, x, r),
  \end{aligned}
\end{equation*}
where the second equality follows from \eqref{eq:Pnk} and \eqref{eq:varrhok}.
The above calculation assumes the summability of the terms in the FGA ansatz, which will be rigorously proved in Section~\ref{sec:absconv}.

\subsection{Connection to surface hopping algorithms}\label{sec:literature}


\red{As we have shown in Section~\ref{probab}, the
  surface hopping ansatzis equivalent to a path integral
  representation given in \eqref{eq:trajavg} based on averaging over
  an ensemble of trajectories, the FGA-SH algorithm is a natural Monte
  Carlo sampling scheme. The FGA-SH algorithm consists of steps of
  sampling the initial points of the trajectory $(0,\,z_0)$,
  numerically integrating the trajectories until the prescribed time
  $t$, and finally evaluating the empirical average to obtain an
  approximation to the solution. Detailed description of the algorithm
  and numerical tests will be presented in
  Section~\ref{sec:numerics}.}

This is a good place to connect to and compare with the surface
hopping algorithms in the chemistry literature. Our algorithm is based
on the stochastic process $\wt{z}_t$ which hops between two energy
surfaces, and thus it is very similar in spirit to the fewest switches
surface hopping and related algorithms. However, the jumping intensity
of $l_t$ is very different from what is used in the FSSH algorithm; in
fact, the hopping in FSSH is determined by an auxiliary ODE for the
evolution of ``population'' on the two surfaces \cite{Tully}. It is
not yet clear to us how such an ODE arises from the Schr\"odinger
equation. \red{On the other hand, given the trajectories produced as
  in FSSH, one could in fact re-weight those to calculate the path
  integral \eqref{eq:trajavg} which might correspond to an importance
  sampling scheme. This connection would be left for future explorations.}

Another major difference with the surface hopping algorithms proposed
in the chemistry literature is that the trajectory $\wt{z}_t$ is
continuous in time on the phase space, while in FSSH and other version
of surface hopping, a momentum shift is introduced to conserve the
classical energy along the trajectory when hopping occurs (if hopping
occurs from energy surface $0$ to $1$, it is required that $h_0(p, q)
= h_1(p', q)$ where $p'$ is the momentum after hopping). Note that as
in the FGA for single surface Schr\"odinger equation, each Gaussian
evolved in the FGA with surface hopping does not solve the matrix
Schr\"odinger equation, and only the average of trajectories gives an
approximation to the solution. Therefore, it is not necessary for each
trajectory to conserve the classical energy. The methods in the
chemistry literature perhaps over-emphasize the energy conservation of
a single trajectory. 

Also, Tully's fewest switches surface hopping algorithm only
calculates the trajectory, without giving an approximation to the wave
function. It is perhaps more like Heller's frozen Gaussian packet
\cite{Heller1} for single surface Schr\"odinger equation, which
compared to the ensemble view point of the Herman-Kluk propagator,
considers instead the evolution of a single Gaussian packet and
captures the correct semiclassical trajectory. The better
understanding of trajectory dynamics in FSSH is an interesting future
direction.

We emphasize that while an ensemble of trajectory is often used for
the surface hopping algorithm, it is rather unclear what the ensemble
average really means in the chemistry literature. There are in fact
debates on the interpretation of the surface hopping trajectories. Our
understanding on the path integral representation clarifies the
average of trajectories and hopefully will shed new light on further
development of the surface hopping algorithms.

Let us also point out that, as far as we have seen, the chemistry
literature seems to miss the weighting terms in \eqref{eq:trajavg}, 
resulting from the non-homogeneous state dependent Poisson process. As
the hopping rules for the surface hopping algorithm all have the
similar feature, this correction factor is very important. In fact,  the approximation is far
off without the correction factors in our numerical tests.

As we already mentioned before, the ansatz we used share some
similarity with those proposed by Wu and Herman in
\cites{WuHerman1,WuHerman2, WuHerman3}, in particular, the total wave
function is also split into a series of wave functions based on the
number of hoppings. However, they are crucially different in many
ways: Whether the trajectory is continuous in the phase space, whether
the weighting terms as in \eqref{eq:trajavg} is included in the average of trajectories, and the work \cites{WuHerman1,WuHerman2,
  WuHerman3} also employs some stationary phase argument, etc. While
we will provide a rigorous proof of the approximation error of our
methods, it is not clear to us that the heuristic asymptotics in
\cites{WuHerman1,WuHerman2, WuHerman3} can be rigorously justified.

\section{Asymptotic derivation}\label{sec:asymptotic}

We present in this section the asymptotic derivation of the FGA with
surface hopping \red{ansatz presented in \S\ref{sec:fgash}.} To
determine the equations for all the variables involved, we substitute
$u_{\FGA}$ into the Schr\"odinger equation
\eqref{SE2}\footnote{Alternatively, one can directly work with the
  matrix Schr\"odinger equation \eqref{vSE}, which will be in fact
  adopted in our proof in Section~\ref{sec:proof}. We present both
  view points as both are often used in the literature.}  and carry
out a matched asymptotics expansion. While the calculation in this
section is formal, the approximation error will be rigorously
\red{controlled} in Section~\ref{sec:convergence}.

We start by examining the term $(i \veps \partial_t - H)
K^{(0)}_{00}$.  By definition \eqref{eq:defK}, we have
\begin{align*}
  i \veps \partial_t  K^{(0)}_{00}  & = i \veps \Psi_0 \, \partial_t u^{(0)},  \\
\intertext{and}
H  K^{(0)}_{00}  & = \left( -\frac{\veps^2}{2} \Delta_x + H_e \right) \Psi_0 u^{(0)} \\
 & = -\frac{\veps^2}{2} \Delta_x \left(  \Psi_0 u^{(0)} \right) +  E_0 \Psi_0 u^{(0)} \\
 & = \Psi_0 H_0 u^{(0)} - \veps^2 \nabla_x \Psi_0 \cdot \nabla_x u^{(0)} +  \left( -\frac{\veps^2}{2} \Delta_x \Psi_0 \right) u^{(0)},
\end{align*}
where we have used the notation $H_i = -\frac{\veps^2}{2}\Delta_x +
E_i$ for $i = 0, 1$. Expand the term $\nabla_x \Psi_0$ in the
adiabatic basis $\{ \Psi_k\}_{k=0,1}$ (recall that we have assumed
only two adiabatic basis functions are important):
\[
\nabla_x \Psi_0 = d_{00}  \Psi_0 + \red{d_{10} }\Psi_1,  
\]
where we recall that
$\red{d_{nm}(x)= \langle \Psi_n, \nabla_x \Psi_m \rangle}$, and
we thus obtain the expansion of $H K^{(0)}_{00}$ in terms of the
adiabatic basis functions
\begin{equation}\label{eq:expanu0}
H  K^{(0)}_{00} = \Psi_0 H_0 u^{(0)} - \veps^2 \Psi_0 d_{00} \cdot  \nabla u^{(0)} - \veps^2 \Psi_1\red{ d_{10}} \cdot \nabla u^{(0)} + \Or(\veps^2), 
\end{equation}
where we have omitted the contribution from
$(-\frac{\veps^2}{2} \Delta_x \Psi_0) u^{(0)}$ which is of order
$\Or(\veps^2)$ \red{(note that the terms like
  $\veps^2 \Psi_0 d_{00}\cdot \nabla u^{(0)}$ is $\Or(\veps)$ instead
  of $\Or(\veps^2)$ due to the oscillation in $u^{(0)}$)}. We see that
the first two terms in \eqref{eq:expanu0} lie in the space spanned by
$\Psi_0$, while the third term is orthogonal. Hence, it is impossible
to construct $u^{(0)}$ to satisfy equation \eqref{eq:expanu0} to the
order of $\Or(\veps)$. In fact, the term
$- \veps^2 \Psi_1 \red{ d_{10} }\cdot \nabla u^{(0)}$ has to be
canceled by terms from $(i \veps \partial_t - H) K^{(1)}_{01}$, since
$\Psi_1$ corresponds to the other energy surface. This explains the
necessity of the surface hopping ansatz.

Let us thus first try to construct $u^{(0)}$ such that 
\begin{equation}\label{eq:singleu0}
 i \veps \partial_t u^{(0)} =H_0 u^{(0)} - \veps^2 d_{00}\cdot \nabla_x u^{(0)} + \Or(\veps^2).
\end{equation}
Note that this is very similar to the situation of the original frozen Gaussian approximation for the single surface Schr\"odinger equation. 
By direct calculation, we get
\begin{align*}
  i \veps \partial_t u^{(0)} & = \frac{i \veps }{(2\pi\veps)^{3m/2}}  \int \ud z_0 \; \partial_t \left[ A^{(0)}(t, z_0) \exp \Bigl( \frac{i}{\veps} \Theta^{(0)}(t,z_0,x) \Bigr)  \right] \\
                             & =  \frac{i\veps }{(2\pi\veps)^{3m/2}}  \int \ud z_0 \;  \partial_t A^{(0)}(t, z_0)  \exp \Bigl( \frac{i}{\veps} \Theta^{(0)}(t,z_0,x) \Bigr)  \\
                             & \qquad -  \frac{1}{(2\pi\veps)^{3m/2}}  \int \ud z_0 \;  A^{(0)}(t, z_0) \Bigl( \partial_t S_0^{(0)}(t, z_0) + \partial_t P_0^{(0)}(t, z_0) \cdot (x - Q_0^{(0)}(t, z_0)) \\
                             & \hspace{17em} - \partial_t Q_0^{(0)}(t, z_0) \cdot \bigl( P_0^{(0)}(t, z_0) + i (x - Q_0^{(0)}(t, z_0)) \bigr) \Bigr)     \exp \Bigl( \frac{i}{\veps} \Theta^{(0)}(t,z_0,x) \Bigr)\\
                             & =  \frac{i\veps }{(2\pi\veps)^{3m/2}}  \int \ud z_0 \;  \partial_t A^{(0)}(t, z_0)  \exp \Bigl( \frac{i}{\veps} \Theta^{(0)}(t,z_0,x) \Bigr)  \\
                             & \qquad -  \frac{1}{(2\pi\veps)^{3m/2}}  \int \ud z_0 \;  A^{(0)}(t, z_0) \Bigl( \partial_t S_0^{(0)}(t, z_0) - \nabla_Q E_0(Q_0^{(0)}(t, z_0)) \cdot (x - Q_0^{(0)}(t, z_0)) \\
                             & \hspace{18em} - P_0^{(0)}(t, z_0) \cdot \bigl( P_0^{(0)}(t, z_0) + i (x - Q_0^{(0)}(t, z_0)) \bigr) \Bigr)     \exp \Bigl( \frac{i}{\veps} \Theta^{(0)}(t,z_0,x) \Bigr).
\end{align*}
Moreover, 
\begin{align*}
  \veps^2 \nabla_x u^{(0)} & =  \frac{i \veps}{(2\pi\veps)^{3m/2}} \int  \ud z_0 \;  A^{(0)}(t, z_0) \bigl(  P^{(0)}(t, z_0) +i (x-Q^{(0)}(t, z_0)) \bigr)  \exp \Bigl( \frac{i}{\veps} \Theta^{(0)}(t, z_0, x) \Bigr), \\
  - \frac{\veps^2}{2} \Delta_x u^{(0)} & =  \frac{m\veps}{2} \frac{1}{(2\pi\veps)^{3m/2}} \int  \ud z_0 \;  A^{(0)}(t, z_0) \exp \Bigl( \frac{i}{\veps} \Theta^{(0)}(t, z_0, x) \Bigr) \\
  & \qquad + \frac{1}{2} \frac{1}{(2\pi\veps)^{3m/2}} \int  \ud z_0 \;  A^{(0)}(t, z_0) \bigl\lvert  P^{(0)}(t, z_0) +i (x-Q^{(0)}(t, z_0)) \bigr\rvert^2  \exp \Bigl( \frac{i}{\veps} \Theta^{(0)}(t, z_0, x) \Bigr).
\end{align*}
Suggested by the semiclassical limit of the single surface case, we
have imposed that $(Q^{(0)}, P^{(0)})$ follows the Hamiltonian flow
with classical Hamiltonian $h_0(q,p) = \frac{1}{2} \abs{p}^2 +
E_0(q)$, namely,
\begin{subequations}
  \begin{align}
    \frac{\ud Q^{(0)}}{\ud t} & = P^{(0)},  \\
    \frac{\ud P^{(0)}}{\ud t} & =- \nabla E_0 (Q^{(0)}). 
  \end{align}
\end{subequations}
To match the term $E_0(x) u^{(0)}$, we expand $E_0(x)$ around
$Q^{(0)}$ to get
\begin{align*}
  E_0(x) u^{(0)}&= \frac{1}{(2\pi\veps)^{3m/2}} \int  \ud z_0 \; E_0(Q^{(0)}(t, z_0)) A^{(0)}(t, z_0) \exp \Bigl( \frac{i}{\veps} \Theta^{(0)}(t, z_0, x) \Bigr) \\
  &\qquad +\frac{1}{(2\pi\veps)^{3m/2}} \int  \ud z_0 \;  (x-Q^{(0)}(t, z_0)) \cdot \nabla_Q E_0(Q^{(0)}(t, z_0)) A^{(0)}(t, z_0) \exp \Bigl( \frac{i}{\veps} \Theta^{(0)}(t, z_0, x)  \Bigr)\\
  &\qquad +\frac{1}{(2\pi\veps)^{3m/2}} \int  \ud z_0 \;  \frac{1}{2} (x-Q^{(0)}(t, z_0)) \cdot \nabla_Q^2 E_0(Q^{(0)}(t, z_0)) (x-Q^{(0)}(t, z_0))  A^{(0)}(t, z_0) \exp \Bigl( \frac{i}{\veps} \Theta^{(0)}(t, z_0, x)  \Bigr) \\
   &\qquad + \Or(\veps^2), 
\end{align*}
\red{where we have used Lemma~\ref{lem:asym} to control the terms
containing $(x - Q^{(0)})^3$ and higher order terms.}  To treat the
terms containing powers of $(x - Q^{(0)})$ in the above expressions,
we apply Lemma~\ref{lem:asym} and get
\begin{align*}
  i \veps \partial_t u^{(0)} & =\frac{i\veps }{(2\pi\veps)^{3m/2}}  \int \ud z_0 \;  \partial_t A^{(0)}(t, z_0)  \exp \Bigl( \frac{i}{\veps} \Theta^{(0)}(t,z_0,x) \Bigr)  \\
                             & \qquad -  \frac{1}{(2\pi\veps)^{3m/2}}  \int \ud z_0 \;  A^{(0)}(t, z_0) \Bigl( \partial_t S_0^{(0)}(t, z_0)  -| P_0^{(0)}(t, z_0)|^2 \Bigr)     \exp \Bigl( \frac{i}{\veps} \Theta^{(0)}(t,z_0,x) \Bigr)
                             \\
                              & \qquad - \frac{\veps}{(2\pi\veps)^{3m/2}} \int  \ud z_0 \; \partial_{z_k}\Bigl( A^{(0)}\bigl(\nabla_Q E_0(Q^{(0)})+iP^{(0)}\bigr)_j (Z^{(0)})_{jk}^{-1}\Bigr) \exp \Bigl( \frac{i}{\veps} \Theta^{(0)}(t, z_0, x) \Bigr),
  \\
  \veps^2 \nabla_x u^{(0)}  & =  \frac{i \veps}{(2\pi\veps)^{3m/2}} \int  \ud z_0 \;  A^{(0)}(t, z_0)  P^{(0)}(t, z_0)    \exp \Bigl( \frac{i}{\veps} \Theta^{(0)}(t, z_0, x) \Bigr) + \mathcal O(\veps^2),  \\
   - \frac{\veps^2}{2} \Delta_x u^{(0)} & =   \frac{1}{2} \frac{1}{(2\pi\veps)^{3m/2}} \int  \ud z_0 \;  A^{(0)}(t, z_0) \bigl\lvert  P^{(0)}(t, z_0)  \bigr\rvert^2  \exp \Bigl( \frac{i}{\veps} \Theta^{(0)}(t, z_0, x) \Bigr) \\ \nonumber
   & \qquad + \frac{m\veps}{2} \frac{1}{(2\pi\veps)^{3m/2}} \int  \ud z_0 \;  A^{(0)}(t, z_0) \exp \Bigl( \frac{i}{\veps} \Theta^{(0)}(t, z_0, x) \Bigr) \\
   & \qquad - \frac{i\veps}{(2\pi\veps)^{3m/2}} \int  \ud z_0 \; \partial_{z_k}\Bigl( A^{(0)}\bigl(P^{(0)}\bigr)_j (Z^{(0)})_{jk}^{-1}\Bigr) \exp \Bigl( \frac{i}{\veps} \Theta^{(0)}(t, z_0, x) \Bigr) \\
   & \qquad - \frac{\veps}{2(2\pi\veps)^{3m/2}} \int  \ud z_0 \;  A^{(0)}(t, z_0) \partial_{z_l}(Q^{(0)})_j (I_m)_{jk} (Z^{(0)})^{-1}_{kl}  \exp \Bigl( \frac{i}{\veps} \Theta^{(0)}(t, z_0, x) \Bigr) + \mathcal O(\veps^2),\\
   E_0(x) u^{(0)}&= \frac{1}{(2\pi\veps)^{3m/2}} \int  \ud z_0 \; E_0(Q^{(0)}(t, z_0)) A^{(0)}(t, z_0) \exp \Bigl( \frac{i}{\veps} \Theta^{(0)}(t, z_0, x) \Bigr) \\
  & \qquad- \frac{\veps}{(2\pi\veps)^{3m/2}} \int  \ud z_0 \; \partial_{z_k}\Bigl( A^{(0)}\bigl(\nabla_Q E_0(Q^{(0)})\bigr)_j (Z^{(0)})_{jk}^{-1}\Bigr) \exp \Bigl( \frac{i}{\veps} \Theta^{(0)}(t, z_0, x) \Bigr)\\
   & \qquad + \frac{\veps}{2(2\pi\veps)^{3m/2}} \int  \ud z_0 \;  A^{(0)}(t, z_0) \partial_{z_l}(Q^{(0)})_j (\nabla_Q^2 E_0(Q^{(0)})_{jk} (Z^{(0)})^{-1}_{kl}  \exp \Bigl( \frac{i}{\veps} \Theta^{(0)}(t, z_0, x) \Bigr) + \mathcal O(\veps^2).
\end{align*}
Therefore, matching terms on the leader order, we get
\begin{equation}
\frac{\ud S^{(0)}}{\ud t} = \frac{1}{2} \Abs{ P^{(0)}}^2 - E_0 (Q^{(0)}). 
\end{equation}
The next order $\Or(\veps)$ gives 
\begin{equation}\label{eq:C}
\frac{\ud}{\ud t} A^{(0)}  =\frac 1 2  A^{(0)} \tr\Bigl( (Z^{(0)})^{-1}\bigl(\partial_z P^{(0)} - i \partial_z Q^{(0)}  \nabla^2_Q E_0(Q^{(0)}) \bigr) \Bigr) - A^{(0)} d_{00}\cdot P^{(0)},
\end{equation}
where
\[
\partial_z=\partial_q - i \partial_p,\quad\mbox{and}\quad Z^{(0)}=\partial_z \left(Q^{(0)}+i P^{(0)}\right).
\]
Note that, compared with single surface case, the only difference is
the extra term $- A^{(0)} d_{00}\cdot P^{(0)}$, which comes from the
term $ - \veps^2 d_{00}\cdot \nabla_x u^{(0)}$ in
\eqref{eq:singleu0}. We also remark that $d_{00}$ is purely imaginary
due to the normalization of $\Psi_0$, and hence this extra term only
contributes to an extra phase of $A^{(0)}$.

Coming back to \eqref{eq:expanu0}, we still need to take care of the
term parallel to $\Psi_1$. This extra term corresponds to intersurface
mixing and should be canceled by contributions from \red{wave packets on the other surface}. More specifically, let us examine the term
$(i \veps \partial_t - H) K^{(1)}_{01}$, by direct calculations, we
get
\begin{align}
i \veps \partial_t K^{(1)}_{01} &=i \veps  \Psi_1 \partial_t u^{(1)} \nonumber \\
& =  i \veps \frac{1}{(2\pi\veps)^{3m/2}} \int  \ud z_0 \int_0^t d t_1 \tau^{(1)}(t_1) \partial_t \left( A^{(1)}(t,t_1,z_0) \exp\left(\frac{i}{\veps}\Theta^{(1)}(t,t_1,z_0,x)\right) \right) \Psi_1 \label{term:2-1} \\
& \quad  + i \veps \frac{1}{(2\pi\veps)^{3m/2}} \int  \ud z_0 \; \tau^{(1)}(t)  A^{(1)}(t,t,z_0) \exp\left(\frac{i}{\veps}\Theta^{(1)}(t,t,z_0,x)\right) \Psi_1, \label{term:2-2} \\
\intertext{and}
H K^{(1)}_{01} & = \Bigl(-\frac{\veps^2}{2} \Delta_x + H_e\Bigr) \Psi_1 u^{(1)} \nonumber \\
& =  -\frac{\veps^2}{2} \Delta_x \bigl(  \Psi_1 u^{(1)} \bigr) + E_1  \Psi_1 u^{(1)} \nonumber \\
& =\Psi_1 H_1 u^{(1)} - \veps^2 \nabla_x \Psi_1 \cdot \nabla_x u^{(1)} +  \left( -\frac{\veps^2}{2} \Delta_x \Psi_1 \right) u^{(1)},  \label{term:2-3}
\end{align}
where $H_1=-\frac{\veps^2}{2} \Delta_x +E_1$ is the effective Hamiltonian on the second energy surface.

Therefore, in $(i \veps \partial_t - H) K^{(1)}_{01}$, all the terms
contain the time integration with respect to $t_1$ except the term
\eqref{term:2-2}, which motivates us to impose this term to cancel the
term $-\veps^2 \Psi_1 \red{d_{10}}\cdot \nabla_x u^{(0)}$ from
$(i \veps \partial_t - H) K^{(0)}_{00}$. From the expression of
$\nabla_x u^{(0)}$, this suggests that we shall construct
$\tau^{(1)}, A^{(1)}$, and $\Theta^{(1)}$ such that
\begin{multline}
  \frac{1}{(2\pi\veps)^{3m/2}} \red{ d_{10}}(x) \cdot  \int  \ud z_0 \; A^{(0)}(t, z_0) \bigl( P^{(0)}(t, z_0) +i (x-Q^{(0)}(t, z_0)) \bigr)  \exp \Bigl( \frac{i}{\veps} \Theta^{(0)}(t, z_0, x) \Bigr) = \\
 =  - \frac{1}{(2\pi\veps)^{3m/2}} \int \ud z_0 \; 
 A^{(1)}(t,t,z_0) \tau^{(1)}(t, z_0)
  \exp\Bigl(\frac{i}{\veps}\Theta^{(1)}(t,t,z_0,x)\Bigr) + \Or(\veps). 
\end{multline}
Expand $\red{ d_{10}}(x)$ around $Q^{(0)}$ and apply Lemma~\ref{lem:asym} again, we want
\begin{multline}
  \frac{1}{(2\pi\veps)^{3m/2}}   \int  \ud z_0 \; A^{(0)}(t, z_0)  \red{d_{01}}\bigl(Q^{(0)}(t, z_0)\bigr) \cdot  P^{(0)}(t, z_0) \exp \Bigl( \frac{i}{\veps} \Theta^{(0)}(t, z_0, x) \Bigr) = \\
 =  -  \frac{1}{(2\pi\veps)^{3m/2}} \int \ud z_0 \; 
 A^{(1)}(t,t,z_0) \tau^{(1)}(t, z_0)
  \exp\Bigl(\frac{i}{\veps}\Theta^{(1)}(t,t,z_0,x)\Bigr) + \Or(\veps). 
\end{multline}
A natural choice is then to set for any $t$ and $z_0$, 
\begin{align}
  & A^{(1)}(t, t, z_0) = A^{(0)}(t, z_0), \label{eq:A1A0}\\
  & P^{(1)}(t, t, z_0) = P^{(0)}(t, z_0), \\
  & Q^{(1)}(t, t, z_0) = Q^{(0)}(t, z_0), \\
  & S^{(1)}(t, t, z_0) = S^{(0)}(t, z_0), \label{eq:S1S0}\\
  & \tau^{(1)}(t, z_0) = - \red{d_{10}}\bigl(Q^{(0)}(t, z_0)\bigr) \cdot  P^{(0)}(t, z_0). \label{eq:tau1}
\end{align}
Therefore, this sets the initial conditions of
$A^{(1)}(t, t_1, z_0), P^{(1)}(t, t_1, z_0), Q^{(1)}(t, t_1, z_0),
S^{(1)}(t, t_1, z_0)$
at $t = t_1$ (recall that from the definition of \eqref{eq:u0n}, those
FGA variables are only needed for $t \ge t_1$).

Now back to the other terms in
$(i \veps \partial_t - H) K^{(1)}_{01}$. Again, in \eqref{term:2-3},
the term $( -\frac{\veps^2}{2} \Delta_x \Psi_1 ) u^{(1)}$ is of order
$\mathcal{O}(\veps^2)$, which will be neglected as the FGA approximation is
determined by terms up to $\Or(\veps)$. We expand $\nabla_x \Psi_1$ in
terms of the adiabatic states
$\nabla_x \Psi_1 = \red{d_{01}} \Psi_0 + d_{11} \Psi_1$, the contribution of
$\Psi_1$ will be asymptotically matched by imposing the appropriate evolution equation for $A^{(1)}$, while the term component in
$\Psi_0$, $\veps^2 \Psi_0 \red{d_{01} }\cdot \nabla_x u^{(1)}$,  has to be
matched by contributions from $(i \veps \partial_t - H) K^{(2)}_{00}$.

Analogously to the construction of $u^{(0)}$, we impose that for $t \ge t_1$, $(P^{(1)}, Q^{(1)})$ satisfy the Hamiltonian flow with the effective Hamiltonian $h_1 = \frac{1}{2} \abs{p}^2 + E_1(q)$. 
\begin{align*}
\frac{\ud}{\ud t} Q^{(1)} & = P^{(1)},  \\
\frac{\ud}{\ud t} P^{(1) }& = -\nabla E_1 (Q^{(1)}).
\end{align*}
The evolution of other FGA variables can be determined by matched
asymptotics, also similar to what was done for $u^{(0)}$. This leads
to for $t \ge t_1$,
\begin{align*}
\frac{\ud}{\ud t} S^{(1)} & = \frac{1}{2} ( P^{(1)} )^2 - E_1 (Q^{(1)}), \\
\frac{\ud}{\ud t} A^{(1)} & = \frac{1}{2}  A^{(1)} \tr\Bigl( (Z^{(1)})^{-1}\bigl(\partial_z P^{(1)} - i \partial_z Q^{(1)}  \nabla^2_Q E_1(Q^{(1)}) \bigr) \Bigr) - A^{(1)} d_{11}\cdot P^{(1)}, 
\end{align*}
with initial conditions given by the continuity condition \red{\eqref{eq:A1A0} and \eqref{eq:S1S0}}. 
Note that the equations have the same structures as the evolution
equations for $S^{(0)}$ and $A^{(0)}$, except that they now evolve on
the energy surface $E_1$.

In a similar way, the evolution equations \eqref{eq:evenevolve}  for all order FGA variables and hopping
coefficients are recursively determined, with continuity condition
\eqref{eq:contcond} at the hopping.

\smallskip 

To end this part, let us generalize the ansatz to the case 
that initial wave functions consist of both $\Psi_0$ and $\Psi_1$. Since the equation is linear, the solution is given by superposition of initial conditions concentrating on each energy surface. Thus, in the general case, the FGA with surface hopping is given by 
\begin{equation}\label{ansatz2}
u_{\FGA}(t,x,r)= K^{(0)}(t,x,r)+ K^{(1)}(t,x,r)+K^{(2)}(t,x,r)+K^{(3)}(t,x,r)+\cdots,
 \end{equation}
where
\[
K^{(n)}=
\begin{cases}
K^{(n)}_{01}+K^{(n)}_{10}= \Psi_1 u^{(n)}_0 + \Psi_0 u^{(n)}_1, & n \; \text{odd},  \\
K^{(n)}_{00}+K^{(n)}_{11}= \Psi_0 u^{(n)}_0 + \Psi_1 u^{(n)}_1, & n \; \text{even}. 
\end{cases}
\]
The expression for $u_0^{(n)}$ and $u_1^{(n)}$ are similar as the
previous case, and hence will be omitted.


\section{Convergence of FGA with surface
  hopping}\label{sec:convergence}

In this section, we will present the main approximation theorem of FGA
with surface hopping and also provide some examples to justify \red{and understand} the assumptions.

\subsection{Assumptions and main theorem}

For the asymptotic convergence of the frozen Gaussian approximation
with surface hopping ansatz, as $\Psi_j(r; x)$ are fixed, the heart of
matter is the approximation of $U = \bigl(\begin{smallmatrix} u_0 \\
  u_1 \end{smallmatrix} \bigr)$, which solves the matrix Schr\"odinger
equation \eqref{vSE}. Recall that in the FGA ansatz, $U$ is
approximated by (note that we have assumed $u_1(0,x) = 0$)
\begin{equation}
  U_{\FGA}(t, x) = 
  \begin{pmatrix}
    u^{(0)} + u^{(2)} + \cdots \\
    u^{(1)} + u^{(3)} + \cdots
  \end{pmatrix}.
\end{equation}
To guarantee the validity of the asymptotic matching, we make some
natural assumptions of $E$, $d$ and $D$, the coefficients appeared in
\eqref{vSE}. 

Beside the coupling terms in the matrix Schr\"odinger equation
\eqref{vSE}, the non-adiabatic transition is also related to the 
gap between the adiabatic energy surfaces, given by
\begin{equation}
  \delta := \inf_x \bigl( E_1(x) - E_0(x) \bigr).
\end{equation}
\red{In the most interesting non-adiabatic regime, $\delta>0$} should also be viewed as a small parameter.  In
fact, if $\delta$ is fixed and $\veps \rightarrow 0$, the matrix
Schr\"odinger equation \eqref{vSE} is approaching its adiabatic limit,
namely we can neglect the transition to the other part of the spectrum; see
\cite{BO1,BO2}. To ensure significant amount of transition as
$\veps \rightarrow 0$, it is most interesting to consider
$\delta \rightarrow 0$ simultaneously.  Therefore, we will consider a family of matrix
Schr\"odinger equations with the coefficients depending on $\delta$.
We will emphasize the $\delta$ dependence and write $H^\delta_e$,
$E^\delta$, $d^\delta$, $D^\delta$, etc. when confusion might occur.

We start by the assumption on the energy surface $E^\delta$. 
\begin{hyp}\label{assuma}
  Each energy surface $E^\delta_k(q) \in C^{\infty}(\R^m)$,
  $k\in \{0, 1\}$ and satisfies the following subquadratic condition,
  where the constant $C_E$ is uniform with respect to $\delta$,
  \begin{equation}\label{cond:subqua}
    \sup_{q \in \RR^m} \Abs{\partial_{\alpha} E^{\delta}_k (q)} \leq C_E, \quad \forall\, \abs{\alpha} = 2. 
  \end{equation}
\end{hyp}
This assumption guarantees that the Hamiltonian flow of each energy
surface satisfies global Lipschitz conditions, such that the global
existence of the flow is guaranteed. In particular, \eqref{cond:subqua} immediately implies that 
\begin{equation}
    \bigl\lvert \nabla_q E^\delta_k(q) \bigr\rvert \leq C_E \bigl(\abs{q} + 1\bigr), \quad \forall\, q \in \RR^m, 
\end{equation}
which we will use later. Note that similar subquadratic assumptions
are needed even for the validity of FGA method for the single surface
model.  We will further investigate the related properties of the
Hamiltonian flow in Section~\ref{sec:prelim}.

We recall that the initial coefficient of the \red{surface hopping ansatz} is given by
\begin{equation}\label{eq:C00}
  A^{(0)}(0, z) = 2^{m/2} \int_{\R^m} e^{\frac{i}{\veps}(-p \cdot (y-q) + \frac{i}{2} \abs{y - q}^2)} u_0(y) \ud y. 
\end{equation}
By Lemma~\ref{lem:initial}, the FGA ansatz recovers the initial
condition of the Schr\"odinger equation at time $0$.

As the Gaussian is not compactly supported, $A^{(0)}$ is in general
not compactly supported either. This causes some trouble as for
example the hopping coefficient $\tau$, given by \textit{e.g.}, $\red{ -}p
\cdot d^\delta_{01}(q)$, is not uniformly bounded with respect to all
starting points $z_0$ of the trajectory.  On the other hand, notice
that $A^{(0)}$ will decay very fast on the phase space, especially
when $\veps$ is small, since the Gaussian $e^{\frac{i}{\veps}(-p \cdot
  (y-q) + \frac{i}{2} \abs{y - q}^2)}$ decays very fast both in 
real and Fourier spaces. Therefore, we may truncate $A^{(0)}$
outside a compact set which only introduces a small error to the
approximation to the initial condition. As this issue arises already
in the usual frozen Gaussian approximation (see \cite{FGA_Conv} for
example), we will make the assumption that there exists a compact set
$K \subset \R^{2m}$ that the initial approximation error, given by
\begin{equation}\label{eq:defein}
  \eps_{\text{in}} = \Norm{  \frac{1}{(2\pi \veps)^{3m/2}} \int_K A^{(0)}(0, z) e^{\frac{i}{\veps} \Phi^{(0)}(0, x, z)} \ud z - u_0(0, x) }_{L^2(\R^m)}
\end{equation}
is negligibly small for the accuracy requirement. Note that
$\eps_{\text{in}}$ depends on the semiclassical parameter $\veps$ and
goes to zero as $\veps \to 0$ (the rate depends on $u_0(0)$). 
Therefore, we will restrict the initial condition $z_0 = (q_0, p_0)$
to the compact set $K$ in the FGA ansatz with surface hopping. 

Moreover, as we shall prove in Proposition \ref{prop:comp}, given
$t>0$ and initial condition $z_0 = (q_0, p_0)$ restricted in a compact
set $K$, the FGA trajectory is confined in a compact set $K_t$ which
is independent of the hopping history within $[0,t]$. For the matrix
Schr\"odinger equation, we also need the assumption on the coupling
coefficients $d^\delta$ and $D^\delta$, and their boundedness on the
compact set $K_t$.
\begin{hyp}\label{assumb}
  For $k,l \in \{0,1\}$, we have $d^\delta_{kl} \in C^{\infty}(\R^m)$
  and $D^\delta_{kl} \in C^{\infty}(\R^m)$. Moreover, given $t>0$, we
  assume that $E^\delta_k$, $d^\delta_{kl}$ and $D^\delta_{kl}$ and
  their derivatives are uniformly bounded with respect to $\delta$ on
  $K_t$, which is the compact set the trajectory stays within up to
  time $t$ (see Proposition~\ref{prop:comp}).
\end{hyp}

To further understand the implications of Assumption~\ref{assumb}, we
have for $k \neq l$ the explicit expression by standard perturbation
theory, when $E^\delta_k \ne E^\delta_l$,
\[
\red{d^\delta_{lk}(x)}=\langle \Psi^{\delta}_l, \nabla_x \Psi^{\delta}_k
\rangle=\frac{\langle \Psi^{\delta}_l, (\nabla_x H^\delta_e)
  \Psi^{\delta}_k \rangle}{E^\delta_k-E^\delta_l}.
\]
As the denominator on the right hand side is given in terms of the
energy gap, when the gap approaches $0$, the coupling vector becomes
unbounded unless the numerator is also getting small. 
\red{On the other hand, we allow the possibility that the gap between
  the two energy surface is very small, even of the order of $\veps$
  (but $d^\delta$ is still $\Or(1)$).}  Examples will be given in the
next subsection.

Now we are ready to state the main approximation theorem. 

\smallskip

\begin{theorem}\label{thm:main}
  Let $U_{\FGA}(t,x)$ be the approximation given by the FGA with
  surface hopping (with phase space integral restricted to $K$) for
  the Schr\"odinger equation \eqref{vSE}, whose exact solution is
  denoted by $U(t,x)$. Under Assumptions~\ref{assuma} and
  \ref{assumb}, for any given final time $t$, there exists a constant
  $C$, such that for any $\veps>0$ sufficiently small and any
  $\delta > 0$, we have
  \begin{equation*}
    \Norm{U_{\FGA}(t, x) -  U(t, x)}_{L^2(\R^m)} \le C \veps + \eps_{\text{in}},
  \end{equation*}
  where $\eps_{\text{in}}$ is the initial approximation error defined in \eqref{eq:defein}.
\end{theorem}

This theorem implies, in the simultaneous limit, $\eps \rightarrow 0$,
$\delta \rightarrow 0$, the FGA method with surface hopping remains a
valid approximation with $\mathcal O(\veps)$ error. This covers the
interesting regime when the transition between surfaces is not negligible
($\Or(1)$ transitions occur) and also the adiabatic regime that
$\veps \to 0$ with a fixed $\delta$ (so non-adiabatic transition is
negligible). In particular, we emphasize that the constant $C$ is
independent of both $\veps$ and $\delta$.

While we will not keep track the precise dependence of the constant
$C$ on $t$, by our proof techniques, we would 
\red{at best prove an exponential growth of the 
constant as $t$ gets large,} due to the use of
Gronwall type inequalities. In our numerical experience, the error
accumulation seems milder.

\subsection{Examples of matrix Schr\"odinger equations} \label{sec:example}

Let us explore some specific examples to better understand the
Assumptions~\ref{assuma} and \ref{assumb}.  Recall that, we consider
the case of two adiabatic states, which means that the Hilbert space
corresponding to the electronic degree of freedom is equivalent to
$\C^2$, and hence the electronic Hamiltonian $H^\delta_e$ is
equivalent to a $2\times 2$ matrix. As discussed above, the most
interesting scenario is when the two surfaces are not well separated. We
recall that the small parameter $\delta$ indicates the gap between the
two energy surfaces and focus on the more interesting cases that the
gap goes to $0$ as $\veps \to 0$.

A general class of electronic Hamiltonian satisfying our assumptions can be given as a product of a scalar function $F^\delta(x)$ and a $2 \times 2$ matrix $M(x)$ independent of $\delta$, namely
\begin{equation*}
H_e^{\delta}(x)= F^\delta(x) M(x).
\end{equation*}
We observe that, due to the specific choice, $H_e^{\delta}$ and $M$
share the same eigenfunctions, and if we denote the eigenvalues of $M$
by $\lambda_k$, then we have
\begin{equation}\label{case1}
E_k^{\delta}(x) =F^\delta (x)  \lambda_k(x).
\end{equation}
Then, we obtain that, for $k \ne l$,
\[
\red{ d^\delta_{lk}}=\frac{\langle \Psi_l, \nabla_x H_e^\delta \Psi_k \rangle}{E_k^{\delta}-E_l^{\delta}} = \frac{ F^\delta \langle \Psi_l, \nabla_x M \Psi_k \rangle}{F^\delta (\lambda_k-\lambda_l)} = \frac{ \langle \Psi_l, \nabla_x M \Psi_k \rangle}{\lambda_k-\lambda_l}. 
\]
Similarly, one can show that,
\[
\red{D^\delta_{lk}}=\langle \Psi_l, \Delta_x \Psi_k \rangle=\frac{\langle \Psi_l, \Delta_x M \Psi_k \rangle-2 \nabla_x \lambda_k \cdot d_{kl}}{\lambda_k-\lambda_l}.
\]
Therefore, $d^\delta$ and $D^\delta$ are independent of independent of $\delta$, and we thereby suppress the appearance of $\delta$. Moreover,  $d$ and $D$ are independent of $F^\delta$, while
we can take $F^{\delta}$ such that energy surfaces become close and
even touch each other as $\delta \to 0$.  The set of almost degenerate
points of the energy surfaces may consist of one single point, several
points, or even \red{an interval}, as we will see below.

We note that even though this construction looks rather special, it is
actually versatile enough to cover many examples considered in the
chemistry literature. We will present three model problems here, which are 
adapted from Tully's original examples in \cite{Tully}.

\noindent\textbf{Example 1(a). Simple avoided crossing.}  We choose
$M$ to be
\[
M=
\begin{pmatrix}
\frac{\tanh (x)}{2\pi}& \frac{1}{10}\\
\frac{1}{10} & -\frac{\tanh (x)}{2\pi}
\end{pmatrix}.
\]
The eigenvalues of $M$ are \[
\pm \sqrt{\frac{\tanh^2 (x)}{4\pi^2}+\frac{1}{100} }.
\] 
We observe that the two eigenvalue surfaces are close around $x=0$. By the plots of $d_{01}$ and $D_{01}$ in Figure \ref{fig:a2}, we see the coupling is significant  around $x=0$ as well. To control the energy gap, we introduce the following $F^\delta$ function,
\[
F^\delta (x)= 1 +(\delta -1 ) e^{-10 x^2},  
\] 
such that $F^\delta(x) = \mathcal{O}(\delta)$ around $x=0$ and
$F^{\delta}(0) = \delta$, so that the energy gap vanishes at $x = 0$
as $\delta \to 0$. The eigenvalues of $H_e$ for different values of
$\delta$ are plotted in Figure \ref{fig:a2}. 
   \begin{figure}
\includegraphics[scale=0.5]{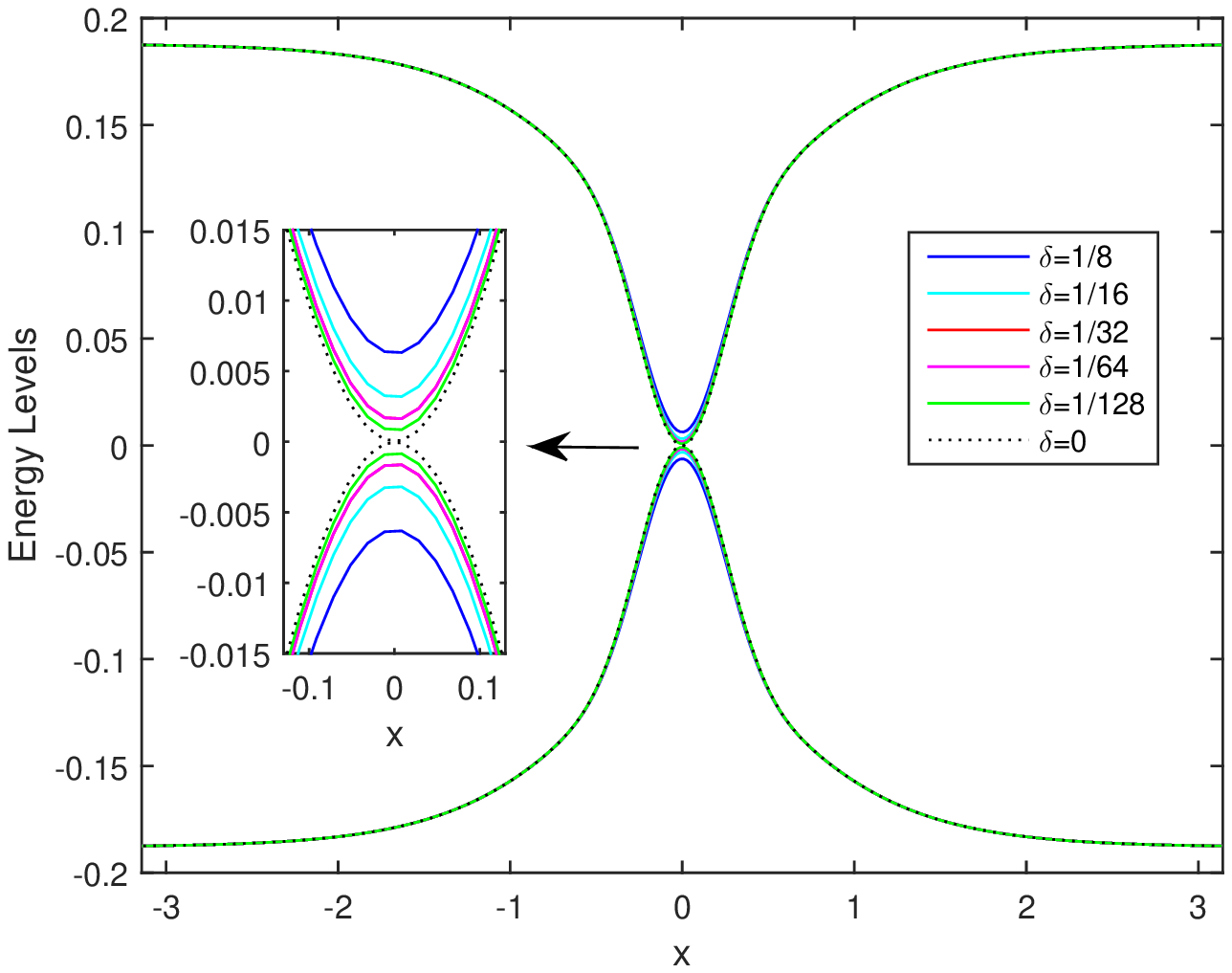}\includegraphics[scale=0.5]{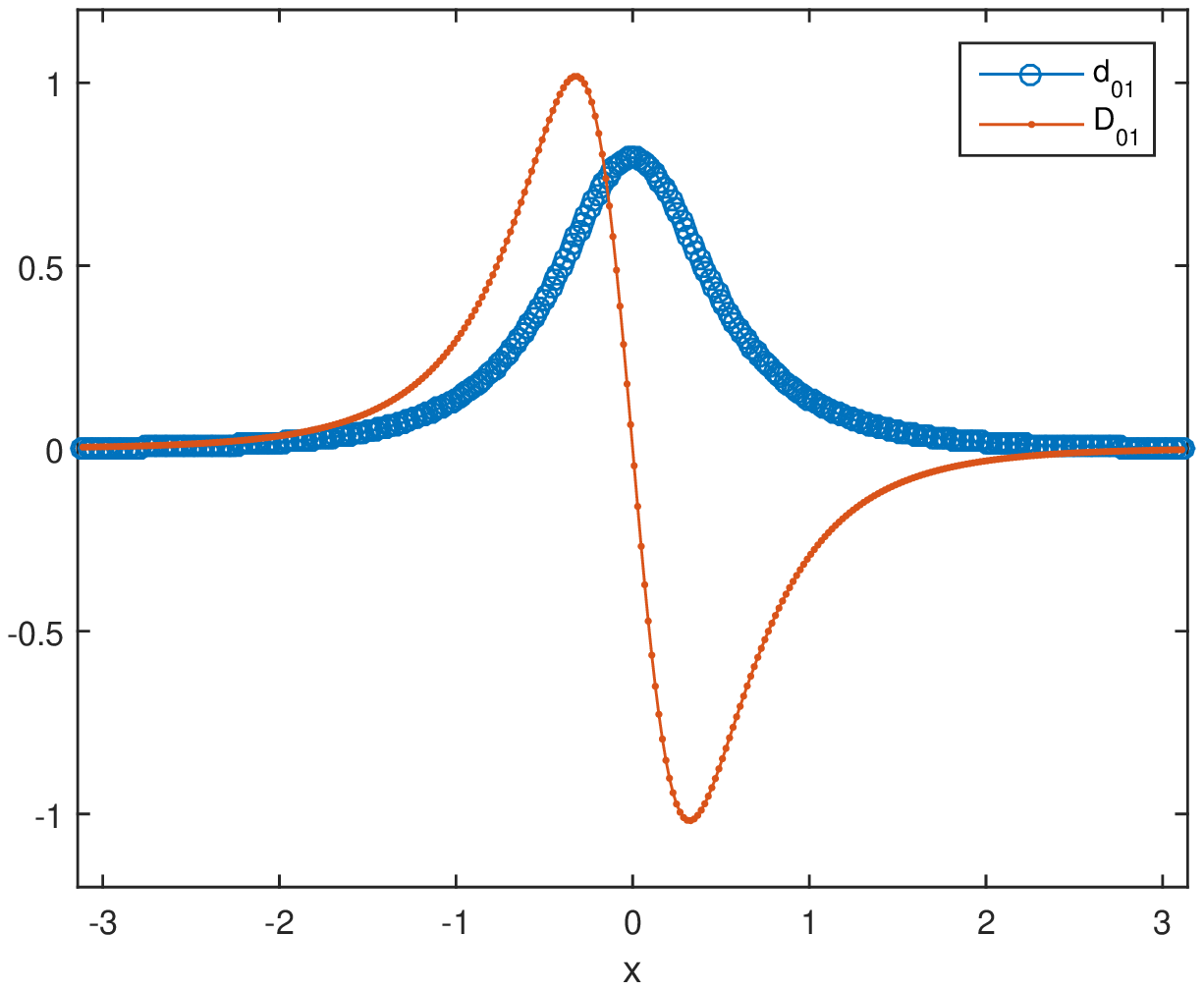} \\
\caption{(Example 1(a)) Left: Eigenvalues of $H_e$, $\delta =\frac 1 8$,
  $\frac 1 {16}$, $\frac 1 {32}$, $\frac 1 {64}$ and $\frac{1}{128}$; reference $\delta=0$. Right:
  the coupling information of $H_e$, invariant with respect to $\delta$.}
\label{fig:a2}
\end{figure}

\noindent\textbf{Example 1(b). Dual avoided crossing.} We choose $M$ to be
\[
M=
\begin{pmatrix}
0 & \frac 1 {20} \\
\frac 1 {20} & -e^{-\frac{x^2}{10}}+\frac 1 2
\end{pmatrix},
\]
The eigenvalues of $M$ are 
\[
\frac  1 2 \left(-e^{- \frac{x^2}{10}}+ \frac 1 2 \right) \pm \sqrt{\frac  1 4 \left(-e^{- \frac{x^2}{10}} + \frac 1 2\right)^2+ \frac{1}{400}}. 
\]
We observe that, the two eigenvalues are closest to each other when $-e^{- \frac{x^2}{10}}+ \frac 1 2=0$, or $x=\pm \sqrt{10 \ln 2}$. The coupling vectors  around these points $x=\pm \sqrt{10 \ln 2}$ are significantly larger than their values elsewhere as shown in Figure \ref{fig:b2}. This explains why the model is often referred to as the dual avoided crossing.

To control the energy gap, we may introduce the following $F^\delta$ function,
\[
F^\delta (x)= 1+ e^{- (2\sqrt{10 \ln 2})^2} +(\delta -1 ) e^{- (x+\sqrt{10 \ln 2})^2}+(\delta -1 ) e^{-(x+\sqrt{2 \ln 2})^2}. 
\]
We can check that $F^\delta = \mathcal O(\delta)$ around  $x=\pm \sqrt{10 \ln 2}$, and
\[
\lim_{\delta \rightarrow 0} F^{\delta} (\pm \sqrt{10 \ln 2}) =
\lim_{\delta \rightarrow 0} \delta \left(1+e^{- (10\sqrt{2 \ln 2})^2}
\right)=0.
\]
Thus the energy gap vanishes at the two points as $\delta \to 0$. This is illustrated in Figure \ref{fig:b2}.
   \begin{figure}
\includegraphics[scale=0.5]{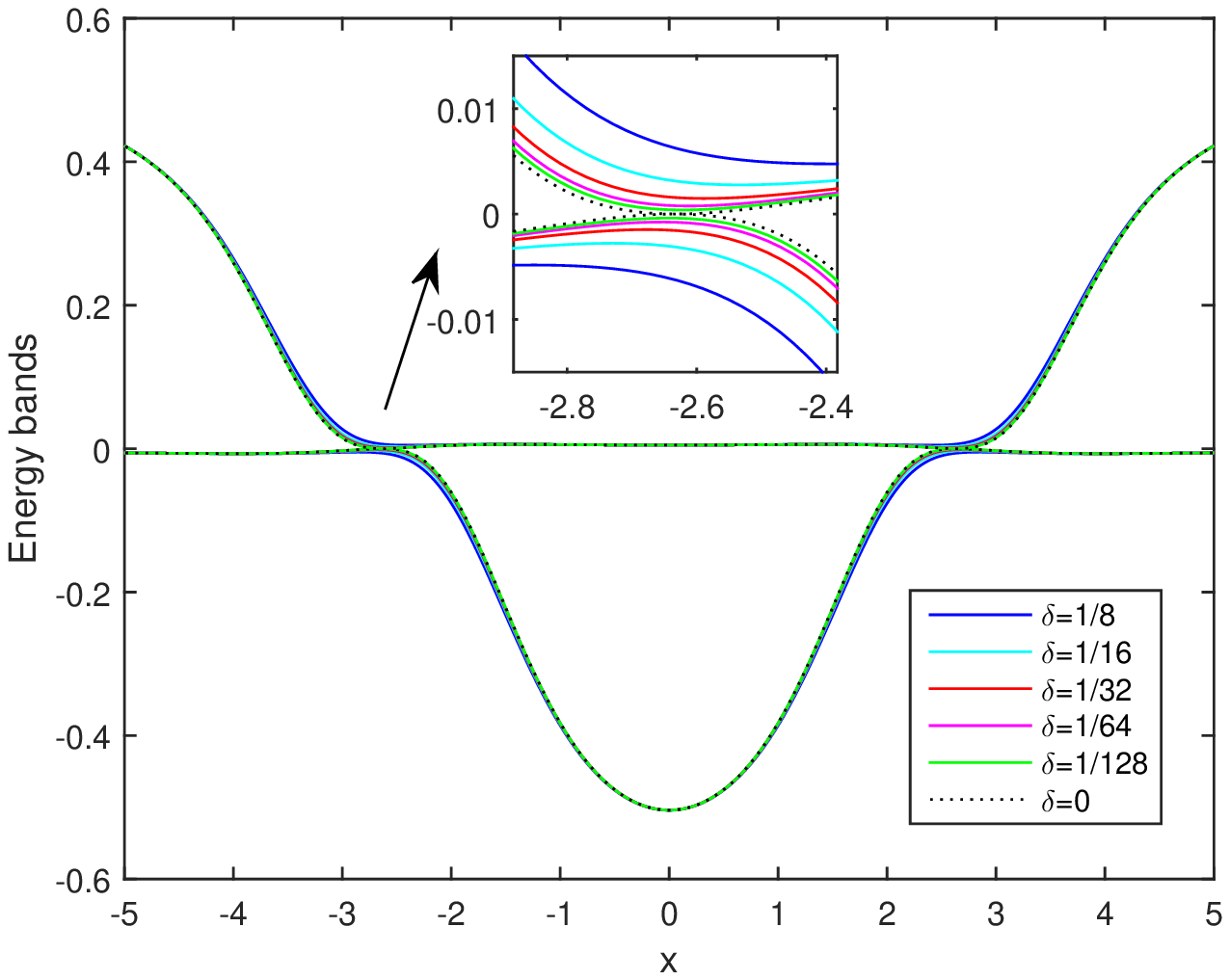}\includegraphics[scale=0.5]{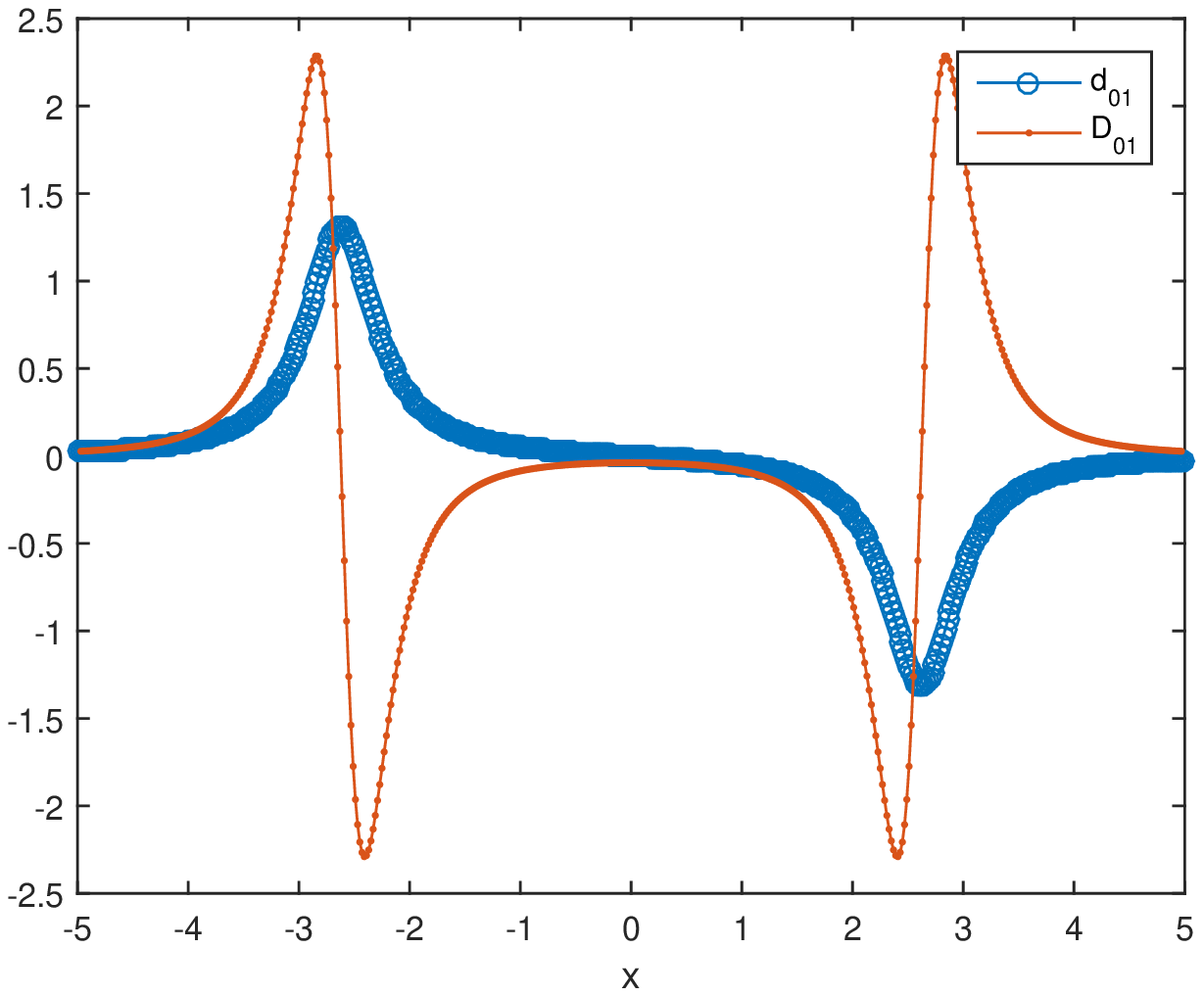} \\
\caption{(Example 1(b)) Left: Eigenvalues of $H_e$,  $\delta =\frac 1 8$,
  $\frac 1 {16}$, $\frac 1 {32}$, $\frac 1 {64}$ and $\frac{1}{128}$; reference $\delta=0$. Right:
  the coupling information of $H_e$, invariant with respect to $\delta$.}
\label{fig:b2}
\end{figure}

\noindent\textbf{Example 1(c). Extended coupling with reflection.} 
In this example, $M$ is set to be 
\[
M=
\begin{pmatrix}
\frac 1 {20} & \frac 1 {10} \left(\arctan(2x) +\frac{\pi}{2} \right) \\
\frac 1 {10} \left(\arctan(2x) +\frac{\pi}{2} \right) & -\frac 1 {20}
\end{pmatrix}.
\]
The eigenvalues of $M$ are 
\[
\pm \sqrt{ \frac 1 {100} \left(\arctan(2x) +\frac{\pi}{2} \right)^2 + \frac 1 {400}}. 
\]
Hence, as $x\rightarrow \infty$, the eigenvalues of $M$,
$\lambda_{\pm}(x) \rightarrow \pm \frac{\pi}{10}$, and as
$x\rightarrow -\infty$, the eigenvalues of $M$,
$\lambda_{\pm}(x) \rightarrow \pm \frac{1}{20}$. As shown in Figure
\ref{fig:c2}, this model involves an extended region of strong
non-adiabatic coupling when $x<0$. Moreover, as $x>0$, the upper
energy surface is \red{increasing} so that trajectories moving from
left to right on the excited energy surface \red{without a large
  momentum} will be reflected while those on the ground energy surface
will be transmitted.

The energy gap between the two surfaces can be controlled by the
following $F^\delta$ function,
\[
F^\delta (x)=\frac{1}{\pi} \left( \arctan(100x)+\frac{\pi}2+\delta \right)
\]
We can check that $F^\delta = \mathcal O(\delta)$ when $x$ is sufficiently small. The family of energy surfaces are illustrated for different values of $\delta$ in Figure \ref{fig:c2}.
   \begin{figure}
\includegraphics[scale=0.5]{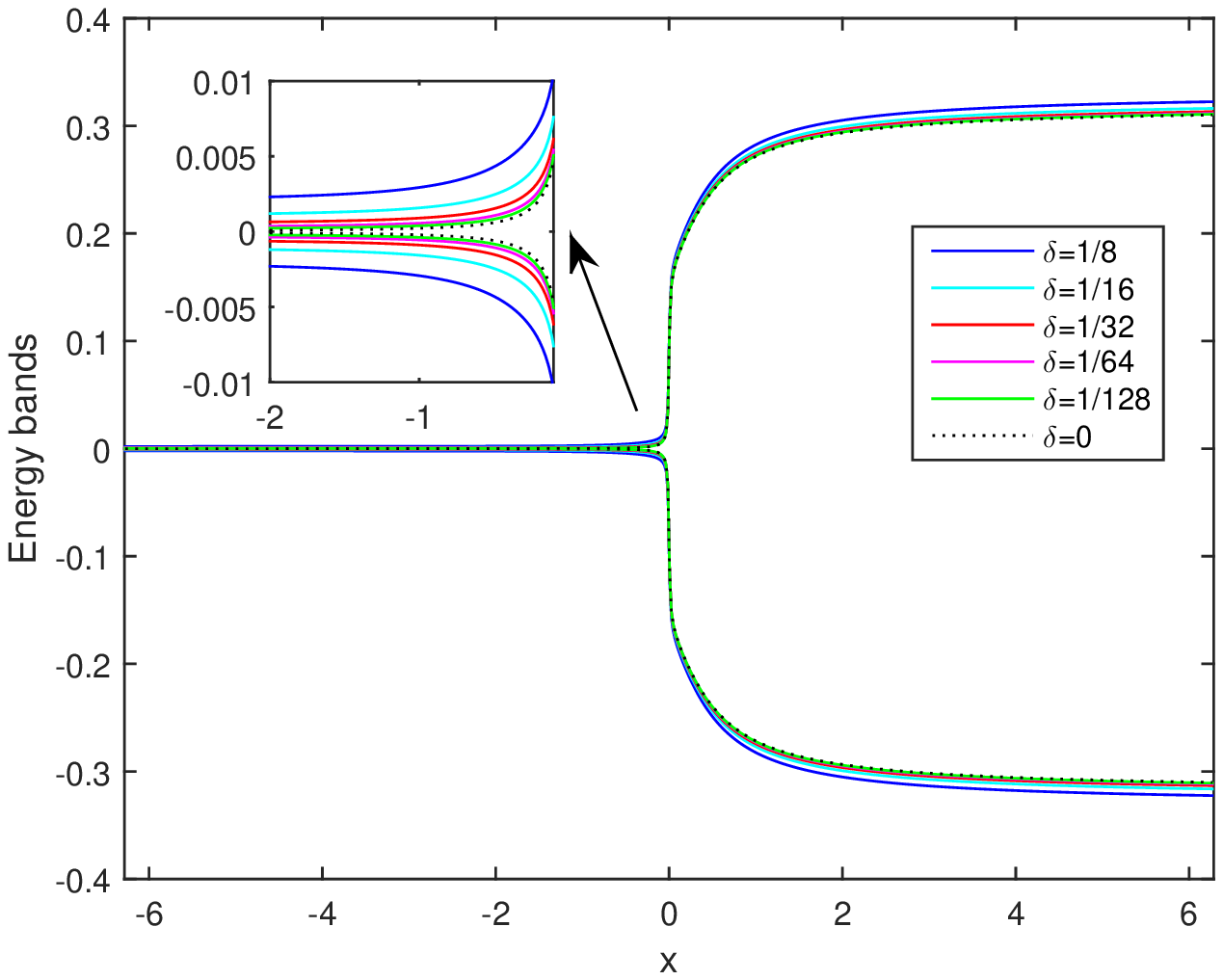}\includegraphics[scale=0.5]{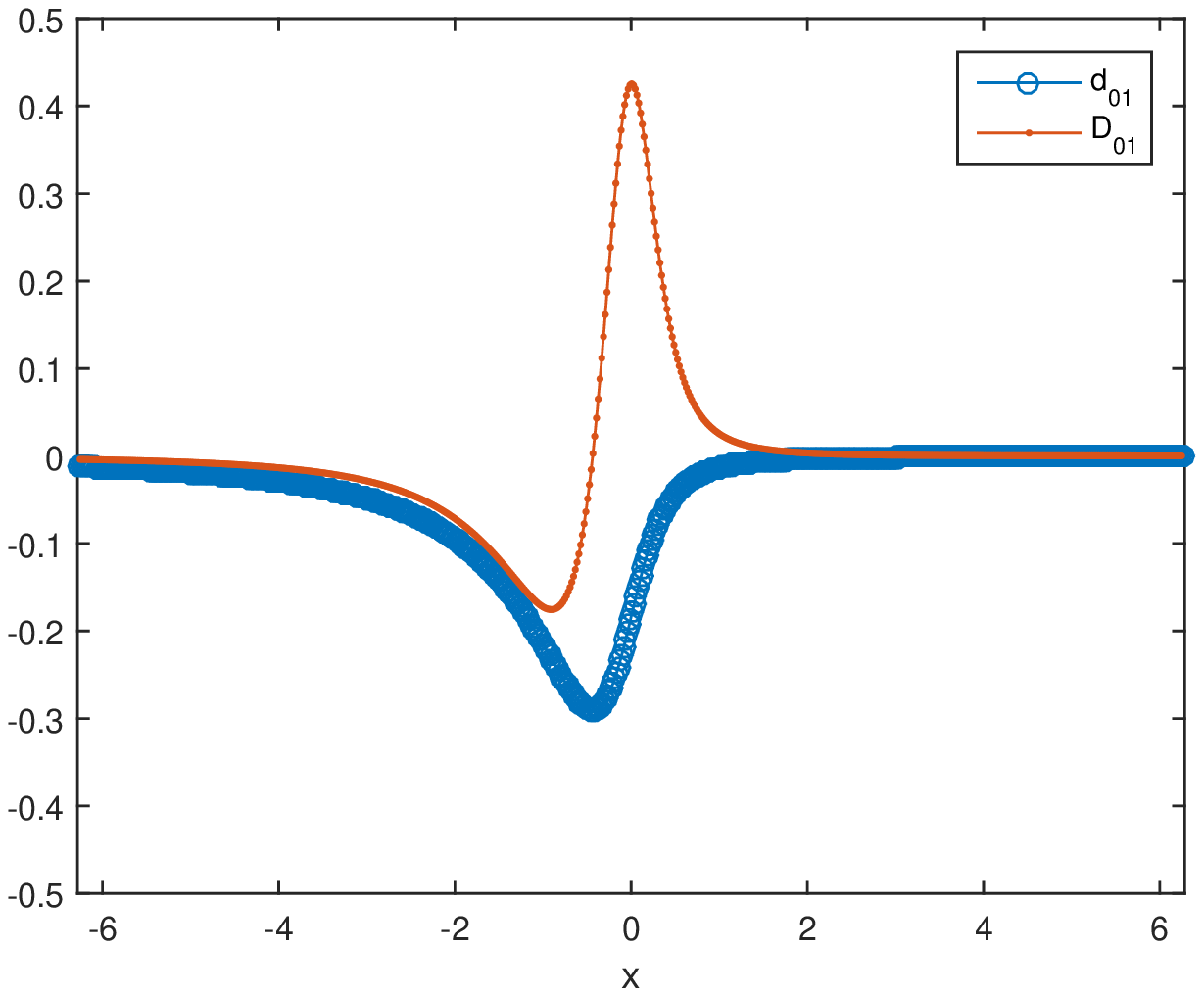} \\
\caption{(Example 1(c)) Left: Eigenvalues of $H_e$,  $\delta =\frac 1 8$,
  $\frac 1 {16}$, $\frac 1 {32}$, $\frac 1 {64}$ and $\frac{1}{128}$; reference $\delta=0$. Right:
  the coupling information of $H_e$, invariant with respect to $\delta$.}
\label{fig:c2}
\end{figure}

\medskip 

\noindent\textbf{Example 2.} 
Let us mention an example that does not satisfy our assumption in the limit $\delta \to 0$, which is in fact the classical conical intersection model with 
\begin{equation}
  H_e^{\delta}(x)  = 
  \begin{pmatrix} 
    x & \delta \\
    \delta & -x 
  \end{pmatrix}. 
\end{equation}
In fact, this is the model often analyzed for Landau-Zener transition.
For this family of Hamiltonians, we have
\[
\Abs{E_+^{\delta}(x)-E_-^{\delta}(x)}=2 \sqrt{x^2+\delta^2},
\]
and
\[
\nabla_x H_e^{\delta} =
\begin{pmatrix} 
1 & 0  \\
0 & -1 
\end{pmatrix}. 
\]
In this case, one can compute the analytical expression for $d_{+-}$ as
\[
d_{+-}^{\delta}(x)=-\frac{\delta}{2(x^2+\delta^2)}.
\] 
Clearly, at $x = 0$,
$d_{+-}^{\delta}=\mathcal{O}(\delta^{-1})$. Therefore, around $0$,
Assumption~\ref{assumb} is violated if $\delta$ goes to zero as
$\veps \to 0$ and our theorem no longer applies.

We remark however that for the practical examples with avoided
crossing, the small parameters (semiclassical parameter and energy
surface gap, etc.) are fixed, rather than converging to
$0$. Therefore, given a particular example, where $H_e^{\delta}$ is
specified for some small $\delta$, we can possibly embed the model
into a different sequence as $\veps \to 0$, so that our method can be
still used. Some numerical studies are presented in Example 4 for this scenario.

Nevertheless, the fact that the asymptotic derivation breaks down for this particular case raises the question that whether one can combine the fewest switch surface hopping type algorithms with the approaches based on Landau-Zener transition. This will be an interesting future research direction.

\section{Numerical examples}\label{sec:numerics}

In this section, we validate the algorithm based on the frozen Gaussian approximation with surface hopping and its probabilistic interpretation. The numerical examples are done for two-level matrix Schr\"odinger equations. 

\subsection{Description of the algorithm}

The algorithm based on the stochastic interpretation in
Section~\ref{sec:prob} is straightforward: We sample the initial point
of the trajectory based on the weight function
$\abs{A^{(0)}(0, z_0)}$; once the initial point is given, we evolve
the trajectory and associated FGA variables with surface hopping up to
some final time; and then we reconstruct the solution based on the
trajectory average \eqref{eq:trajavg}. The value of $A^{(0)}(0,z_0)$
will be calculated on a mesh of $(q, \, p)$ with numerical quadrature
of \eqref{eq:c00}.\footnote{\red{This is of course only possible for
    low-dimensional examples; approximation methods are needed for
    higher dimension calculation, which we do not address here. }} The
time evolution ODEs are integrated using the forth-order Runge-Kutta
scheme. After each time step, we calculate the hopping probability
during the time step $\Delta t |\tau|$ and generate a random number to
see if a hop occurs: If a hopping happens, we change the label of the
current surface and record the phase factor
$\frac{\tau}{|\tau|}$. After the trajectory is determined up to time
$t$, we can calculate the weighting factors in \eqref{eq:trajavg} by
again a numerical quadrature. Our code is implemented in
\textsf{Matlab}.

Note that the algorithm above is the most straightforward Monte Carlo
algorithm for evaluating the average of trajectories
\eqref{eq:trajavg}. With the path integral representation, it is
possible to design more sophisticated algorithms trying to further
reduce the variance. This will be considered in future works.

Comparing the numerical solution with the exact solutions to the Schr\"odinger equations, we have several sources of error, listed below:
\begin{enumerate}
\item[a.] Initial error. This is the error coming from numerical quadrature of $A^{(0)}(0, z_0)$, the mesh approximation in the phase space, and also due to the choice of a compact domain $K$ in the phase space; 
\item[b.] Asymptotic error. This is the $\mathcal O (\veps)$ error coming from the higher order term we neglected in the derivation of the frozen Gaussian approximation with surface hopping ansatz; 
\item[c.] Sampling error. Since the algorithm is a Monte Carlo
  algorithm to compute the average of trajectories \eqref{eq:trajavg}, for
  finite sample size, we will have statistical error compared to the
  mean value. Since this error is due to the variance of the sampling,
  it decays as $1/\sqrt{N_{\text{traj}}}$ where $N_{\text{traj}}$ is the total number of 
  trajectories. This is confirmed in Figure~\ref{fig:sqrtN} and Table~\ref{sqrtNdata} for a
  fixed (and somewhat large) $\veps = \frac{1}{16}$.  \red{In Table~\ref{sqrtNdata}, convergence rates for tests with different number of trajectories are computed by
\[
\text{Conv. Rate} := \log_{N^{b}_{\text{traj}}/N^a_{\text{traj}}} \frac{\mathbb E(e^a)}{\mathbb E(e^b)}.
\]
}
  \begin{figure}
    \includegraphics[scale=0.5]{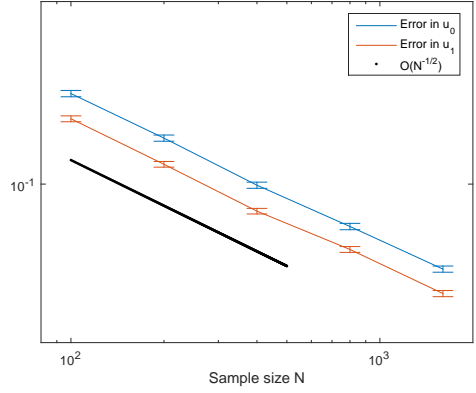}\\
    \caption{For $\veps=\frac{1}{16}$ and various numbers of 
  trajectories $N_{\text{traj}}$, the empirical averages of the total numerical error with 95\% confidence intervals.}
    \label{fig:sqrtN}
  \end{figure}
  \begin{table}
  \begin{tabular}{ c |c| c| c|c |c}
    \hline
    $\veps=\frac{1}{16}$ & $N_{\text{traj}}=100$  & $N_{\text{traj}}=200$ & $N_{\text{traj}}=400$ &  $N_{\text{traj}}=800$ & $N_{\text{traj}}=1600$  \\  \hline
    $\mathbb E(e_0)$ & 1.9889e-01&   1.4182e-01&   9.9173e-02&   7.2472e-02& 5.2443e-02
 \\  \hline
\red{ Conv. Rate} & &\red{0.4879 }  &\red{0.5160} &   \red{0.4525} &   \red{0.4667} \\  \hline
    $\text{Var}(e_0)$& 2.3702e-03 &  1.1585e-03 &  5.6254e-04 &   3.1891e-04  &1.5769e-04 \\  \hline
    $\mathbb E(e_1)$ & 1.6423e-01 &  1.1624e-01 &  8.1430e-02 &   6.0873e-02 & 4.3546e-02  \\  \hline
\red{Conv. Rate }& & \red{ 0.4987}  &  \red{0.5135} &   \red{0.4198} &   \red{0.4833} \\  \hline
    $\text{Var}(e_1)$&1.4000e-03 &   6.5504e-04 &   3.1067e-04 &  1.8907e-04 & 1.0734e-04 \\  
    \hline
  \end{tabular}  
  \medskip
  \caption{For $\veps=\frac{1}{16}$ and various numbers of 
  trajectories $N_{\text{traj}}$, the empirical averages and sample variance of the total numerical error based on $400$ implementations for each test.}
  \label{sqrtNdata}
\end{table}
\item[d.] Quadrature error. In solving the evolution of FGA variables
  and make phase changes at hoppings, the ODE solvers will introduce
  numerical error. Note that, some high order solvers (e.g., RK4 here)
  are preferred for the FGA variables because in the phase function,
  the numerical error is magnified by $\mathcal O(1/\veps)$.
\end{enumerate}

In the numerical tests, we use the following initial sampling
strategy. We first choose a partition integer $M \in \N^+$, and the
corresponding partition constant is defined as
\[
d_M = \max_{ (q, p) \in K} \frac{\bigl\lvert A_0^{(0)}\left(0,q,\,p\right) \bigr\rvert}{M}.
\]
For a specific grid point $(q, p)$, we generate $n_{(q, p)}$
independent trajectories starting with the initial point $(q, p)$ where
\[
n_{(q, p)}=\left \lceil \frac{\bigl\lvert A_0^{(0)}\left(0,q,\,p \right) \bigr\rvert}{d_M} \right \rceil.
\]
For each trajectory initiated from this grid point, the initial weight
$A_0^{(0)}(0, q, p)$ is equally divided. As the partition integer $M$
increases, the number of the trajectories increases, and thus the
numerical error reduces.

\subsection{Numerical tests}

All the test problems we consider in this paper are $1D$ two-state
matrix Schr\"odinger equation with the electronic Hamiltonian $H_e(x)$
assumed to be a $2 \times 2$ matrix potential. We compare with the results from our surface hopping algorithm with the numerical reference solution from the time splitting spectral method (TSSP) (see e.g., \cites{TS,reviewsemiclassical,TSSL}) with sufficiently fine mesh.

\noindent {\bf Example 3.} In this example, we take the electronic Hamiltonian $H_e(x)$ to be same as {Example 1(a)} in Section \ref{sec:example} with $\delta=\veps$, which we recall here for convenience,
\begin{equation}
H_e(x)=\left(1+(\veps-1)e^{-10x^2} \right)\begin{pmatrix}
    \frac{\tanh (x)}{2 \pi} & \frac {1} {10} \\
    \frac {1} {10} & -\frac{\tanh (x)}{2 \pi}
  \end{pmatrix}.
\end{equation}
As as shown in Figure \ref{fig:a2}, the coupling vectors are not
negligible when $-1<x<1$, and hence hopping might occur. 

We choose the initial condition to the two-state Schr\"odinger equation as
\[
 u(0,r,x)=u_0(0,x)\Psi_0(r;x)=\red{(16 \veps)^{-1/4}} \exp\left(\frac{i2x}{\veps}\right) \exp \left(- 16(x-1)^2 \right)\Psi_0(r;x).
\]
this initial condition corresponds to a wave packet in the ground
state energy surface localized at $q=-1$ traveling to the right with
speed $p=2$.

Typical FGA trajectories with surface hopping are plotted in Figure~\ref{fig:traj} for $\veps=\frac{1}{16}$ and $\veps=\frac{1}{128}$. \begin{figure}
\includegraphics[scale=0.5]{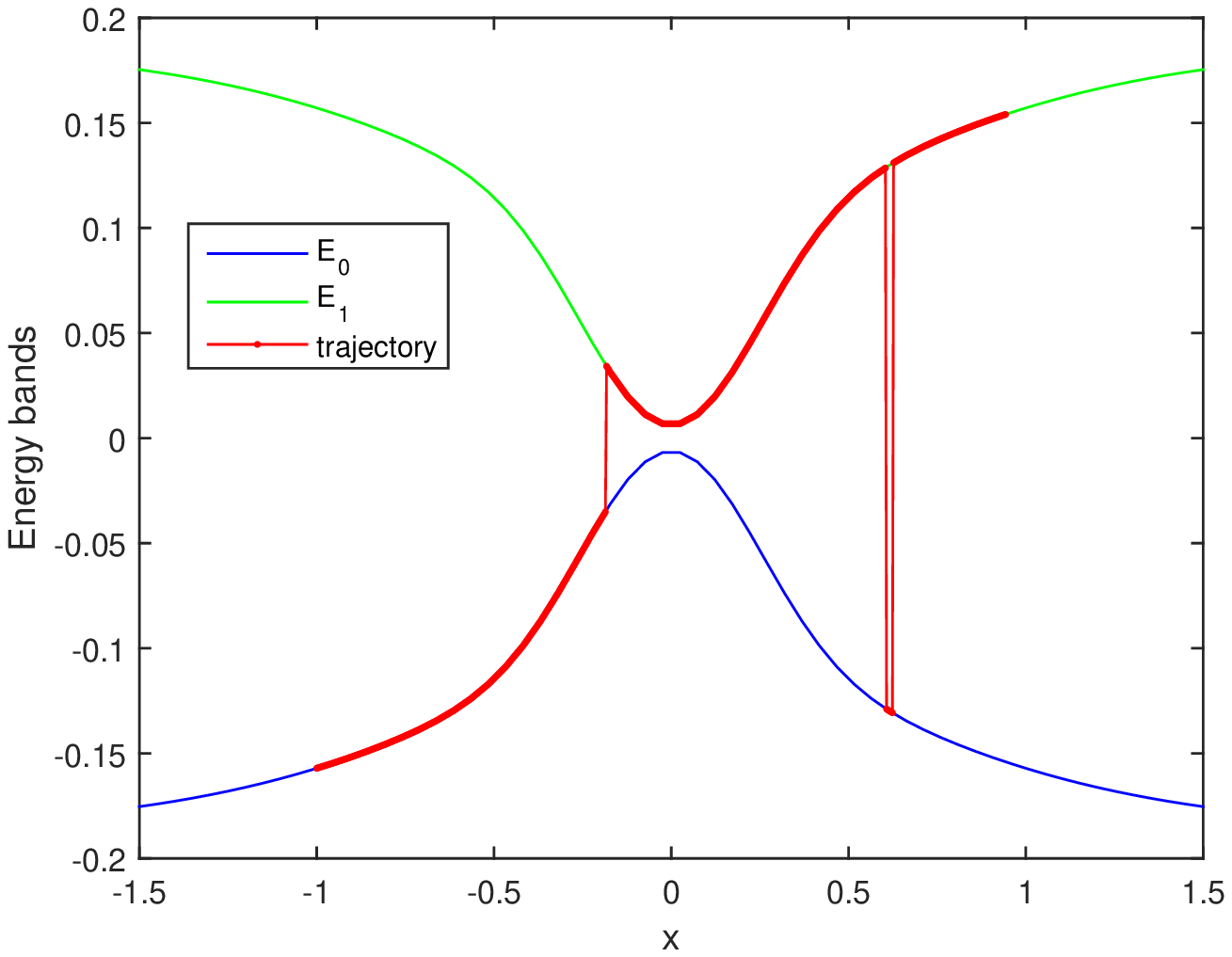} \includegraphics[scale=0.5]{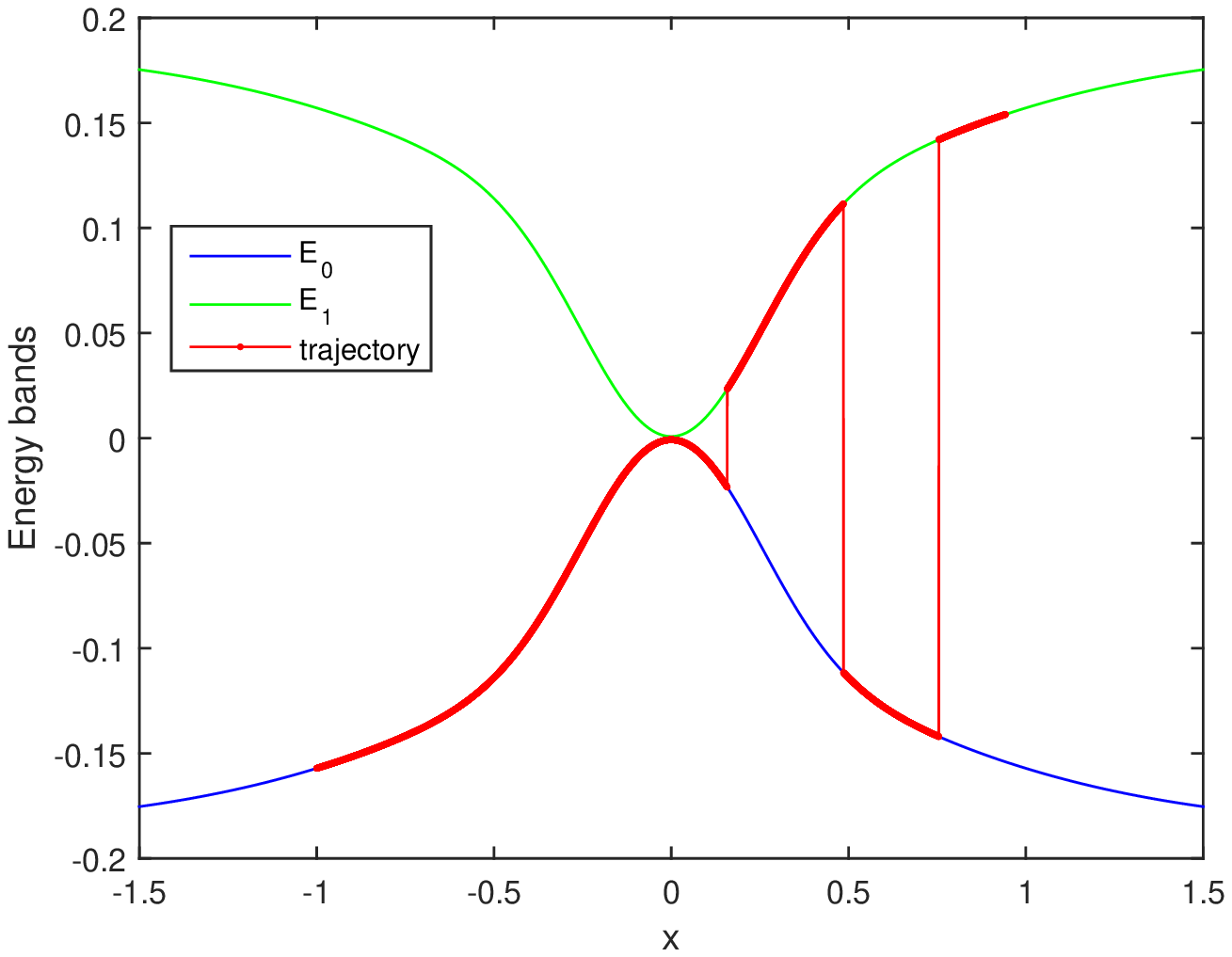}\\
\caption{Typical trajectories in the FGA algorithm. Left: $\veps=\frac{1}{16}$. Right: $\veps=\frac{1}{128}$.}
\label{fig:traj}
\end{figure}

For $\veps=\frac{1}{16}$, $\frac{1}{32}$ and $\frac{1}{64}$, and for
the partition integer $M=1$, $2$, $4$, $8$, $16$, we use the FGA
algorithm to compute the $u_0(t,x)$ and $u_1(t,x)$ till $t=1$. At
$t=1$, the wave packet has traveled across the hopping zone. We choose
the computation domain for $y$ and for $x$ to be $[-\pi,\pi]$, and the
computation domain for $(q,p)$ to be $[-\pi,\pi]\times [0.5,3.5]$. We
choose the following mesh sizes in the FGA method for initial sampling and for reconstructing the solution.
\begin{equation}\label{ex1mesh}
\Delta q= \frac{2 \pi \veps}{8}, \quad \Delta p = \frac{3 \veps}{4}, \quad \Delta x=\Delta y = \frac{2 \pi \veps}{32}.
\end{equation}
With this mesh, the initial error $e_0^\veps$ are summarized in Table
\ref{ex1init}, which is made significantly smaller than the other
parts of the error from the approximation.  Also, we choose the step
size to be very small $\Delta t =\frac{\eps}{32}$ and apply the forth
order Runge-Kutta method to solve the FGA variables. Hence, the total
error is dominated by the asymptotic error and the sampling error in
these tests. The reference solution is computed on $[-\pi,\pi]$ by a second order (in time) TSSP method with sufficiently fine mesh
\[
\Delta x = \frac{2 \pi \veps}{ 64},\quad \Delta t=  \frac{\veps}{32}.
\]

\begin{table}
\scalebox{1.0}{
\begin{tabular}{ l |c| c| c}
\hline
  $\veps$ & $\frac{1}{16}$  & $\frac{1}{32}$ & $\frac{1}{64}$ \\  \hline
$e_0^\veps$ & 8.3178e-05 &  \red{1.4173e-07 }&  \red{1.1697e-07 }\\
   \hline
\end{tabular}
}
\caption{Initial error for  $\veps=\frac{1}{16}$, $\frac{1}{32}$ and $\frac{1}{64}$ with mesh given by \eqref{ex1mesh}.}
\label{ex1init}
\end{table}

To quantify the sampling error, we repeat each test for $400$ times
and estimate the empirical average $\mathbb E(e_k)$ and its variance
and $\text{Var}(e_k)$ are summarized in Table \ref{ex1data}, where
$e_k$ denotes the $L^2$ error of the $k$-th component of the solution,
\red{and convergence rates for different $M$ are
  estimated by
\[
\text{Conv. Rate} := \log_{M^b/M^a} \frac{\mathbb E(e^a)}{\mathbb E(e^b)}.
\]}
The errors with their $95\%$ confidence intervals are
plotted in Figure \ref{fig:ex1_2}.  From the numerical results, we see
clearly that increasing the partition integer $M$ can effectively
reduce the numerical error. 
  
\begin{table}
  \begin{tabular}{ c |c| c| c|c|c}
    \hline
    $\veps=\frac{1}{16}$ & $M=1$  & $M=2$ & $M=4$ &  $M=8$ & $M=16$ \\  \hline
    $\mathbb E(e_0)$ & 6.9485e-02 &5.5660e-02 &4.4370e-02 & 3.3320e-02 & 2.6339e-02 \\  \hline
\red{Conv. Rate }& & \red{0.3201}  &  \red{0.3271} &   \red{0.4132}  &  \red{0.3391}\\  \hline
    $\text{Var}(e_0)$& 4.9452e-04 & 2.3494e-04 & 1.5885e-04 &  8.7959e-05 &   5.5006e-05  \\  \hline
    $\mathbb E(e_1)$ & 5.7770e-02&   4.6380e-02&   3.8663e-02&   3.0728e-02&   2.5886e-02 \\  \hline
\red{Conv. Rate} & &  \red{0.3168}  &  \red{0.2626} &   \red{0.3314} &   \red{0.2474}\\  \hline
    $\text{Var}(e_1)$& 2.7984e-04&   1.3865e-04&   1.0308e-04&   6.0901e-05&   4.8058e-05 \\  
    \hline
  \end{tabular}
  \medskip

  \red{
  \begin{tabular}{ c |c| c| c|c|c}
    \hline
    $\veps=\frac{1}{32}$ & $M=1$  & $M=2$ & $M=4$ &  $M=8$ & $M=16$ \\  \hline
    $\mathbb E(e_0)$ & 5.6999e-02&   4.6254e-02&  3.5705e-02&   2.6951e-02&   2.0107e-02\\  \hline
Conv. Rate & & 0.3014  &  0.3735 &   0.4058  &  0.4227\\  \hline
    $\text{Var}(e_0)$&    2.8623e-04&   1.5798e-04&   1.0331e-04&   5.2099e-05&   2.8266e-05  \\  \hline
    $\mathbb E(e_1)$ & 5.2574e-02& 4.2267e-02&  3.3255e-02&   2.5514e-02&   2.0454e-02 \\  \hline
Conv. Rate & &0.3149  &  0.3460 &   0.3823 &   0.3189 \\  \hline
    $\text{Var}(e_1)$& 1.8928e-04&      1.2575e-04&   6.7133e-05&   4.0982e-05&   2.6102e-05 \\  
    \hline
  \end{tabular}
}

  \medskip

  \red{
  \begin{tabular}{ c |c| c| c|c|c}
    \hline
    $\veps=\frac{1}{64}$ & $M=1$  & $M=2$ & $M=4$ &  $M=8$ & $M=16$ \\  \hline
    $\mathbb E(e_0)$ & 4.8308e-02&   3.8534e-02&    2.9118e-02&   2.2112e-02&   1.6811e-02\\  \hline
Conv. Rate & &  0.3261  &  0.4042 &   0.3970 &   0.3955 \\  \hline
    $\text{Var}(e_0)$& 1.9927e-04&   9.3424e-05&   5.0794e-05&   3.1704e-05&   1.7614e-05  \\  \hline
    $\mathbb E(e_1)$ &4.6589e-02&   3.6833e-02&   2.8203e-02&   2.1660e-02&   1.6785e-02 \\  \hline
Conv. Rate & & 0.3390  &  0.3852  &  0.3808   & 0.3678 \\  \hline
    $\text{Var}(e_1)$& 1.5864e-04&   8.1220e-05&   5.4446e-05&   2.6814e-05&   1.5051e-05 \\  
    \hline
  \end{tabular}
  \medskip
}
  
  \caption{\red{(Example 3)} For various $\veps$ and partition integers $M$, the empirical averages and sample variance of the total numerical error based on $400$ implementations for each test.}
  \label{ex1data}
\end{table}

  \begin{figure}
\includegraphics[scale=0.55]{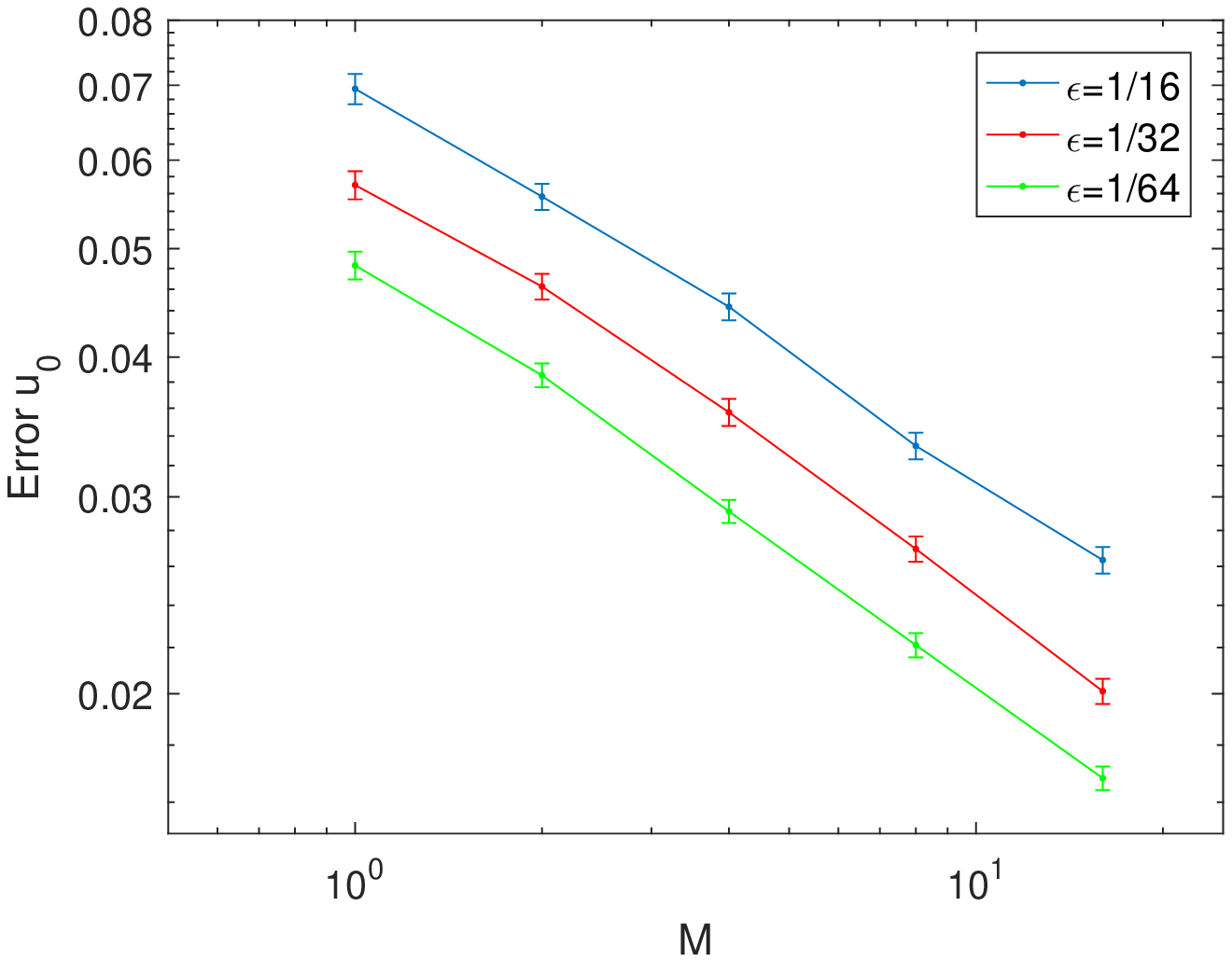} \includegraphics[scale=0.55]{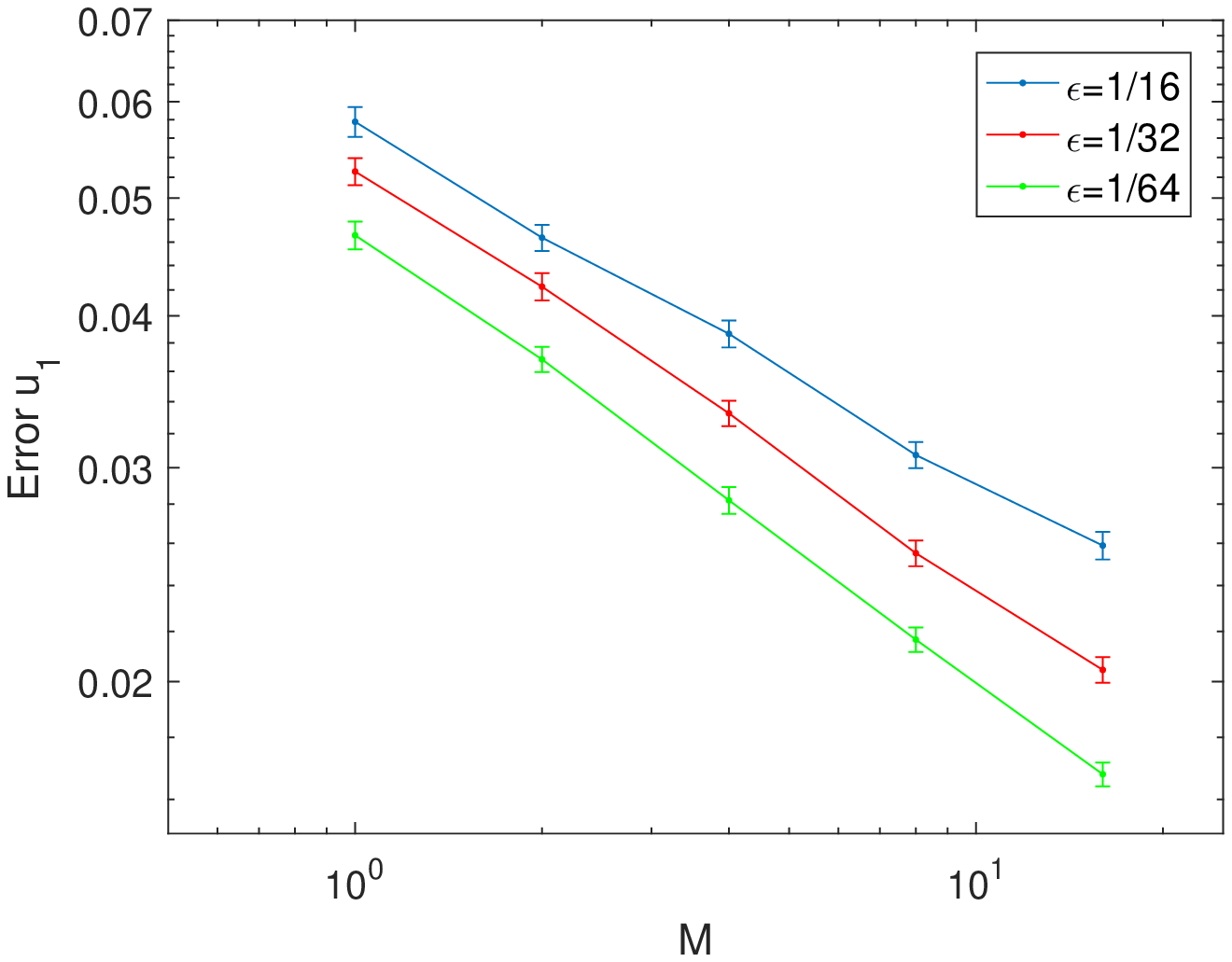}\\
\caption{\red{(Example 3)} For various $\veps$ and partition integers $M$, the empirical averages of the total numerical error with 95\% confidence intervals.}
\label{fig:ex1_2}
\end{figure}

\red{ Finally, we aim to demonstrate the application of the
  FGA-SH method in calculating the transition rate versus time.  For
  $\veps=\frac{1}{16}$ and $\frac{1}{128}$, we carry out the test with
  $N_{\text{traj}}=6400$ trajectories, and calculate the transition
  rates at different times till $t=1.5$. The results are plotted in
  Figure \ref{fig:ex1_3}, from which we observe very nice agreements
  with the reference calculations.}

  \begin{figure}
\includegraphics[scale=0.55]{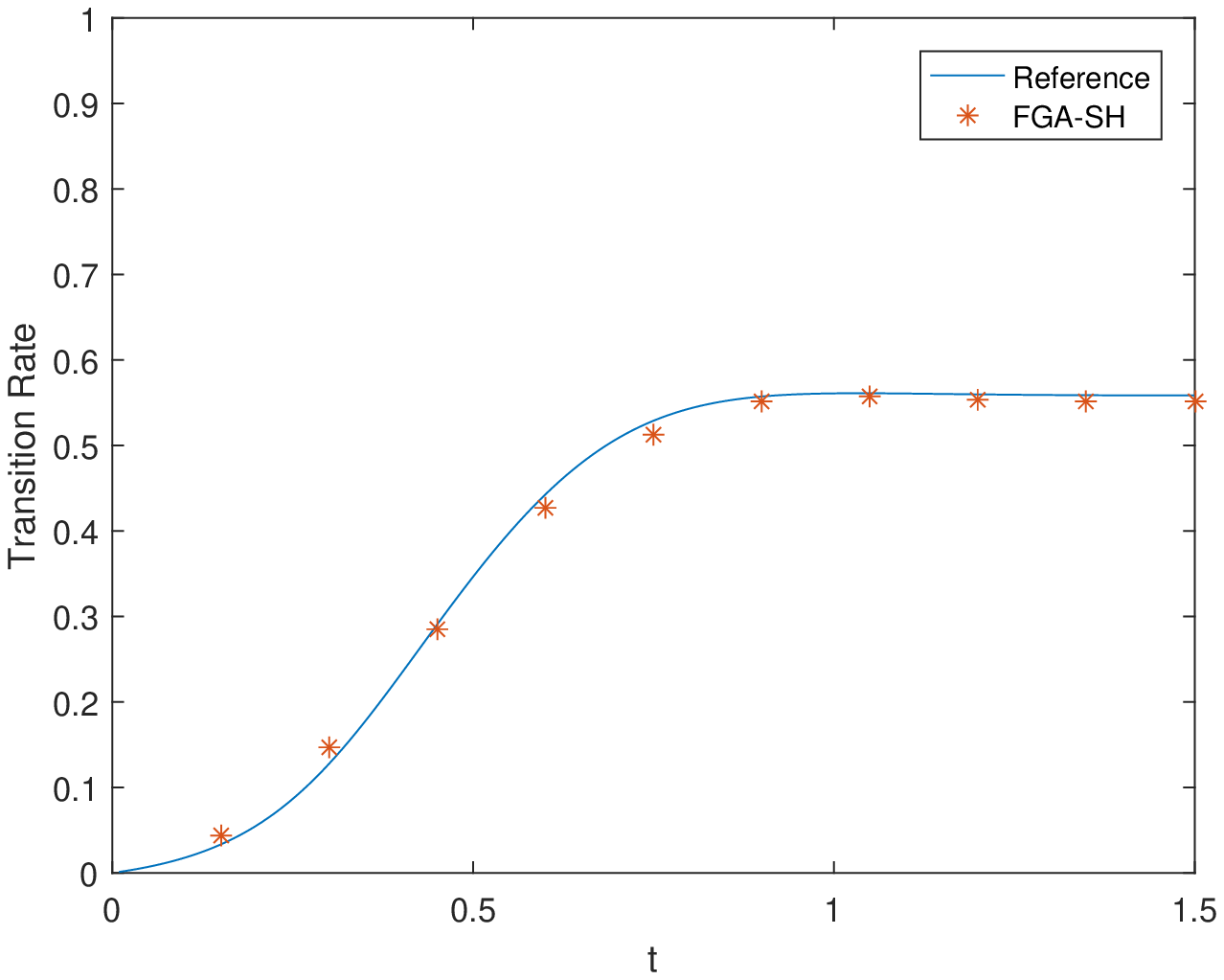} \includegraphics[scale=0.55]{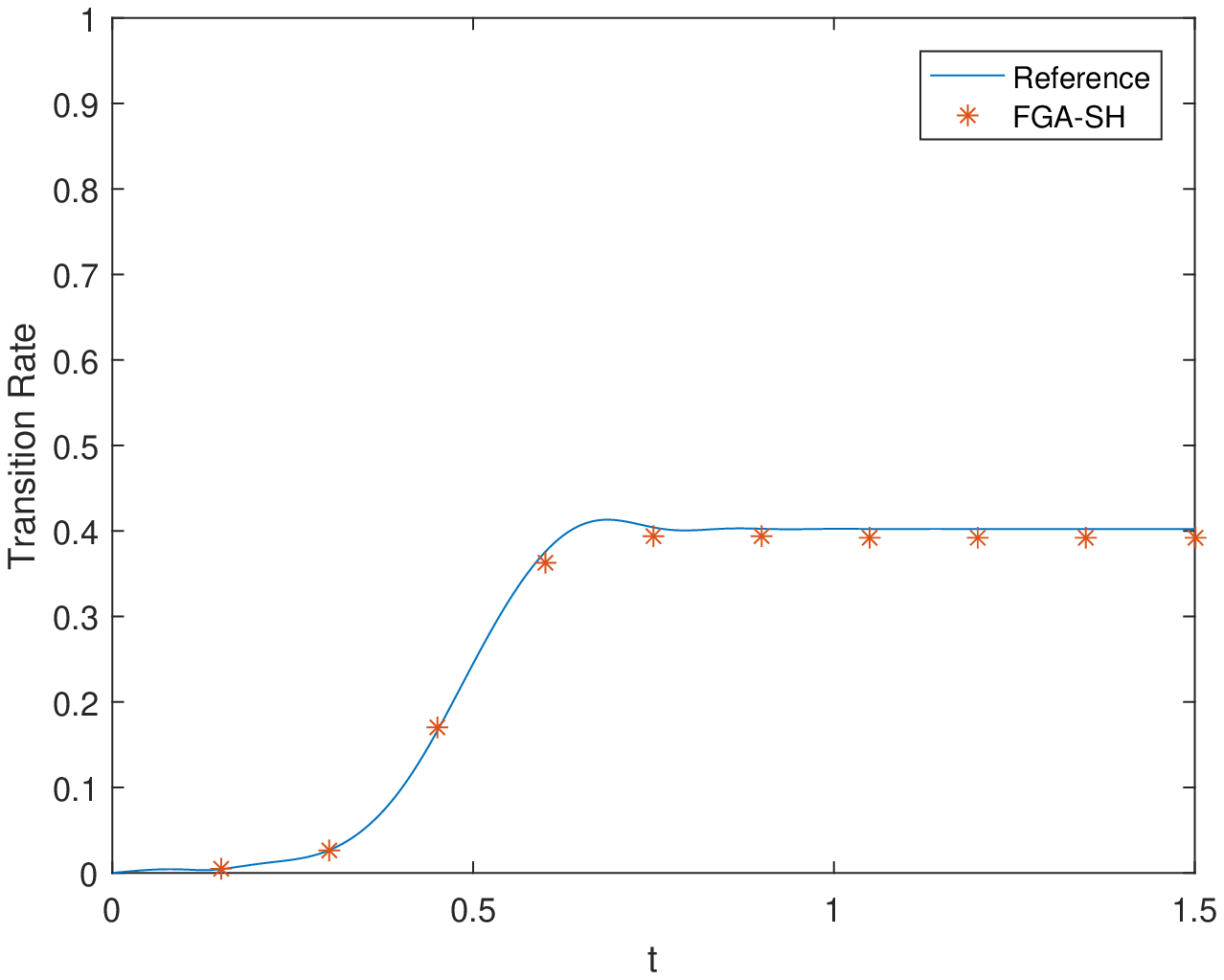}\\
\caption{\red{(Example 3) For various $\veps$, the typical behavior of the FGA-SH
  method in calculating the transition rate versus time. Left:
  $\veps= \frac{1}{16}$. Right: $\veps=
  \frac{1}{128}$. }}
\label{fig:ex1_3}
\end{figure}

\noindent {\bf Example 4.} In this example, the electronic Hamiltonian $H_e(x)$ is given by
\begin{equation}
H_e(x)=\begin{pmatrix}
    \frac{x}{5} & \frac {1} {10} \\
    \frac {1} {10} & -\frac{x}{5}
  \end{pmatrix}.
\end{equation}
This matrix potential is similar to the Example 2 in
Section~\ref{sec:example} except that the we have fixed a small
$\delta$ as $\veps$
varies. 
Hence, as $\veps$ goes to $0$, the energy surfaces of the electronic
Hamiltonian stay unchanged. Thus, the FGA method applies to this
case. We plot the energy surfaces, $d_{01}$ and $D_{01}$ of this matrix
potential in Figure~\ref{fig:ex4}, from which we observe that the
coupling vector is not negligible around $x=0$.

\begin{figure}
\includegraphics[scale=0.5]{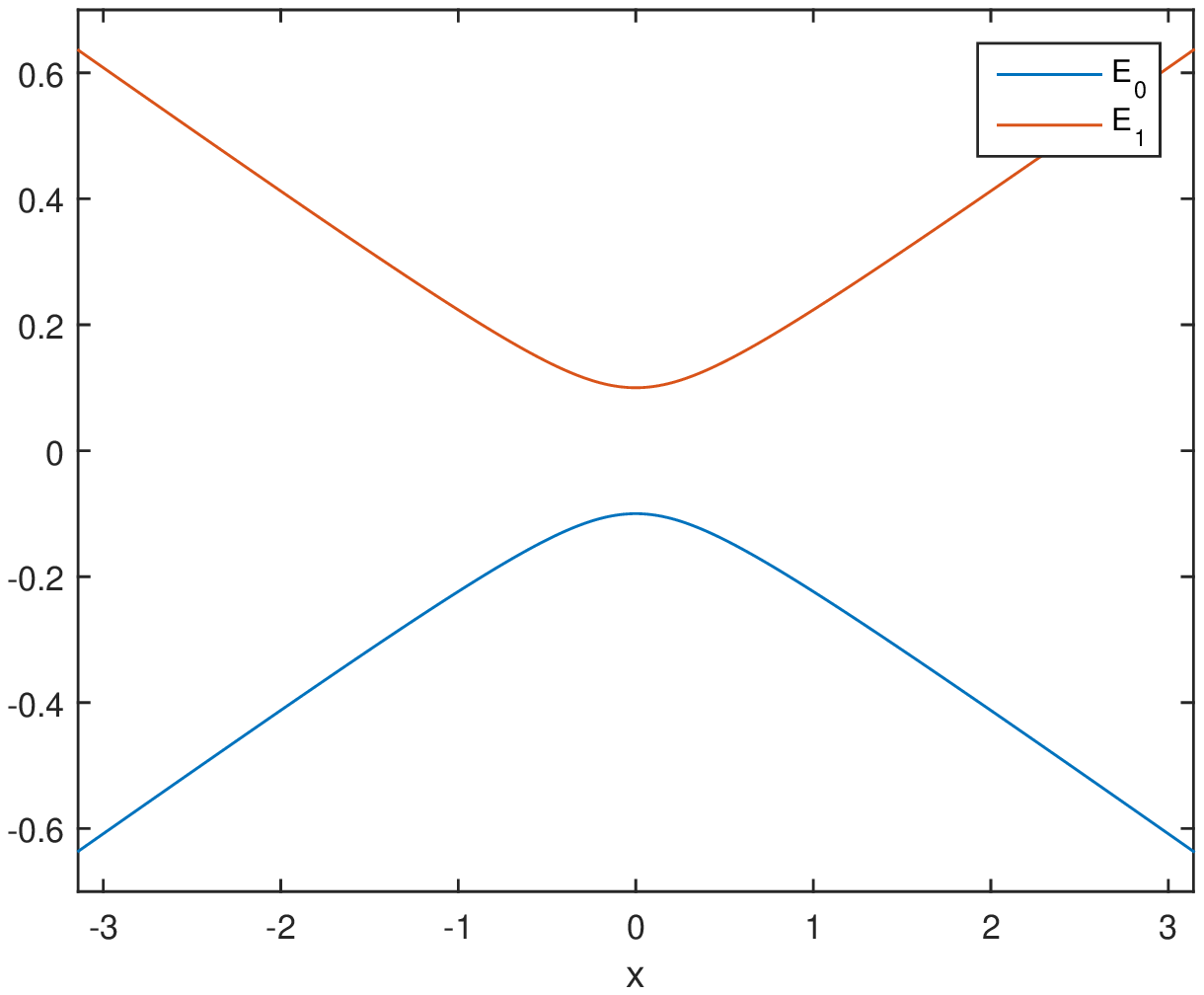}\includegraphics[scale=0.5]{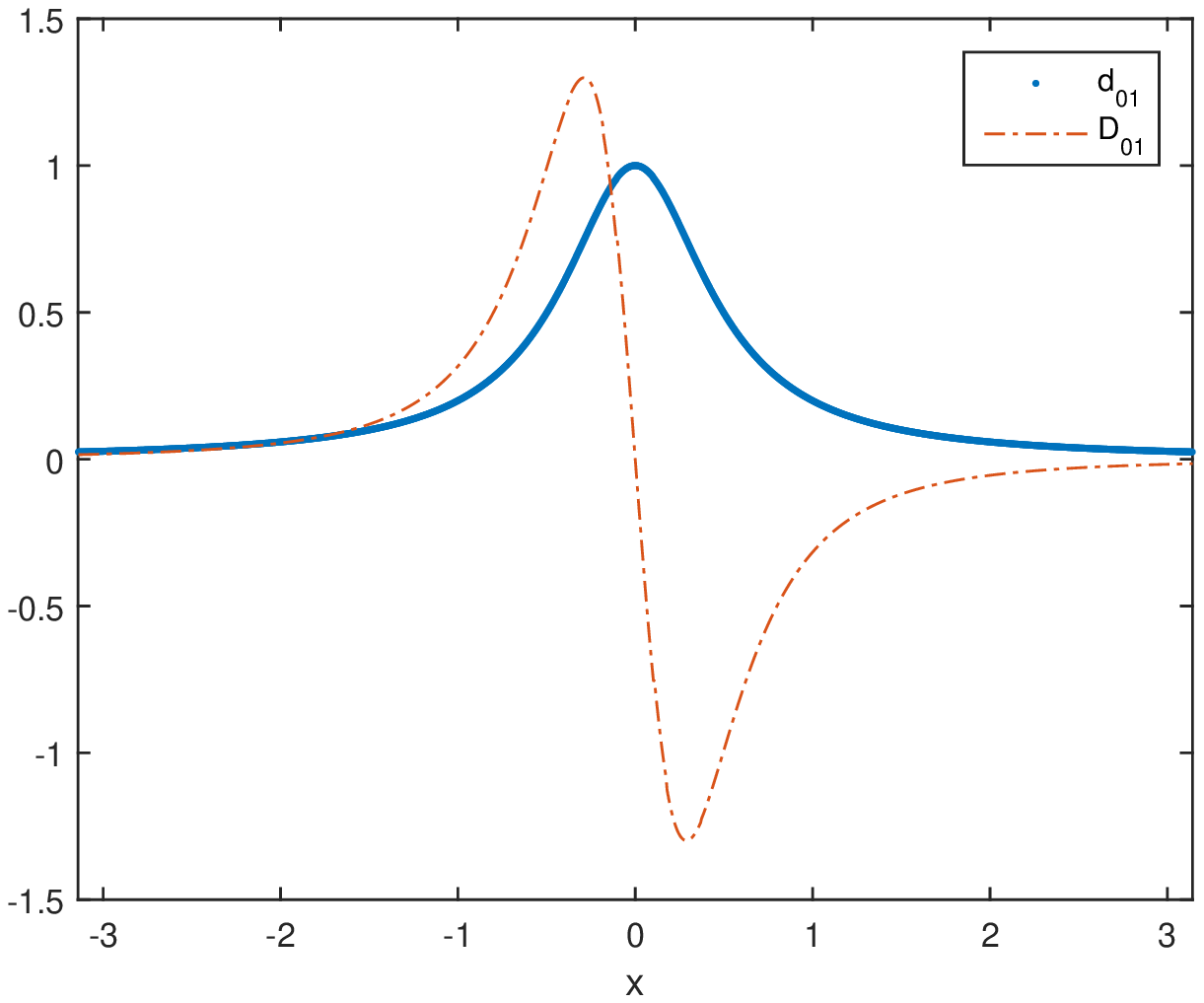} \\
\caption{(Example 4) Left: Eigenvalues of $H_e$. Right:
  the coupling information of $H_e$.}
\label{fig:ex4}
\end{figure}

We choose the same initial condition to the two-state Schr\"odinger equation
\[
u(0,r,x)=u_0(0,x)\Psi_0(r;x)=\exp\left(\frac{i2x}{\veps}\right) \exp \left(- 16(x-1)^2 \right)\Psi_0(r;x),
\]
and the same computation domain and meshing sizes for the FGA algorithm and reference solver as Example 3. For $\veps=\frac{1}{16}$ and $\frac{1}{128}$, and for the partition integer $16$, we use the FGA algorithm to compute the $u_0(t,x)$ and $u_1(t,x)$ till $t=1$. 

The reference solution indicates that when $\veps=\frac{1}{16}$, the
transition portion is significant but when $\veps=\frac{1}{128}$ the
transition between the surfaces is practically small. This is expected as the
 gap is finite and fixed, while $\veps \rightarrow 0$, so that the
non-adiabatic transition is approaching $0$ (see \textit{e.g.},
\cites{BO1,BO2}).

Whereas, in the FGA algorithm, the hopping rate is only related to the
coupling vectors and the momentum along the FGA trajectory. Therefore,
for $\veps=\frac{1}{16}$ and $\frac{1}{128}$, the hopping
probabilities are similar along the FGA trajectory, but when
$\veps=\frac{1}{128}$, the hopped trajectories on the exited state
should average to $0$, which is verified by the numerical results
plotted in Figure~\ref{fig:ex2_1} together with the reference
solution. Besides, we show by comparing the reference solutions in Figure~\ref{fig:ex2_1} that the weighting factors in \eqref{eq:trajavg} are crucial in reconstructing the correct wave functions.

\begin{figure}
\includegraphics[scale=0.5]{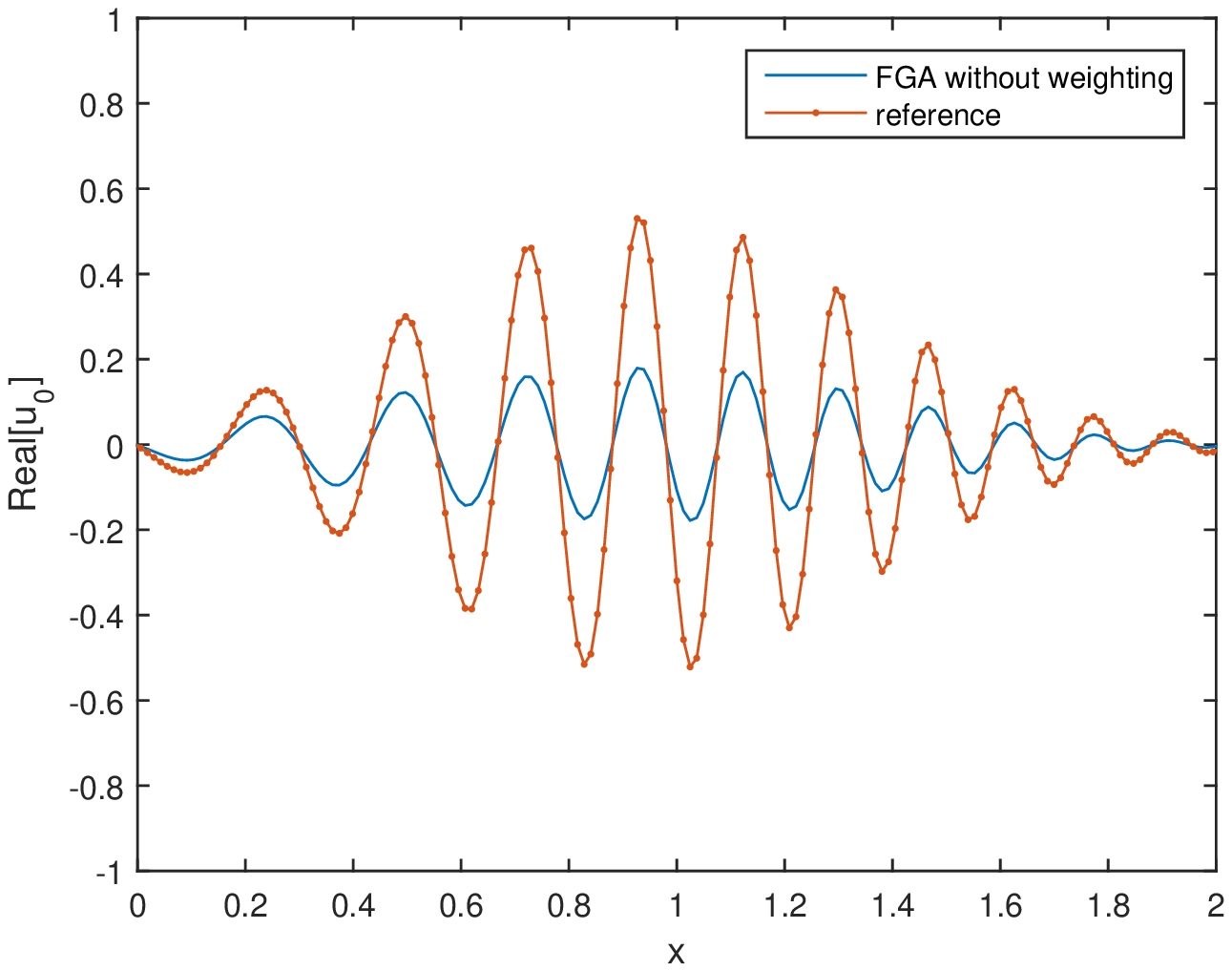} \includegraphics[scale=0.5]{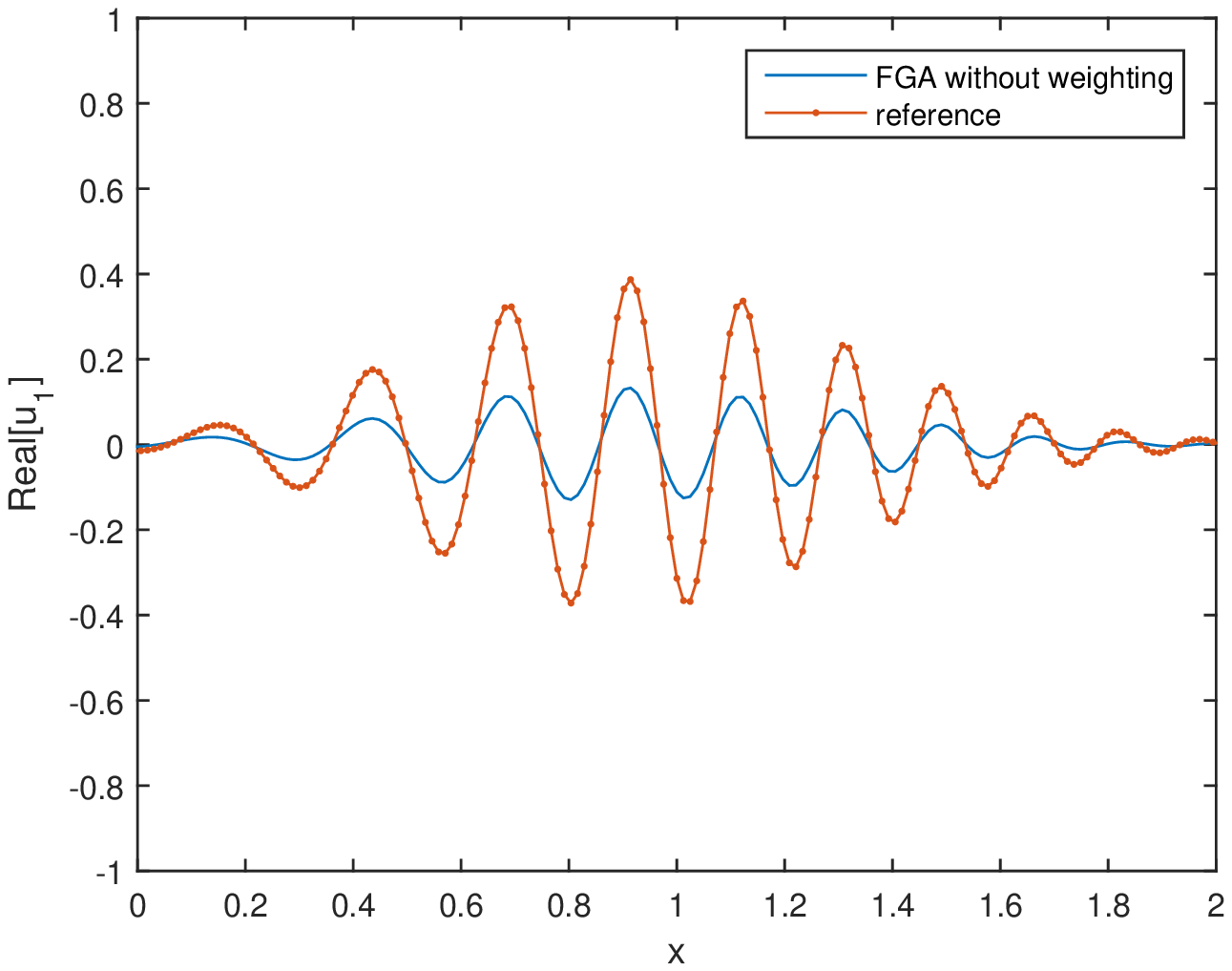}\\
\includegraphics[scale=0.5]{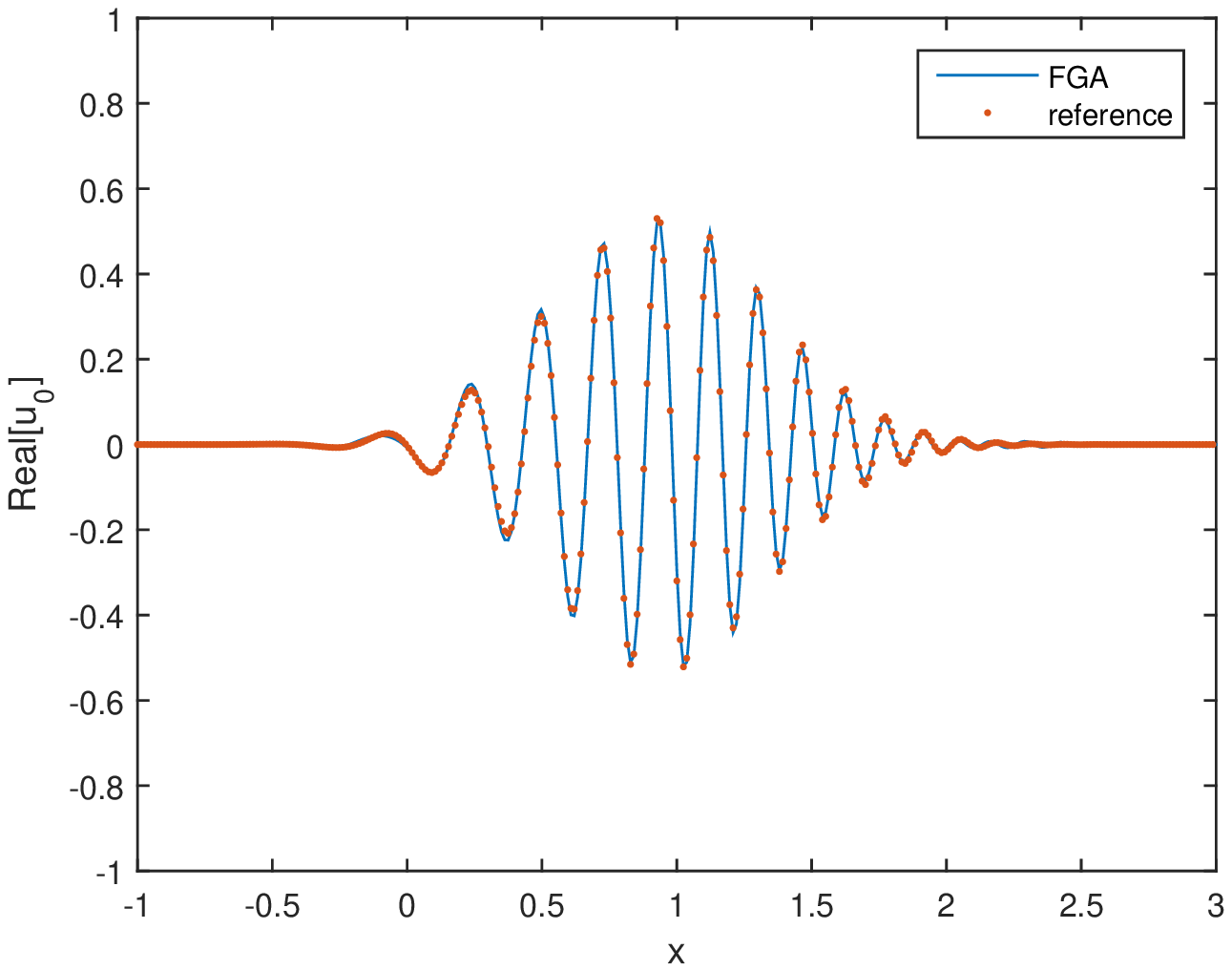} \includegraphics[scale=0.5]{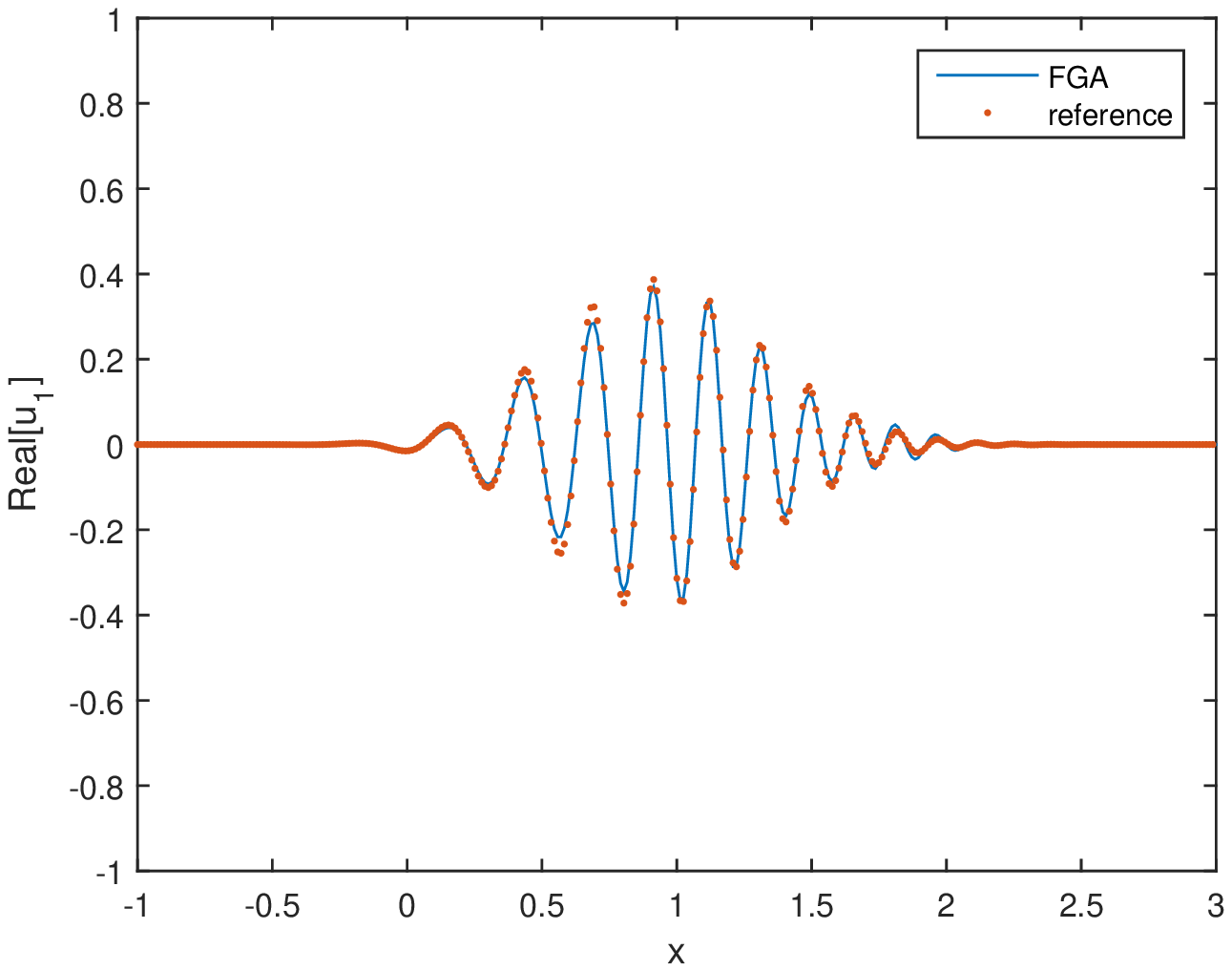}\\
\includegraphics[scale=0.5]{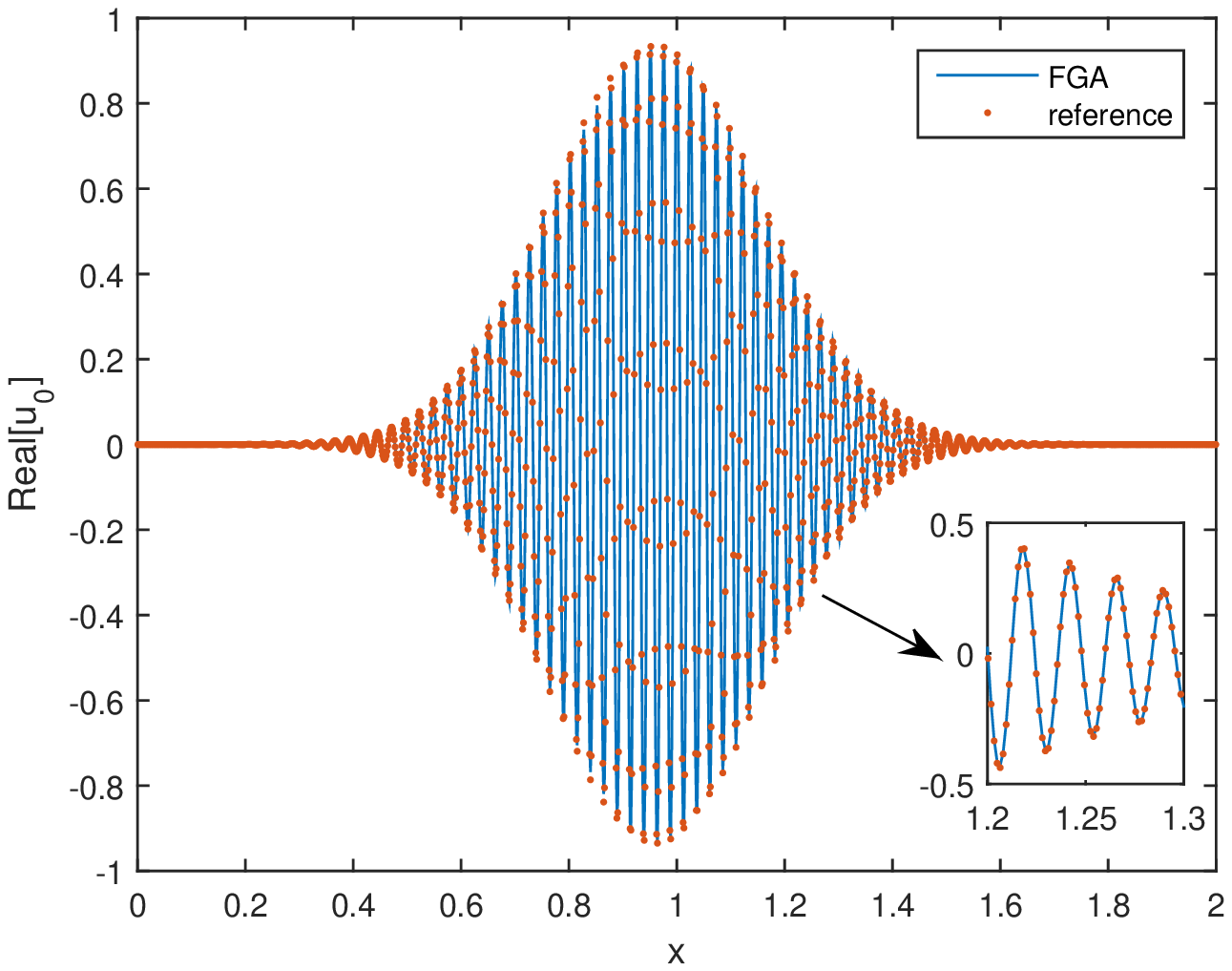} \includegraphics[scale=0.5]{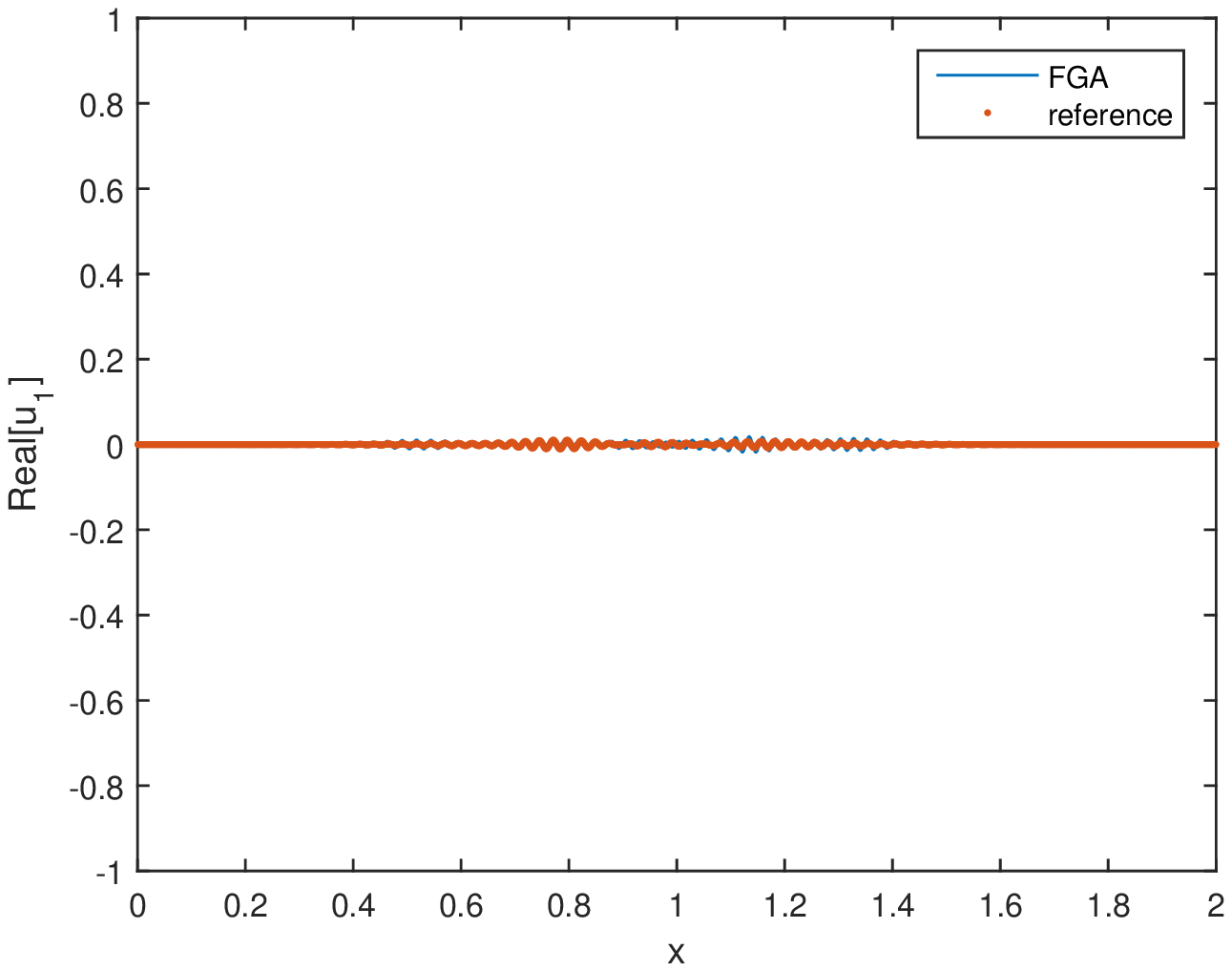}\\
\caption{\red{(Example 4)} Comparison between the FGA algorithm and the reference solutions. Top: $\veps=\frac{1}{16}$, the FGA method is implemeted without the weighting factors.  Middle: $\veps=\frac{1}{16}$. Bottom: $\veps=\frac{1}{128}$, zoomed-in plots included.}
\label{fig:ex2_1}
\end{figure}


\section{Convergence proof} \label{sec:proof}

We now prove that the ansatz is a good approximation to the true
solution of the matrix Schr\"odinger equation.  For simplicity of
notations, we omit the appearance of $\delta$ in the surface energy
and coupling vectors. By Assumption \ref{assumb}, the boundedness of
related quantities in the following analysis is uniform with respect
to $\delta$, and hence all the estimates below are independent of
$\delta$. In \S\ref{sec:prelim}, we study the trajectories, which
follow Hamiltonian flows on each energy surfaces with hopping between
surfaces. The absolute convergence of the infinite sum used in the
surface hopping ansatz is shown in \S\ref{sec:absconv}. Finally, in \S\ref{sec:error}, we prove the main convergence result Theorem~\ref{thm:main}.

\subsection{Preliminaries}\label{sec:prelim}

To study the absolute convergence of the FGA with surface hopping
ansatz, we fix a time $t$ and recall that
\begin{equation}\label{eq:FGAU}
  U_{\FGA}(t, x)  = \sum_{k=0}^{\infty} \begin{pmatrix}
    u^{(2k)}  \\
    u^{(2k+1)}
  \end{pmatrix}.
\end{equation} 
For convenience of the readers, we also recall
\begin{multline}
  u^{(j)}(t, x)  = \frac{1}{(2\pi\veps)^{3m/2}} \int_{K} \ud z_0 \int_{0<t_1<\cdots<t_j<t} d T_{j:1} \; \tau^{(1)}(T_{1:1}, z_0)\cdots \tau^{(j)} (T_{j:1},z_0) \times \\
  \times A^{(j)}(t, T_{j:1}, z_0) \exp\left( \frac{i}{\veps}
    \Theta^{(j)}(t,T_{j:1}, z_0, x) \right),
\end{multline}
which is an integration over all possible $j$ hopping times $t_1,
\cdots, t_j$. Notice that as discussed above Theorem~\ref{thm:main},
we restrict the domain of integration on the phase space to $K$.

Also recall that the FGA variables in the integrand of each $u^{(j)}$
are evolved piecewisely to final time $t$. To be more specific, the
hopping time sequence $\{t_k\}_{k=1,\cdots, j}$ defines a partition of
the interval $[0,t]$, $0\le t_1\le\cdots\le t_j\le t$, such that
within each interval, the FGA trajectory and associated variables
evolve on a single energy surface, and at hopping times
$\{t_k\}_{k=1,\cdots, j}$ switch to another surface with the
continuity conditions \eqref{eq:contcond}.

We remark that since we study here the case with two energy surfaces,
it suffices to specify the hopping times to uniquely determine the
trajectory. In general, for more energy surfaces, besides the hopping
time, we also need to track  which surface the trajectory hops to
(which makes the notations more complicated).

Let us first collect some properties of the Hamiltonian flow with surface
hopping. Given the hopping times $T_{j:1} = \{t_j, \cdots, t_1\}$, we
denote the map on the phase space from initial time $0$ to time $t$ by  $\kappa_{t, T_{j:1}}$ ($t$ can be smaller than $t_j$ here):
\begin{align*}
\kappa_{t, T_{j:1}}: \quad   \R^{2m} & \rightarrow \R^{2m} \\
(q,p) & \longmapsto   \left(Q^{\kappa_{t, T_{j:1}}}(q,p), P^{\kappa_{t, T_{j:1}}}(q,p) \right), 
\end{align*}
such that 
\begin{equation}
\bigl(Q^{\kappa_{t, T_{j:1}}}(q,p), P^{\kappa_{t, T_{j:1}}}(q,p) \bigr)
= \begin{cases}
  \bigl(Q^{(0)}(t, q, p), P^{(0)}(t, q, p)\bigr), & t \le T_1; \\
  \bigl(Q^{(i)}(t, T_{i:1}, q, p), P^{(i)}(t, T_{i:1}, q, p)\bigr), & t \in [T_i, T_{i+1}], i \in \{1, \ldots, j\}; \\
  \bigl(Q^{(j)}(t, T_{j:1}, q, p), P^{(j)}(t, T_{j:1}, q, p)\bigr), & t \ge T_j,
  \end{cases}
\end{equation}
where the trajectory follows the Hamiltonian flow on one of the energy
surface and hops to the other at the hopping times.  Let us emphasize
that, due to the continuity condition \eqref{eq:contcond}, even with
surface hopping, the trajectory $(Q, P)$ is still continuous on the
phase space as a function of $t$.

The following proposition states that for any possible number of hops
and sequence of hopping time, the trajectory under $\kappa_{t, T}$
remains (uniformly) in a  compact set.
\begin{proposition}\label{prop:comp}
  Given $t>0$ and a compact subset $K \subset \R^{2m}$, there exists a
  compact set $K_t \subset \R^{2m}$, such that $\forall j \in \N$,
  $\forall \delta > 0$ and any sequence of hopping times
  $T_{j:1} \subset [0,t]$
  \begin{equation}
    \kappa_{t, T_{j:1}}(K) \subset K_t, 
  \end{equation}
  namely, for any $(q, p) \in K$ and any $s \in [0, t]$ 
  \begin{equation}\label{comp2}
    \left(Q^{\kappa_{s, T_{j:1}}}(q,p), P^{\kappa_{s, T_{j:1}}}(q,p) \right) \in K_{t}.
  \end{equation}
\end{proposition}

\begin{proof}
Fix an arbitrary sequence of hopping time $T_{j:1} \subset [0, t]$.
For any time $s \in [0, t]$ which belongs to one of the interval
$(t_k, t_{k+1})$ for $ k = 0, \cdots, j$ (we identify $t_0 = 0$ and
$t_{j+1} = t$). Denote the index of the energy surface during the time
interval $(t_k, t_{k+1})$ as $l_k$, we then have 
\begin{align*}
  & \frac{\ud}{\ud s}\Abs{P^{\kappa_{s,T_{j:1}}}} = \frac{P^{\kappa_{s,T_{j:1}}}}{\Abs{P^{\kappa_{s,T_{j:1}}}}} \cdot  \frac{\ud}{\ud s} P^{\kappa_{s,T_{j:1}}} = - \frac{P^{\kappa_{s,T_{j:1}}} \cdot \nabla_Q E_{l_k}}{\Abs{P^{\kappa_{s,T_{j:1}}}}} \le C_E \Bigl(\Abs{Q^{\kappa_{s, T_{j:1}}}} + 1\Bigr), \\
  & \frac{\ud}{\ud s}\Abs{Q^{\kappa_{s,T_{j:1}}}} = \frac{Q^{\kappa_{s,T_{j:1}}}}{\Abs{Q^{\kappa_{s,T_{j:1}}}}}  \cdot \frac{\ud}{\ud s} Q^{\kappa_{s,T_{j:1}}} = \frac{Q^{\kappa_{s,T_{j:1}}} \cdot P^{\kappa_{s,T_{j:1}}}}{\Abs{Q^{\kappa_{s,T_{j:1}}}}} \le \Abs{P^{\kappa_{s,T_{j:1}}}}, 
\end{align*}
where we have used the subquadraticity of the Hamiltonian by
Assumption~\ref{assuma} (recall that $C_E$ is uniform with respect to
$\delta$). Therefore,
\begin{equation*}
  \frac{\ud}{\ud s} \Bigl( \Abs{P^{\kappa_{s,T_{j:1}}}}^2 + \Abs{Q^{\kappa_{s,T_{j:1}}}}^2 \Bigr) \le 2 (C_E + 1) \Abs{P^{\kappa_{s,T_{j:1}}}} \Bigl( \Abs{Q^{\kappa_{s,T_{j:1}}}} + 1 \Bigr)\le 2 (C_E + 1) \Bigl( \Abs{P^{\kappa_{s,T_{j:1}}}}^2 + \Abs{Q^{\kappa_{s,T_{j:1}}}}^2  + 1\Bigr).
\end{equation*}
Here we emphasize that the constant on the right hand side is
universal in the sense that it does not depend on the particular
hopping time sequence. The conclusion of the Proposition follows
immediately from the differential inequality.
\end{proof}

As a corollary, since the trajectory uniformly stays in a compact set, given the final time $t$, we can take a constant $C_{\tau}$ such that the following estimate holds 
\begin{equation} \label{assum1} \sup_{z_0\in K, T_{n:1} \subset [0,t],
    n\in \N^+, \, i = 0, 1} |\tau_i^{(n)}(T_{n:1}, z_0)| \le \sup_{z
    \in K_t} \max\bigl\{ \abs{p \cdot d^\delta_{10}(q)}, \abs{p \cdot
    d^\delta_{01}(q)} \bigr\} \leq C_{\tau},
\end{equation}
where the second inequality uses Assumption~\ref{assumb} and recall
that the constants are uniform with respect to $\delta$.  Thus,
the coupling coefficient stays $\Or(1)$ along all possible FGA
trajectories.

For a transformation of the phase space $\kappa: \R^{2m} \to \R^{2m}$,
we denote its Jacobian matrix as
\begin{equation}\label{jacobi}
  J^{\kappa}(q,p)=   
  \begin{pmatrix}
    \left(\partial_q Q^{\kappa} \right)^T (q,p) & \left(\partial_p Q^{\kappa} \right)^T (q,p)  \\
    \left(\partial_q P^{\kappa} \right)^T (q,p) & \left(\partial_p
      P^{\kappa} \right)^T (q,p)
  \end{pmatrix}.
\end{equation}
We say the transform $\kappa$ is a \emph{canonical} if
$J^{\kappa}$ is symplectic for any $(q,p)\in \R^{2m}$, namely,
\begin{equation}\label{cond:symp}
\left( J^{\kappa} \right)^T 
\begin{pmatrix} 
0 &  I_m \\
-I_{m} & 0
\end{pmatrix} J^{\kappa} = 
\begin{pmatrix} 
0 &  I_m \\
-I_{m} & 0
\end{pmatrix}.
\end{equation}
Here, $I_m$ denotes the $m\times m$ identity matrix.

The map given by the FGA trajectories $\kappa_t$ is always canonical,
as stated in the following proposition, which also gives bounds of the
Jacobian and its derivatives.

\begin{proposition}
  Given $t>0$ and a compact subset $K \subset \R^{2m}$, the associated
  map $\kappa_{t, T_{j:1}}$ is a canonical transformation for any
  sequence of hopping times $T_{j:1}$, $\forall\, j$. Moreover, for
  any $k \in \N$, there exists a constant $C_{k}$ such that
  \begin{equation}\label{est:devF}
    \sup_{(q, p) \in K} \max_{\abs{\alpha_p} + \abs{\alpha_q} \le k} \Abs{ \partial_q^{\alpha_q} \partial_p^{\alpha_p} \bigl[J^{\kappa_{t, T_{j:1}}}(q, p)\bigr]} \le C_{k}, 
  \end{equation}
  uniformly for any $j \in \N$, any $\delta$ and any sequence of
  hopping times $T_{j:1}$.
\end{proposition}
\begin{proof} 
  Recall that the time evolution of $(Q^{\kappa}, P^{\kappa})$ is
  piecewisely defined in the time interval between hoppings, and
  remains continuous at the hopping times. During each time interval,
  the symplectic condition \eqref{cond:symp} is clearly satisfied by
  the Hamiltonian flow. The continuity condition guarantees the
  validity of symplectic relation at the hopping times. Therefore, the
  map $\kappa_{t,T_{j:1}}$ is a canonical transform.
 
  For any time $s \in [0, t]$ which belongs to one of the interval
  $(t_k, t_{k+1})$ for $ k = 0, \cdots, j$ (we identify $t_0 = 0$ and
  $t_{j+1} = t$). Denote the index of the energy surface during the
  time interval $(t_k, t_{k+1})$ as $l_k$, we then have by
  differentiating $J^{\kappa_{s,T_{j:1}}}$ with
  respect to $s$ 
  \begin{equation}\label{eq:F}
    \frac{\ud}{\ud s} J^{\kappa_{s,T_{j:1}}}= 
    \begin{pmatrix} 
      \partial_P\partial_Q H_{l_k} &  \partial_P\partial_P H_{l_k} \\
      -\partial_Q\partial_Q H_{l_k} & -\partial_Q\partial_P H_{l_k}
    \end{pmatrix} J^{\kappa_{s,T_{j:1}}}.
  \end{equation}
  Then, Assumption~\ref{assuma} implies there exists a constant $C$ such that
  \begin{equation}
    \frac{\ud}{\ud s} \Abs{J^{\kappa_{s,T_{j:1}}}}
    \le \left\lvert \begin{pmatrix}
        \partial_P\partial_Q H_{l_k} & \partial_P\partial_P H_{l_k} \\
        -\partial_Q\partial_Q H_{l_k} &  -\partial_Q\partial_P H_{l_k}
      \end{pmatrix} \right\rvert \Abs{J^{\kappa_{s,T_{j:1}}}} \le C
    \Abs{J^{\kappa_{s,T_{j:1}}}}.
  \end{equation}
  It is worth emphasizing that this constant $C$ is independent of the
  hopping time sequence and $\delta$. The boundedness of
  $|J^{\kappa_{t,T_{j:1}}}|$ then follows immediately from Gronwall's
  inequality and the fact that $|J^{\kappa_{0, T_{j:1}}}|=1$ since
  $\kappa_{0, T_{j:1}}$ is just an identity map. To get the estimate
  for derivatives of $J$, we differentiate the equation \eqref{eq:F}
  with respect to $(q, p)$ and use an induction argument, and we omit
  the straightforward calculations here.
\end{proof}

For a canonical transform $\kappa$, we define
\[
Z^{\kappa} (q,p) = \partial_z \left( Q^{\kappa} (q,p)+ i P^{\kappa}(q,p) \right),
\]
where $\partial_z = \partial_q - i \partial_p$. $Z^{\kappa}$ is a
complex valued $m \times m$ matrix.  By mimicking the proof of
\cite{FGA_Conv}*{Lemma 5.1} and the above Proposition, we obtain the following properties of $Z^{\kappa}$.
\begin{proposition}
  Given $t>0$ and a compact subset $K \subset \R^{2m}$, for any
  sequence of hopping times $T_{j:1}$, $\forall\, j$, $Z^{\kappa_{t,
      T_{j:1}}}$ is invertible. Moreover, for any $k \in \N$, there
  exists a constant $C_{k}$ such that
  \begin{equation}\label{est:devZ}
    \sup_{(q, p) \in K} \max_{\abs{\alpha_p} + \abs{\alpha_q} \le k} \Abs{ \partial_q^{\alpha_q} \partial_p^{\alpha_p} \bigl[ \bigl(Z^{\kappa_{t, T_{j:1}}}(q, p)\bigr)^{-1}\bigr]} \le C_{k}, 
  \end{equation}
  uniformly for any $j \in \N$, any $\delta > 0$ and any sequence of
  hopping times $T_{j:1}$.
\end{proposition}

For the frozen Gaussian approximation, it is useful to introduce the
following Fourier integral operator. For
$M \in L^\infty (\R^{2m}; \C)$, $u\in \mathcal S (\R^m;\C)$, and a FGA
flow denoted by $\kappa_{t,T_{j:1}}$ with $T_{j:1} \subset [0, t]$, we
define
\begin{equation}
  \Bigl(\mathcal I_{\kappa_{t,T_{j:1}}}^\veps (M) u\Bigr) (x) = (2 \pi \veps)^{- \frac{3m}{2} } \int_{\R^m}\int_{\R^{2m}} \exp \Bigl( \frac i \veps \Phi^{(j)} (t, x, y, z) \Bigr) M(z) u(y) \ud z \ud y,
\end{equation}
where the phase function $\Phi^{(j)}$ is given by
\[
\Phi^{(j)} (t,x,y,z)= S^{(j)}(t,T_{j:1},z) + \frac{i}{2}
\bigl\lvert x-Q^{(j)}(t,T_{j:1},z)\bigr\rvert^2 + P^{(j)}(t,T_{j:1},z)
\cdot \bigl(x-Q^{(j)}(t,T_{j:1},z)\bigr) + \frac{i}{2}|y-q|^2 - p\cdot (y-q), 
\]
where the FGA variables $P^{(j)}, Q^{(j)}, S^{(j)}$ are evolved as in the surface hopping ansatz, with given $t$ and the hopping time sequence $T_{j:1}$. 
With this Fourier integral operator, we may rewrite the surface hopping ansatz \eqref{eq:u0n} for $u_0^{(n)}$ as
\begin{equation*}
  u_0^{(n)}(t)= \int_{0<t_1<\cdots<t_n<T} \ud T_{n:1}  \; \mathcal{I}_{\kappa_{t,T_{n:1}}}^\veps \Bigl( a_i^{(n)} \prod_{j=1}^n \tau_i^{(j)} \chi_K\Bigr) u_{0}(0),
\end{equation*}
where $\chi_K$ is the characteristic function on the set $K$, which
restricts the initial $z_0$ in the FGA ansatz.  This representation is
particularly convenient for our estimates, as we have the following
proposition for the norm of the Fourier integral operators. The
version of Proposition without hopping was proved in
\cite{FGA_Conv}*{Proposition 3.7}. The proof in fact can be
almost verbatim used in the current situation (with some notational
change), and thus we skip the details here.

\begin{proposition}\label{operator}
  For any $t$ and any hopping time sequence
  $\{t_1, t_2, \cdots, t_j\}$ for $\forall\, j\in \N$, denoting the
  the symplectic transform for the FGA with surface hopping flow as
  $\kappa_{t,T_{j:1}}$, the operator
  $\mathcal I_{\kappa_{t,T_{j:1}}}^\veps(M)$ can be extended to a
  linear bounded operator on $L^2(\R^m,\C)$, and we have
  \begin{equation}
    \biggl\lVert \mathcal I_{\kappa_{t,T_{j:1}}}^\veps (M) \biggr\rVert_{\mathcal L (L^2(\R^m;\,\C))} \le 2^{-\frac{m}{2}} \norm{M}_{L^\infty (\R^{2m};\, \C)}.
  \end{equation}
\end{proposition}

\subsection{Absolute convergence of surface hopping ansatz}
\label{sec:absconv}

Now we estimate the contribution of the terms in the FGA ansatz.
\begin{proposition}\label{FI} 
  For a given time $t$, there exists a constant $C_a$, depending only on
  $t$ and the initial conditions of the FGA variables, such that for
  any $n\in \N$ and any hopping moment sequence $T_{n:1} \subset
  [0,t]$, it holds
  \begin{equation}\label{est1}
    \Biggl\lVert \frac{1}{(2\pi\veps)^{3m/2}} \int_K \ud z_0 \; 
    \prod_{j=1}^n \tau^{(j)}(T_{j:1}, z_0) 
    A^{(n)}(t, T_{n:1}, z_0) \exp\Bigl( \frac{i}{\veps} 
    \Theta^{(n)}\bigl(t,T_{n:1}, z_0,x\bigr) \Bigr) 
    \Biggr\rVert_{L^2(\R^m)} \le C_a.
  \end{equation}
\end{proposition}
\begin{proof}

  Recall that we have 
  \begin{multline*}
    \frac{1}{(2\pi\veps)^{3m/2}} \int_K \ud z_0 \;
    \prod_{j=1}^n \tau^{(j)}(T_{j:1}, z_0) A^{(n)}(t, T_{n:1}, z_0) \exp\Bigl( \frac{i}{\veps} \Theta^{(n)}(T,T_{n:1}, z_0,x) \Bigr) \\
    = \Bigl(\mathcal I_{\kappa_{t,T_{n:1}}}^\veps \Bigl(\prod_{j=1}^n
    \tau^{(j)}(T_{j:1}, \cdot) a^{(n)}(t, T_{n:1}, \cdot) \chi_{K}
    \Bigr) u_{0}(0)\Bigr)(x).
  \end{multline*}
  Thus, using Proposition~\ref{operator} and the bound \eqref{assum1}
  for the hopping coefficient $\tau$'s, it suffices to control
  $a^{(n)}(t, T_{n:1}, z_0)$ for $z_0 \in K$. Recall that
  \[
  A^{(k)}(t, T_{k:1}, z_0) =a^{(k)}(t, T_{k:1}, z_0) \int_{\R^m}
  u_{0}(0, y) e^{\frac{i}{\veps} (-p\cdot(y-q)+ \frac{i}{2}|y-q|^2)}
  \ud y,
  \]
  and hence $a^{(k)}$ satisfies the same linear equation as those for
  $C^{(k)}$ with the continuity conditions at the hopping times:
  Depending on whether $k$ is even or odd
  \begin{align*}
    \frac{\ud}{\ud t} a^{(k)} & = \frac 1 2 a^{(k)} \tr\left(
      (Z^{(k)})^{-1}\left(\partial_z P^{(k)} - i \partial_z Q^{(k)}
        \nabla^2_Q E_0(Q^{(k)}) \right) \right) - a^{(k)} d_{00}\cdot
    P^{(k)}, \qquad k \text{ even}; \\
    \frac{\ud}{\ud t} a^{(k)} & = \frac 1 2 a^{(k)} \tr\left(
      (Z^{(k)})^{-1}\left(\partial_z P^{(k)} - i \partial_z Q^{(k)}
        \nabla^2_Q E_1(Q^{(k)}) \right) \right) - a^{(k)} d_{11}\cdot
    P^{(k)}, \qquad k \text{ odd}.
  \end{align*}
  Note that the coefficients on the right hand side are all uniformly
  bounded along the trajectory thanks to \eqref{comp2},
  \eqref{assum1}, \eqref{est:devF}, and \eqref{est:devZ} in the
  preliminaries. Therefore, $a^{(n)}$ is also bounded uniformly with
  respect to all hopping sequences, which concludes the proof. 
\end{proof}
   
With Proposition~\ref{FI}, we further estimate the contribution $u^{(n)}$ in the surface hopping ansatz. 
\begin{theorem}\label{thm:ansatz}
  Under Assumptions~\ref{assuma} and \ref{assumb}, given a fixed final
  time $t$, there exist constants $C_t$ and $C$, independent of
  $\veps$ and $\delta$, such that for any $n \in \N$, we have
  \[
  \Norm{u^{(n)}(t, x)}_{L^2(\R^m)} \le C \frac{ (C_t)^{n}}{n!},
  \quad\text{and}\quad \Norm{\veps \nabla_x u^{(n)}(t, x)}_{L^2(\R^m)}
  \le C \frac{ (C_t)^{n}}{n!}.
  \] 
In particular, the surface hopping ansatz \eqref{ansatz2} is
absolutely convergent.
\end{theorem}

\begin{proof}
Note that 
\begin{equation*}
  u^{(n)} = \int_{0<t_1<\cdots<t_n<t} \ud T_{n:1} \; \mathcal{I}_{\kappa_{t, T_{n:1}}}^{\veps} \Bigl( \prod_{j=1}^n \tau^{(j)}( T_{j:1}, \cdot) a^{(n)}(t, T_{n:1}, \cdot) \chi_K \Bigr) u_0. 
\end{equation*}
We estimate using Proposition~\ref{FI} 
\begin{align*}
  \Norm{u^{(n)}}_{L^2(\R^m)} & \le \int_{0<t_1<\cdots<t_n<t} \ud T_{n:1} \; \Biggl\lVert \mathcal{I}_{\kappa_{t, T_{n:1}}}^{\veps} \Bigl( \prod_{j=1}^n \tau^{(j)}( T_{j:1}, \cdot) a^{(n)}(t, T_{n:1}, \cdot) \chi_K \Bigr) u_0 \Biggr\rVert_{L^2(\R^m)} \\
                             & \le C \int_{0<t_1<\cdots<t_n<t} \ud T_{n:1} \; C_{\tau}^n = C \frac{(tC_{\tau})^n}{n!}
\end{align*}
The absolute convergence of $U_{\FGA}(t, x)$ then follows from
dominated convergence.

The control of $\veps \nabla_x u^{(n)}$ is quite similar, except that we
shall use Lemma~\ref{lem:asym} to control the term $(x - Q^{(n)})$
resulting from the gradient.  
Actually, 
\begin{multline}
  \veps \nabla_x u^{(n)}(t, x)  = \frac{1}{(2\pi\veps)^{3m/2}} \int_{K} \ud z_0 \int_{0<t_1<\cdots<t_n<t} \ud T_{n:1} \; \tau^{(1)}(T_{1:1}, z_0)\cdots \tau^{(n)} (T_{n:1},z_0) \times \\
  \times i \left( P^{(n)}+ i (x - Q^{(n)}) \right) A^{(n)}(t, T_{n:1}, z_0) \exp\left( \frac{i}{\veps}
    \Theta^{(n)}(t,T_{n:1}, z_0, x) \right).
\end{multline}
The control of the term involving $P^{(n)}$ is the same as that of
$u^{(n)}$. By Lemma~\ref{lem:asym}, the term involving $(x - Q^{(n)})$
is an even smaller term, which follows from the estimate
\eqref{est:devZ} and a slight variation of Proposition~\ref{FI}.
\end{proof}

\subsection{The analysis of approximation error}\label{sec:error}

We now estimate the approximation error of the FGA approximation with
surface hopping to the matrix Schr\"odinger equation \eqref{vSE}.  We
first state a consistency result by estimating the error of
substituting $U_{\FGA}$ into \eqref{vSE}. All the estimates and
constants below are uniform in $\delta$. 
\begin{theorem}\label{thm:consistency}
Under Assumptions~\ref{assuma} and \ref{assumb}, given a final time $t$, there exists a constant $C_t$, such that 
\begin{multline*}
  \Biggl\lVert
  i \veps \partial_t U_{\FGA} + \frac{\veps^2}{2} \Delta_x
  U_{\FGA} +
  \begin{pmatrix} 
    E_0 \\
    & E_1
  \end{pmatrix} U_{\FGA} -\red{\frac{\veps^2}{2}}  \begin{pmatrix}
    D_{00} & D_{01} \\
    D_{10} & D_{11}
  \end{pmatrix} U_{\FGA} - \red{\veps^2} \sum_{j=1}^m 
  \begin{pmatrix} 
    d_{00} & d_{01} \\
    d_{10} & d_{11}
  \end{pmatrix}_j
  \partial_{x_j} U_{\FGA} \Biggl\rVert_{L^2(\R^m)} 
  \le \veps^2 e^{C_t}.
\end{multline*}

\end{theorem}

\begin{proof}
We first consider the term arises from the time derivative and denote for $j$ even
\begin{equation*}
  I^j_1 +I^j_2 =
  i \veps \partial_t 
  \begin{pmatrix} 
    u_0^{(j)} \\
    0
  \end{pmatrix},
\end{equation*}
where 
\begin{equation}
  I^j_1 =  
  \begin{pmatrix}\dps
    i \veps \frac{1}{(2\pi \veps)^{3m/2}} \int_{K} \ud z_0 \int_{0<t_1<\cdots<t_j<t} \ud T_{j:1} \; \tau^{(1)}\cdots \tau^{(j)} \partial_t \Bigl[ A^{(j)}\exp \Bigl(\frac i \veps \Theta^{(j)} \Bigr) \Bigr] \\
    0
  \end{pmatrix}
\end{equation}
coming from the time derivative acting on the integrand, and for $j \ge 1$, 
\begin{equation}
  I^j_2 = 
  \begin{pmatrix}\dps 
    i \veps \frac{1}{(2\pi \veps)^{3m/2}} \left[ \int_{K} \ud z_0 \int_{0<t_1<\cdots<t_j<t} \ud T_{j:1} \; \tau^{(1)}\cdots \tau^{(j)}   A^{(j)}\exp \Bigl(\frac i \veps \Theta_0^{(j)} \Bigr) \right]_{t_j=t} \\
    0
  \end{pmatrix}
\end{equation}
resulting from the time derivative acting on the upper limit of the
integral. The expression for odd $j$ is similar except that the top
and bottom rows are flipped: $i\veps \partial_t
\bigl( \begin{smallmatrix} 0 \\ u_0^{(j)} \end{smallmatrix} \bigr)$.

We also write the terms from the right hand side of \eqref{vSE} for
even $j$ as:
\begin{equation*}
  I^j_3 + I^j_4 = 
  -\frac{\veps^2}{2} \Delta_x \begin{pmatrix} u^{(j)} \\
    0 \end{pmatrix} + 
  \begin{pmatrix} 
    E_0 \\
    & E_1
  \end{pmatrix} \begin{pmatrix} u^{(j)} \\
    0 \end{pmatrix} -\red{\frac{\veps^2}{2}} \begin{pmatrix}
    D_{00} & D_{01} \\
    D_{10} & D_{11}
  \end{pmatrix} \begin{pmatrix} u^{(j)} \\
    0 \end{pmatrix} -\red{\veps^2}  \sum_{j=1}^d
  \begin{pmatrix} 
    d_{00} & d_{01} \\
    d_{10} & d_{11}
  \end{pmatrix}_j
  \partial_{x_j} \begin{pmatrix} u^{(j)} \\
    0 \end{pmatrix}
\end{equation*}
where 
\begin{align}
  I^j_3 & = 
  \begin{pmatrix} 
    \bigl(H_0 - \frac{\veps^2}{2} D_{00} \bigr) u^{(j)}- \veps^2
    d_{00} \cdot \nabla_x u^{(j)} \\
    0
  \end{pmatrix}, \\
  I^j_4 & = 
  \begin{pmatrix}
    0 \\
    - \veps^2 d_{01} \cdot \nabla_x u^{(j)} - \frac{\veps^2}{2}
    D_{01} u^{(j)}
  \end{pmatrix}.
\end{align}
Here, $I^j_3$ contains all the terms which govern the inner-surface evolution on each energy surface, while $I^j_4$ contains the coupling terms (note that the subscripts are swapped). The expressions for odd $j$ is similar.

Denote $U_{\FGA}^n$ the sum of the first $n$ terms in the FGA ansatz
\eqref{eq:FGAU}.  Summing over the contributions up to $u^{(n)}$, we
have
\begin{multline*}
  i \veps \partial_t U_{\FGA}^n + \frac{\veps^2}{2} \Delta_x
  U_{\FGA}^n +
  \begin{pmatrix} 
    E_0 \\
    & E_1
  \end{pmatrix} U_{\FGA}^n  -\red {\frac{\veps^2}{2}} \begin{pmatrix}
    D_{00} & D_{01} \\
    D_{10} & D_{11}
  \end{pmatrix} U_{\FGA}^n -\red {\veps^2} \sum_{j=1}^d 
  \begin{pmatrix} 
    d_{00} & d_{01} \\
    d_{10} & d_{11}
  \end{pmatrix}_j
  \partial_{x_j} U_{\FGA}^n \\
  =\sum_{j=0}^n I^j_1 + \sum_{j=1}^n I^j_2 - \sum_{j=0}^n I^j_3 -
  \sum_{j=0}^n I^j_4 = \sum_{j=0}^n \left(I^j_1-I^j_3\right)+
  \sum_{j=0}^{n-1} \left(I^{j+1}_2-I^j_4\right) + I^n_4.
\end{multline*}
We now estimate three terms on the right hand side. 

\smallskip 

\noindent\emph{Term} $I_4^n$: By Theorem~\ref{thm:ansatz}, we have 
\begin{equation}\label{eq:I4n}
  \begin{aligned}
    \bigl\lVert I_4^n \bigr\rVert_{L^2(\R^m)} & \le C \veps \norm{\veps \nabla_x u^{(j)}}_{L^2(\R^m)} + C \veps^2 \norm{u^{(j)}}_{L^2(\R^m)}
     \le C \veps \frac{(C_t)^n}{n!}.
  \end{aligned}
\end{equation}

\smallskip 
\noindent 
\emph{Term $\sum (I^j_1-I^j_3)$}: The difference $(I^j_1 - I^j_3)$
contains all the formally $O(\veps^2)$ terms we have dropped in
determining the equation for $A^{(j)}$. To estimate those, we use
the Taylor expansions with respect the beam center $Q$
\begin{align*}
  E_k(x) & = \sum_{|\alpha|\le 3} \frac{\partial_\alpha E_k(Q)}{\alpha !} (x-Q)^\alpha+ R_{4,Q}[E_k]; \\
  d_{kl}(x) & = \sum_{|\alpha|\le 1} \frac{\partial_\alpha
    d_{kl}(Q)}{\alpha !} (x-Q)^\alpha+ R_{2,Q}[d_{kl}],
\end{align*}
where $R_{k,Q}[f]$ denotes the $k$-th order remainder in the Taylor
expansion of the function $f$ at $Q$. 
\begin{equation*}
  \bigl\lVert I^j_1-I^j_3 \bigr\rVert_{L^2(\R^m)} \le I^j_{11} + I^j_{12} + I^j_{13},
\end{equation*}
where ($k = 0$ if $j$ even and $k = 1$ if $j$ odd)
\begin{align*}
  I^j_{11} & = \sum_{|\alpha|=3} \left\| \frac{1}{(2\pi\veps)^{3m/2}}
    \int_K \ud z_0 \int_{0<t_1<\cdots<t_j<t} \ud T_{j:1}\;
    \tau^{(1)}\cdots \tau^{(j)}
    A^{(j)} e^{\frac{i}{\veps} \Theta^{(j)}} \frac{\partial_\alpha E_k(Q^{(j)})}{\alpha !} (x-Q^{(j)})^\alpha \right\|_{L^2}   \\
  & \quad + \veps \sum_{|\alpha|=1} \left\|
    \frac{1}{(2\pi\veps)^{3m/2}} \int_K \ud z_0 \int_{0<t_1<\cdots<t_j<t} \ud T_{j:1}\; \tau^{(1)}\cdots \tau^{(j)}
    A^{(j)} e^{ \frac{i}{\veps} \Theta^{(j)}} P^{(j)} \cdot  \frac{\partial_\alpha d_{kk}(Q^{(j)})}{\alpha !} (x-Q^{(j)})^\alpha \right\|_{L^2} \\
  & \quad +\veps \left\| \frac{1}{(2\pi\veps)^{3m/2}} \int_K \ud z_0
   \int_{0<t_1<\cdots<t_j<t} \ud T_{j:1} \; \tau^{(1)}\cdots
    \tau^{(j)} A^{(j)} e^{ \frac{i}{\veps} \Theta^{(j)}}
    d_{kk}(Q^{(j)}) \cdot (x-Q^{(j)}) \right\|_{L^2}, \\
  I^j_{12} &= \left\| \frac{1}{(2\pi\veps)^{3m/2}} \int_K \ud z_0
    \int_{0<t_1<\cdots<t_j<t} \ud T_{j:1} \; \tau^{(1)}\cdots
    \tau^{(j)}
    A^{(j)} e^{ \frac{i}{\veps} \Theta^{(j)}} R_{4,Q^{(j)}}[E_k] \right\|_{L^2} \\
  & \quad + \veps \left\|\frac{1}{(2\pi\veps)^{3m/2}} \int_K \ud z_0
    \int_{0<t_1<\cdots<t_j<t} \ud T_{j:1} \; \tau^{(1)}\cdots
    \tau^{(j)} A^{(j)}e^{ \frac{i}{\veps} \Theta^{(j)}}
    P^{(j)} \cdot R_{2,Q^{(j)}}[d_{kk}]  \right\|_{L^2} \\
  & \quad + \veps \left\|\frac{1}{(2\pi\veps)^{3m/2}}
    \int_K \ud z_0\int_{0<t_1<\cdots<t_j<t} \ud T_{j:1} \;
    \tau^{(1)}\cdots \tau^{(j)} A^{(j)}e^{ \frac{i}{\veps}
      \Theta^{(j)}}R_{1,Q^{(j)}}[d_{kk}] \cdot
    (x-Q^{(j)})\right\|_{L^2}, \\
  I^j_{13} & = \frac{\veps^2}{2} \left\|D_{kk} u^{(j)}\right\|_{L^2}.
\end{align*}
Here, $I^j_{11}$ contains the next order Taylor expansion terms after asymptotic matching, $I^j_{12}$ contains the remainder terms in the Taylor expansions, and $I^j_{13}$ contains the contribution from $D_{kk}$.

To estimate $I_{11}^j$, note that by Assumption~\ref{assuma},
Proposition~\ref{operator}, and Lemma~\ref{lem:asym}, we have for
$\abs{\alpha} = 3$
\[
\left\|  \frac{1}{(2\pi\veps)^{3m/2}} \int_K \ud z_0\; A^{(n)} e^{\frac{i}{\veps} \Theta^{(n)}} \frac{\partial_\alpha E_k(Q)}{\alpha !} (x-Q)^\alpha \right\|_{L^2} \le C \veps^2.
\]
Thus by a similar calculation as in the proof of
Theorem~\ref{thm:ansatz}, we obtain
\[
 \left\| \frac{1}{(2\pi\veps)^{3m/2}} \int_K \ud z_0 \int_{0<t_1<\cdots<t_j<t} \ud T_{j:1} \; \tau^{(1)}\cdots \tau^{(j)} 
A^{(j)} e^{\frac{i}{\veps} \Theta^{(j)}} \frac{\partial_\alpha E_k(Q)}{\alpha !} (x-Q)^\alpha\Psi_k \right\|_{L^2}\le C \veps^2  \frac{(C_t)^j}{j !} .
\]
We can similarly estimate the other two terms in $I_{11}^j$, which can by controlled by the same bound, which yields 
\begin{equation}\label{eq:I11}
  I_{11}^j \le C \veps^2 \frac{(C_t)^j}{j!}. 
\end{equation}

The estimate of the term $I_{12}^j$ is similar as by Lemma~\ref{lem:asym}, the powers of $(x - Q)^{\alpha}$ is of higher order in $\veps$. In particular, we have 
\begin{equation*}
  \left\|  \frac{1}{(2\pi\veps)^{3m/2}} \int_{K} \ud z_0 A^{(j)} e^{\frac{i}{\veps} \Theta^{(j)}} R_{4,Q}[E_k] \right\|_{L^2}  
  \leq C \sum_{|\alpha|=4} \left\| \frac{1}{(2\pi\veps)^{3m/2}} \int_{K}
    \ud z_0 A^{(j)} e^{\frac{i}{\veps} \Theta^{(j)}}
    |x-Q^{(j)}|^{\alpha}\right\|_{L^2} = O(\veps^2), 
\end{equation*}
and hence,
\begin{equation*}
 \left\| \frac{1}{(2\pi\veps)^{3m/2}} \int_{K} \ud z_0 \int_{0<t_1<\cdots<t_j<t} \ud T_{j:1} \; \tau^{(1)}\cdots \tau^{(j)} 
A^{(j)} e^{ \frac{i}{\veps} \Theta^{(j)}}R_{4,Q}[E_k] \right\|_{L^2} \le C \veps^2 \frac{(C_t)^j}{j!}.
\end{equation*}
The other two terms in $I_{12}^j$ can be similarly bounded, and we arrive at 
\begin{equation}\label{eq:I12}
  I_{12}^j \le C \veps^2 \frac{(C_t)^j}{j!}.
\end{equation}
 
The $I_{13}^j$ term can be  estimated using Assumption~\ref{assumb} and Theorem~\ref{thm:ansatz}, which yields
\begin{equation}\label{eq:I13}
I_{13}^j \le C \veps^2 \norm{u^{(j)}}_{L^2(\R^m)} \le C \veps^2 \frac{(C_t)^j}{j!}.
\end{equation}

Now adding up \eqref{eq:I11}, \eqref{eq:I12}, and \eqref{eq:I13} from
$j = 0$ to $n$, we get
\begin{equation}\label{eq:I1I3}
  \sum_{j=0}^n \norm{I_1^j - I_3^j}_{L^2(\R^m)} \le C \veps^2 \sum_{j=0}^n \frac{(C_t)^j}{j!} \le C \veps^2 e^{C_t}. 
\end{equation}

\smallskip 
\noindent
\emph{Term $\sum (I^{j+1}_2-I^j_4)$:} The difference $(I^{j+1}_2 -
I^j_4)$ contains all the formally $O(\veps^2)$ terms we have dropped
in specifying the hopping coefficients $\tau^{(j)}$. By Taylor
expansion,
\begin{equation*}
  \bigl\lVert I^{j+1}_2-I^j_4 \bigr\rVert_{L^2(\R^m)} \le I^j_{21} + I^j_{22} + I^j_{23},
\end{equation*}
where (for $j$ even, the formula for odd $j$ is similar except that
$d_{01}, D_{01}$ change to $d_{10}, D_{10}$ respectively)
\begin{align*}
  I^j_{21} & =  \veps \sum_{|\alpha|=1} \left\|
    \frac{1}{(2\pi\veps)^{3m/2}} \int_K \ud z_0 \int_{0<t_1<\cdots<t_j<t} \ud T_{j:1}\; \tau^{(1)}\cdots \tau^{(j)}
    A^{(j)} e^{ \frac{i}{\veps} \Theta^{(j)}} P^{(j)} \cdot  \frac{\partial_\alpha d_{01}(Q^{(j)})}{\alpha !} (x-Q^{(j)})^\alpha \right\|_{L^2} \\
  & \quad +\veps \left\| \frac{1}{(2\pi\veps)^{3m/2}} \int_K \ud z_0
    \int_{0<t_1<\cdots<t_j<t} \ud T_{j:1} \; \tau^{(1)}\cdots
    \tau^{(j)} A^{(j)} e^{ \frac{i}{\veps} \Theta^{(j)}}
    d_{01}(Q^{(j)}) \cdot (x-Q^{(j)}) \right\|_{L^2}, \\
  I^j_{22} &= \veps \left\|\frac{1}{(2\pi\veps)^{3m/2}} \int_K \ud z_0
   \int_{0<t_1<\cdots<t_j<t} \ud T_{j:1} \; \tau^{(1)}\cdots
    \tau^{(j)} A^{(j)}e^{ \frac{i}{\veps} \Theta^{(j)}}
    P^{(j)} \cdot R_{2,Q^{(j)}}[d_{01}]  \right\|_{L^2} \\
  & \quad + \veps \left\|\frac{1}{(2\pi\veps)^{3m/2}}
    \int_K \ud z_0 \int_{0<t_1<\cdots<t_j<t} \ud T_{j:1} \;
    \tau^{(1)}\cdots \tau^{(j)} A^{(j)}e^{ \frac{i}{\veps}
      \Theta^{(j)}}R_{1,Q^{(j)}}[d_{01}] \cdot
    (x-Q^{(j)})\right\|_{L^2}, \\
  I^j_{23} & = \frac{\veps^2}{2} \left\|D_{01} u^{(j)}\right\|_{L^2}.
\end{align*}
The estimates of these terms are similar to that we have done for the
terms arising from $\left(I^j_1-I^j_3\right)$, and hence we omit the
details. We get
\begin{equation}\label{eq:I2I4}
  \sum_{j=0}^n \norm{I_2^{j+1} - I_4^{j}}_{L^2(\R^m)} \lesssim \veps^2 \sum_{j=0}^n \frac{(C_t)^j}{j!} \le C \veps^2 e^{C_t}. 
\end{equation}

\smallskip

Therefore, putting together \eqref{eq:I4n}, \eqref{eq:I1I3}, \eqref{eq:I2I4}, we get 
\begin{multline*}
  \Biggl\lVert
  i \veps \partial_t U_{\FGA}^n + \frac{\veps^2}{2} \Delta_x
  U_{\FGA}^n +
  \begin{pmatrix} 
    E_0 \\
    & E_1
  \end{pmatrix} U_{\FGA}^n  -\red {\frac{\veps^2}{2}}  \begin{pmatrix}
    D_{00} & D_{01} \\
    D_{10} & D_{11}
  \end{pmatrix} U_{\FGA}^n -\red{\veps^2}  \sum_{j=1}^d 
  \begin{pmatrix} 
    d_{00} & d_{01} \\
    d_{10} & d_{11}
  \end{pmatrix}_j
  \partial_{x_j} U_{\FGA}^n \Biggl\rVert_{L^2(\R^m)} \\
  \le C \veps \frac{(C_t)^n}{n!} + C \veps^2 e^{C_t}.
\end{multline*}
Taking the limit $n \to \infty$ and by increasing $C_t$ to absorb the
constant $C$ above, we arrive at the conclusion.
\end{proof}

To control the propagation of the consistency error of the FGA solution in time, we need the next lemma.
 
\begin{lemma}\label{lem:conv}
Suppose $H^\veps$ is a family of self-adjoint operators for $\veps > 0$. Suppose a time dependent wave function $\phi ^\veps (t)$,  which belongs to the domain of $H^\veps$, is continuously differentiable in $t$. In addition, $\phi^\veps(t)$ satisfies the following equation,
\begin{equation} \label{eq:approx_remainder}
\left( i \veps \frac{\partial}{\partial t} - H^\veps \right) \phi^\veps (t) = \zeta^\veps(t),
\end{equation}
where the remainder $\zeta^{\veps}$ satisfying the following estimate
\[
\|\zeta^\veps (t)\|_{L^2} \le \mu^\veps(t).
\]
Then, let $\widetilde \phi^\veps $ be the solution to  the Schr\"odinger equation with Hamiltonian $H^\veps$, and 
\[
\|\phi^\veps(0)-\widetilde \phi^\veps (0)\|_{L^2} \le e_0.
\]
We have then
\begin{equation}\label{eq:stab_est}
\|\phi^\veps(t)-\widetilde \phi^\veps (t)\|_{L^2} \le e_0 + \frac{\int_0^t \mu^\veps (s) \ud s}{ \veps}.
\end{equation}
\end{lemma}
\begin{proof}
Since $H^\veps$ is self-adjoint, it generates a unitary propagator $\mathcal U^\veps(t,s)=\exp\left(\int^t_s -i H^\veps \ud s'/\veps\right)$, such that
\[
\mathcal U^\veps (t,s) \widetilde \phi^\veps (s) = \widetilde \phi^\veps(t).
\] 
Therefore, we obtain,
\begin{align*}
\|\phi^\veps(t)-\widetilde \phi^\veps (t)\|_{L^2} & = \|\phi^\veps(t)- \mathcal U^\veps (t,0)  \widetilde \phi^\veps (0)\|_{L^2}  \\
&=  \|  \mathcal U^\veps (0,t) \phi^\veps(t)-  \widetilde \phi^\veps (0)\|_{L^2}.
\end{align*}
Here, we have used $ (\mathcal U^\veps)^{-1} (t,0)=\mathcal U^\veps (0,t)$. Then, by triangle inequality, we have
\begin{align*}
\|\phi^\veps(t)-\widetilde \phi^\veps (t)\|_{L^2} & \le \|  \mathcal U^\veps (0,t) \phi^\veps(t)-   \phi^\veps (0)\|_{L^2} + e_0 \\[5 pt]
& = \left\|  \int_0^t  \frac{\partial}{\partial s} \left( \mathcal U^\veps (0,s) \phi^\veps(s) \right) \ud s \right\|_{L^2} + e_0 \\
& = \left\| \int_0^t \left(   \frac{\partial}{\partial s} \mathcal U^\veps (0,s)\phi^\veps(s) +\mathcal U^\veps (0,s)  \frac{\partial}{\partial s} \phi^\veps(s) \right) \ud s  \right\|_{L^2} + e_0.
\end{align*}
Then, by using properties of the unitary propagator and equation \eqref{eq:approx_remainder}, we get
\begin{align*}
\|\phi^\veps(t)-\widetilde \phi^\veps (t)\|_{L^2}  & \le \left\|  \frac{i}{\veps}\int_0^t \left( - \mathcal U^\veps (0,s) H^\veps \phi^\veps(s) +\mathcal U^\veps (0,s)  ( H^\veps \phi^\veps(s) + \zeta^\veps(s)) \right) \ud s  \right\|_{L^2} +e_0 \\
& = \frac{1}{\veps} \left\|  \int_0^t \left( \mathcal U^\veps (0,s)  \zeta^\veps(s)\right) \ud s  \right\|_{L^2} +e_0.
\end{align*}
We arrive at \eqref{eq:stab_est} by noticing that
\[
 \left\|  \int_0^t \left( \mathcal U^\veps (0,s)  \zeta^\veps(s)\right) \ud s  \right\|_{L^2}  \le \int_0^t  \left\| \mathcal U^\veps (0,s)  \zeta^\veps(s)\right\|_{L^2} \ud s \le \int_0^t \mu^\veps(s) \ud s.
\]
\end{proof}

In the lemma above, $\phi^\veps$ almost solves the Schr\"odinger
equation with Hamiltonian $H^\veps$ in the sense of equation
\eqref{eq:approx_remainder}, where the remainder term $\zeta^\veps$ is
controlled. Then, $\phi^\veps$ can be considered as an approximate
solution to $\widetilde \phi^\veps$, if the right hand side of the
estimate \eqref{eq:stab_est} is small.  Therefore, if we take
$\phi^\veps$ to the approximation given by FGA with surface hopping,
then with the stability lemma, we can conclude the error estimate in
Theorem~\ref{thm:main}.

\begin{proof}[Proof of Theorem~\ref{thm:main}]
Note that the initial error is given by $\eps_{\text{in}}$ by definition. The theorem is then a corollary of Theorem~\ref{thm:consistency} and Lemma~\ref{lem:conv}.
\end{proof}

\bibliographystyle{amsxport}
\bibliography{surfacehopping} 
 
\end{document}